\documentclass[12pt]{amsart}
\usepackage{amsfonts, amssymb, latexsym, hyperref, tikz,tikz-cd, tikz-3dplot, pgfplots, transparent}
\usepackage[dvipsnames]{xcolor}
\usepackage{ytableau, caption, subcaption}
\usepackage{pgfplots}
\pgfplotsset{compat=1.18}

\newtheorem{theorem}{Theorem}[section]
\newtheorem{proposition}[theorem]{Proposition}
\newtheorem{lemma}[theorem]{Lemma}

\newtheorem{claim}[theorem]{Claim}

\newtheorem*{claim*}{Claim}
\newtheorem{corollary}[theorem]{Corollary}
\newtheorem{Main Conjecture}[theorem]{Main Conjecture}

\newtheorem{problem}[theorem]{Problem}
\newtheorem{myalgorithm}[theorem]{Algorithm}

\theoremstyle{definition}
\newtheorem{definition}[theorem]{Definition}
\theoremstyle{remark}

\newtheorem{example}[theorem]{Example}

\newtheorem{remark}[theorem]{Remark}

\theoremstyle{plain}

\newtheorem{question}{Question}

\hypersetup{colorlinks, linkcolor={red!75!black},
citecolor={blue!50!black}, urlcolor={blue!80!black}}

\newcommand\into{\operatorname*{\hookrightarrow}}
\newcommand\complexes{{\mathbb C}}
\newcommand\integers{{\mathbb Z}}

\newcommand{\antidiag}{{\prec_{\mathrm{antidiag}}}} 
\newcommand{\diag}{\prec_{\mathrm{diag}}}

\newcommand{\gap}{\hspace{1in} \\ \vspace{-.2in}}

\newcommand{\excise}[1]{}

\newcommand{\C}{\mathbb{C}}

\newcommand{\inv}{^{-1}}

\newcommand{\init}{\mathrm{init}}
\newcommand{\monos}{{\sf Mat}}

\newcommand{\ytabb}[1]{
    \ytableausetup{boxsize=1.2em,aligntableaux=center}
    \mbox{\small\begin{ytableau} #1 \end{ytableau}}
}

\newcommand{\ytabs}[1]{
    \ytableausetup{smalltableaux,aligntableaux=center}
    \begin{ytableau} #1 \end{ytableau}
    \ytableausetup{nosmalltableaux}
}
 
\newcommand{\ydiag}[1]{
    \ytableausetup{boxsize=.6em,aligntableaux=center}
    \ydiagram{#1}
}

\newcommand{\ydiags}[1]{
    \ytableausetup{boxsize=0.4em,aligntableaux=center}
    \ydiagram{#1}
}

\begin{document}
\pagestyle{plain}
\title{Gr\"obner Crystal Structures}
\author{Abigail Price}
\author{Ada Stelzer}
\author{Alexander Yong}
\address{Dept.~of Mathematics, U.~Illinois at Urbana-Champaign, Urbana, IL 61801, USA}
\email{price29@illinois.edu, astelzer@illinois.edu, ayong@illinois.edu}
\date{October 8, 2025}

\begin{abstract}
We develop a theory of \emph{bicrystalline ideals}, synthesizing Gr\"obner basis techniques and Kashiwara's crystal theory. 
This provides a unified algebraic, combinatorial, and computational approach that applies to ideals of interest, old and new.
The theory concerns ideals in the coordinate ring of matrices, stable under the action of some Levi group, whose quotients admit standard bases equipped with a crystal structure. We construct an effective algorithm to decide if an ideal is bicrystalline. When the answer is affirmative, we provide a uniform, generalized \emph{Littlewood--Richardson rule} for computing the multiplicity of irreducible representations either for the quotient or the ideal itself. 
\end{abstract}

\maketitle

\section{Introduction}\label{sec:intro}

We are motivated by I.~M.~Gelfand's perspective viewing representation theory as the study of symmetries on function spaces. A group $G$ acting on a space $\mathfrak{X}$ induces a $G$-action on various spaces of functions on $\mathfrak{X}$. The appropriate choice of function space depends on the structures with which $G$ and $\mathfrak{X}$ are endowed. For example, if $G$ is a topological group, one studies its action on a measure space $\mathfrak{X}$, leading to the theory of strongly continuous unitary representations on the Hilbert space $L^2(\mathfrak{X})$. If $G$ is an algebraic group, one instead considers actions on (affine) algebraic varieties or schemes $\mathfrak{X}$, giving rise to the theory of polynomial representations on the coordinate ring $\complexes[\mathfrak{X}]$ of polynomial functions on $\mathfrak{X}$. 

The central thesis of our work is that, in many cases, $\mathbb{C}[\mathfrak{X}]$ admits a \emph{Gr\"obner crystal structure} (GCS), a standard basis equipped with a Kashiwara crystal graph structure \cite{Kashiwara, Kashiwara2} which yields a generalized Littlewood--Richardson rule for computing irreducible multiplicities of $\mathbb{C}[\mathfrak{X}]$ as a $G$-repre\-sentation. We complete development of a GCS framework, initiated in our companion paper \cite{AAA} and rooted in the interaction of Gr\"obner theory with crystal combinatorics, that makes this idea precise, algorithmically decidable, and broadly applicable.
 
 In addition to the explicit role of \cite{Kashiwara, Kashiwara2} in our construction, in the form studied by M.~van Leeuwen \cite{bicrystal1} and V.~I.~Danilov--G.~A.~Koshevoi \cite{bicrystal2}, we mention other major influences. Crystal operators
 are also a combinatorial shadow of the canonical bases of G.~Lusztig in \cite{Lusztig, Lusztig2}. P.~Littelmann's work \cite{Littelmann94} on tensor-product and Levi-branching multiplicities relates \cite{Lusztig, Lusztig2, Kashiwara} and the \emph{Standard Monomial Theory} (SMT) of V.~Lakshmibai--C.~Musili--C.~S.~Seshadri \cite{SMT79}, originally developed for flag and Schubert varieties. The roots of SMT trace back to W.~V.~D.~Hodge's study of Pl\"ucker embeddings of Grassmannians \cite{Hodge}, which also presaged the development of Gr\"obner theory by B.~Buchberger \cite{Buchberger}.

Suppose $G$ is a Levi subgroup of $GL_m \times GL_n$ acting via row and column operations on a subvariety (or subscheme) $\mathfrak{X} \subseteq {\sf Mat}_{m,n}$, the space of $m \times n$ complex matrices. In~\cite{AAA}, we 
defined the \emph{bicrystalline} notion for \emph{varieties} and studied an instance of this notion arising from Schubert varieties \cite{Fulton:duke}. The subcase where ${\mathfrak{X}}={\sf Mat}_{m,n}$ is already interesting, being equivalent to 
\emph{Schur--Weyl duality} between irreducible representations of general linear and symmetric groups. This paper develops the 
theory in the proper generality of  \emph{ideals}. Missing from~\cite{AAA} 
was an algorithm to decide whether a given ideal is bicrystalline. We now provide an effective algorithm, and when the bicrystalline property holds, we offer a new, uniform combinatorial rule for computing irreducible multiplicities in
$\mathbb{C}[\mathfrak{X}]$. 

Beyond the bicrystalline GCS framework, we wish to demonstrate that Gr\"obner crystal structures provide a flexible method to organize and reveal multiplicity data outside of classical representation theory,
opening novel directions at the nexus of algebraic geometry, representation theory, and combinatorics. 

In the representation theory of general linear groups, the problems of giving combinatorial rules for irreducible multiplicities of tensor products and Levi-branchings of representations are both solved by the \emph{Littlewood--Richardson rule}; see, e.g., \cite{Fulton, ECII}. This was vastly extended in a root-system uniform manner to complex semisimple Lie algebras (and their associated complex Lie groups) in the aforementioned works of \cite{Littelmann94, Kashiwara}, and with a different, polytopal solution in work of A.~Berenstein--A.~Zelevinsky \cite{BZ}. 

Rather than generalizing to well-behaved (e.g., connected, complex, reductive) Lie groups and their Lie algebras as in \cite{Kashiwara, Littelmann94, BZ}, we pursue an extension to bicrystalline ideals inside ${\mathbb C}[{\sf Mat}_{m,n}]$. 
View the Littlewood--Richardson rule as the solution to the branching problem for ${\mathbb C}[{\sf Mat}_{m,n}]$ under the action of $(GL_k\times GL_{m-k})\times GL_n$ (see \cite{Howe} and Example~\ref{exa:Levi-branch}). By varying the choice of the Levi group $G$, we put the classical information of the Hilbert function and these Levi-multiplicities  on a single spectrum. 

\subsection{Motivating examples}\label{subsec:motivating} 

\begin{example}[Symmetric algebra]\label{exa:symmetric}
Identify
${\mathrm{Sym}}({\mathbb C}^m)$ with the coordinate ring of ${\mathfrak X}={\mathbb C}^m$
(the space of $m\times 1$ matrices). There is a 
$GL_m\times {\mathbb C}^\times$ action on ${\mathfrak X}$, and hence on ${\mathrm{Sym}}({\mathbb C}^m)$. 

An identity for the character of ${\mathrm{Sym}}({\mathbb C}^m)$ is
\begin{equation}
\label{eqn:onesidedcauchy}
\prod_{k=1}^{m} \frac{1}{1-x_iy}=\sum_{d=0}^{\infty} h_d(x_1,\ldots,x_m)y^d,
\end{equation}
where $h_d(x_1,\ldots,x_m)$ is the \emph{homogenenous symmetric polynomial} of degree $d$.
\end{example}

\begin{example}[Determinantal varieties]\label{exa:det}
Let $V$ be a $k$-dimensional vector space. The group $GL(V)$ acts on the space $V^{\oplus n}\oplus
(V^*)^{\oplus m}$ of $n$ vectors and $m$ covectors . The ring of invariants ${\mathbb C}[V^{\oplus n}\oplus
(V^*)^{\oplus m}]^{GL(V)}$ is finitely generated as a ${\mathbb C}$-algebra by \emph{contractions} $z_{ij}$, defined
by setting $z_{ij}(\ldots, \vec e_j,\ldots; \ldots, \vec f_i, \ldots)=\vec f_i(\vec e_j)$ ($1\leq i\leq m, 1\leq j\leq n$) and extending linearly. There is a ring isomorphism
\[{\mathbb C}[V^{\oplus n}\oplus
(V^*)^{\oplus m}]^{GL(V)}\cong {\mathbb C}[{z}_{ij}]/I_{k+1},\]
where $I_{k+1}$ is the ideal of $(k+1)\times (k+1)$ minors of the generic matrix $[z_{ij}]_{1\leq i\leq m,1\leq j\leq n}$. The \emph{determinantal variety}
$\mathfrak{X}_{k}$ of matrices with rank at most $k$, cut out by $I_{k+1}$, has an action of $GL_m\times GL_n$ by row and column operations.
These are  the only (reduced) varieties in the space ${\sf Mat}_{m,n}$ of $m\times n$ matrices 
with this action. The character of $\mathfrak{X}_{k}$ is given by the expression (cf.~Example~\ref{exa:detidealredux})
\begin{equation}
\label{eqn:detcauchy}
\sum_{\lambda: \ell(\lambda)< k+1} s_{\lambda}(x_1,\ldots,x_m)s_{\lambda}(y_1,\ldots,y_n),
\end{equation}
where the sum is over integer partitions $\lambda$ with at most $k$ parts. Here, 
e.g., $s_{\lambda}(x_1,\ldots,x_m)$ is the \emph{Schur polynomial}, the character of an irreducible $GL_m$ representation.
\end{example}

\begin{example}[Veronese embeddings]\label{exa:veronese}
The  second 
\emph{Veronese embedding}
${\mathbb P}^2\to {\mathbb P}^5$ is 
 \[[z_0:z_1:z_2]\mapsto [z_0^2: z_0z_1:z_0z_2:z_1^2: z_1z_2: z_2^2]=[w_0:w_1:w_2:w_3:w_4:w_5].\]
$GL_3$ acts linearly on the original variables $z_0,z_1,z_2$, inducing an action
on $w_0,w_1,\ldots,w_5$. 
The image ${\mathfrak X}$ is cut out by the 
$2\times 2$ minors of the symmetric matrix
$M=\left[\begin{smallmatrix} w_0 & w_1 & w_2 \\ 
w_1 & w_3 & w_4 \\ 
w_2 & w_4 & w_5\end{smallmatrix}\right]$. Take the action of $g\in GL_3$ to be by $g^{-1}M{(g^{-1})}^t$. 

The character of ${\mathbb C}[{\mathfrak X}]$ is
\begin{equation}
\label{eqn:Veronesechar}
1+h_{2}(x_1,x_2,x_3)+h_4(x_1,x_2,x_3)+h_6(x_1,x_2,x_3)+\cdots.
\end{equation}
\end{example}

\begin{example}[Matrix matroid varieties]\label{exa:matroid}
A \emph{realizable matroid} is an ordered configuration of vectors 
$\vec v_1,\ldots,\vec v_k\in {\mathbb C}^n$,
viewed as columns of an $n\times k$ matrix $C$. Following \cite[Example~2.2]{FNR}, let
$C=\left[\begin{smallmatrix}
0 & 0 & 0 & 1 & 1 & 0 \\
1 & 1 & 1 & 1 & 1 & 0
\end{smallmatrix}\right]$.
The \emph{matrix matroid ideal} $I_C$ of \cite{FNR} defines the closure of the 
$GL_2\times T$ orbit of this matrix, where $GL_2$ acts on the rows and $T=({\mathbb C}^\times)^6$ rescales the columns.

A.~Berget--A.~Fink \cite{Berget.Fink} express the character of $\complexes[{\sf Mat}_{2, 6}]/I_C$ in the quotient form
\[\frac{
1- s_{\ydiags{1,0}} y_6  
- s_{\ydiags{1, 1}} y_1 y_2 +\cdots 
+ s_{\ydiags{3, 3}} y_2 y_3 y_4 y_5 y_6^2  
- s_{\ydiags{4, 3}} y_1 y_2 y_3 y_4 y_5 y_6^2  
+ s_{\ydiags{4, 2}}   y_1 y_2 y_3 y_4 y_5 y_6}{(1-x_1y_1)\cdots (1-x_1y_6)(1-x_2y_1)\cdots (1-x_2y_6)},\]
where each $s_{\mu}:=s_{\mu}(x_1,x_2)$.

However, earlier work did not give a rule for the positive expansion:
\begin{multline}\nonumber
\!\!\!\!\!\!\!=1 + s_{\ydiags{1, 0}}y_1 + s_{\ydiags{1, 0}}y_2 + s_{\ydiags{1, 0}}y_3 + s_{\ydiags{1, 0}}y_4 + s_{\ydiags{1, 0}}y_5
+ s_{\ydiags{2, 0}}y_1^2 + s_{\ydiags{2, 0}}y_1y_2 + s_{\ydiags{2, 0}}y_1y_3 + s_{\ydiags{1, 1}}y_1y_4+\cdots
\end{multline}
We give a rule for matrix matroid ideals satisfying the bicrystalline hypothesis, valid for any Levi that acts.  See the rule Theorem~\ref{thm:simplifiedGrobnerdet} and its application in Example~\ref{ex:matriodLR} 
for more details. Another instance is Example~\ref{exa:graphicalmatroid}, which comes from
a \emph{graphical matroid}.
\end{example}

\begin{example}[Double Bruhat ideals]\label{exa:doublebruhat}
\emph{Double Bruhat cells} \cite{BFZ, FZ} play a role in \emph{total positivity} and are among the original motivating examples for the theory of \emph{cluster algebras} \cite{FZ:cluster}. They are 
defined as  
$B_{-}uB\cap BvB_{-}\subset GL_n$, 
where $B, B_{-}$ are, respectively, invertible upper and lower triangular matrices in $GL_n$,
and $u,v$ are permutation matrices.\footnote{Originally, they are defined as $BuB\cap B_{-}vB_{-}$, but our choice of convention is better for our exposition.} A.~Knutson considered their closure inside ${\sf Mat}_{n,n}$.\footnote{Private communication to the third author, circa 2005.}

For $u=v=2143$, the corresponding double Bruhat ideal is
\[I=\langle z_{11}, \text{northwest $3\times 3$ minor, southeast $3\times 3$ minor}, z_{44}\rangle \!\subseteq\!
{\mathbb C}[{\sf Mat}_{4,4}].\]
This is a $(GL_1\times GL_2\times GL_1)\times (GL_1\times GL_2\times GL_1)$-stable ideal.  The character of $\complexes[{\sf Mat}_{4, 4}]/I$ begins
\[1+x_4s_{\ydiags{1}}(y_2,y_3)+x_4y_1+s_{\ydiags{1}}(x_1,x_2)y_4+s_{\ydiags{1}}(x_2,x_3)s_{\ydiags{1}}(y_2,y_3)+s_{\ydiags{1}}(x_2,x_3)y_1+x_1y_4+x_1s_{\ydiags{1}}(y_2,y_3)+\cdots\]
By Theorem~\ref{thm:knutsonbicrystal}, the character of any double Bruhat ideal is computed using Theorem~\ref{thm:LRrule}.
\end{example}

\begin{example}[Buchsbaum--Eisenbud variety of complexes \cite{Buchsbaum}]\label{exa:quiver}
Let 
\[\mathfrak{X}=\{(A,B)\in {\sf Mat}_{2,2}\times{\sf Mat}_{2, 2} :  AB={\bf 0}\}.\] 
Then $GL_2\times GL_2\times GL_2$ acts on $\mathfrak{X}$ (which is not irreducible) by 
\[(g_1,g_2,g_3)\cdot(A,B)=(g_1 A g_2^{-1}, g_{2}B g_3^{-1}).\] 
The character of $\complexes[\mathfrak{X}]$ begins
\[1+s_{\ydiags{1}}\otimes s_{\ydiags{1}}\otimes s_{\emptyset} 
+s_{\emptyset}\otimes s_{\emptyset, -\ydiags{1}}\otimes s_{\ydiags{1}} 
+s_{\ydiags{1,1}}\otimes s_{\ydiags{1,1}}\otimes s_{\emptyset}+
s_{\ydiags{2}}\otimes s_{\ydiags{2}}\otimes s_{\emptyset} +
s_{\ydiags{1}}\otimes s_{\ydiags{1},-\ydiags{1}}\otimes s_{\ydiags{1}}+\cdots,
\]
where $s_{\lambda}\otimes s_{\mu}\otimes s_{\nu}:=s_{\lambda}(x_1,x_2)s_{\mu}(y_1,y_2)s_{\nu}(z_1,z_2)$. Also,
 $s_{\emptyset, -\ydiags{1}}$ and $s_{\ydiags{1},-\ydiags{1}}$ are rational Schur polynomials for the partitions $(\emptyset,-\ydiags{1})=(-1)$ and
 $(\ydiags{1},-\ydiags{1})=(1,-1)$ respectively (we refer to 
\cite{Stembridge:rational}). For a related study of ${\mathbb C}[{\mathfrak X}]$ see De Concini--Strickland's \cite{Strickland}. Discussion
in Section~\ref{sec:concluding} remarks on more recent work regarding these varieties.
\end{example}

There are also naturally occurring \emph{non}-reduced examples of bicrystalline ideals:

\begin{example}[Thick determinantal ideals]\label{exa:powers}
Let $I\subseteq {\mathbb C}[{\sf Mat}_{3,3}]$ be the ideal of $2\times 2$ minors of a generic $3\times 3$ matrix. 
The \emph{ordinary power} $I^2$ is generated by all products of elements of~$I$. The \emph{symbolic power}
$I^{(2)}$ is generated by all polynomials in ${\mathbb C}[{\sf Mat}_{3,3}]$ that vanish to order at least $2$ on the (prime) ideal $I$.\footnote{More precisely: for a prime ideal $I$ in a polynomial ring over an algebraically closed field, $I^{(n)} = \bigcap_{\mathfrak{m}}\mathfrak{m}^n$, where the intersection is over all maximal ideals $\mathfrak{m}$ containing $I$. See, e.g., \cite[Theorem 3.14]{Eisenbud}.}
 Neither $I^2$ nor $I^{(2)}$ is radical. Each of $I,I^2,I^{(2)}$ carries a $GL_3\times GL_3$ action. 

The character for ${\mathbb C}[{\sf Mat}_{3,3}]/I$ is computed using \eqref{eqn:detcauchy}. Now,
letting 
\[s_{\lambda}\otimes s_{\lambda}:=s_{\lambda}(x_1,x_2,x_3)s_{\lambda}(y_1,y_2,y_3),\] 
the characters for the quotients by the two powers of $I$ only differ in one term, as marked:
\[I^2: \ 1+s_{\ydiags{1}}\otimes s_{\ydiags{1}}+s_{\ydiags{1,1}}\otimes s_{\ydiags{1,1}}
+\fbox{$s_{\ydiags{1,1,1}}\otimes s_{\ydiags{1,1,1}}$}
+s_{\ydiags{2}}\otimes s_{\ydiags{2}}+
s_{\ydiags{2,1}}\otimes s_{\ydiags{2,1}}+s_{\ydiags{3}}\otimes s_{\ydiags{3}}+
s_{\ydiags{3,1}}\otimes s_{\ydiags{3,1}}+\cdots\]
\[I^{(2)}: \ 1+s_{\ydiags{1}}\otimes s_{\ydiags{1}}+s_{\ydiags{1,1}}\otimes s_{\ydiags{1,1}}+s_{\ydiags{2}}\otimes s_{\ydiags{2}}+
s_{\ydiags{2,1}}\otimes s_{\ydiags{2,1}}+s_{\ydiags{3}}\otimes s_{\ydiags{3}}+
s_{\ydiags{3,1}}\otimes s_{\ydiags{3,1}}
+\cdots\]
We describe a general rule for these expansions that, in particular, explains said difference; see 
Example~\ref{exa:detpower2}. This is done using results about \emph{bitableaux} \cite{DRS} from invariant theory. 
See \cite[Section 3.9.1]{Eisenbud} for further discussion of this example from the commutative algebra perspective.
\end{example}

\subsection{The GCS thesis}
M.~Kashiwara~\cite{Kashiwara, Kashiwara2}  introduced the notion of \emph{crystal graphs} to the
study of complex semisimple Lie algebras and their representations.\footnote{See G.~Lusztig's \cite{Lusztig, Lusztig2} which introduced, in a geometric manner, the same underlying \emph{crystal bases}, there called \emph{canonical bases}. We will not use these bases \emph{per se} in this paper.}

\begin{example}[tensor power of the standard representation]\label{exa:wordcrystal} The standard representation of $GL_n$ is $V_{\ydiags{1}}={\mathbb C}^n$ with the matrix multiplication action. Every irreducible representation of degree $k$ appears in the tensor power $V_{\ydiags{1}}^{\otimes k}$. The decomposition of $V_{\ydiags{1}}^{\otimes k}$ is modeled by a crystal graph $\mathcal{W}_k$ on $k$-letter \emph{words} from the alphabet $[n]:=\{1,2,\ldots,n\}$. It is defined via \emph{raising operators} $e_i$
and \emph{lowering operators} $f_i$ on words, which output another word or $\varnothing$. The operators $e_i$ and $f_i$ are defined using the \emph{bracket sequence} ${\sf bracket}_i(w)$, obtained by recording a {\tt )} symbol for each $i$ in $w$ and a {\tt (} symbol for each $i+1$ in $w$.

For example, if 
$w=213142$ and $i=1$, the map ${\sf bracket}_1$ sends
\[\underline{2} \ \underline{1} \ 3 \ \underline{1} \ 4 \ \underline{2} \mapsto \text{{\tt ())(}}. \]
The crystal operators then alter the brackets (and thereby the word $w$) as follows.
\begin{itemize}
\item $f_i=$ lowering: turns the rightmost unmatched {\tt )} to {\tt (}.
\medskip
\item $e_i=$ raising: turns the leftmost unmatched {\tt (} to {\tt )}.
\end{itemize}
\[f_1(213\fbox{1}42)=213\fbox{2}42, \ e_1(21314{\fbox{2}})=21314{\fbox{1}}\]
If there is no unmatched {\tt )} or {\tt (}, then $f_i(w)=\varnothing$ or $e_i(w)=\varnothing$. 

The lowering operators (or the raising operators) define a directed graph on words. Each connected component
has a unique source vertex, a \emph{highest weight word} for which every raising operator returns~$\varnothing$.
The generating series for a connected component is an irreducible $GL_n$-character, i.e., a Schur polynomial. Thus crystals group the monomials of a character into Schur polynomials to give expressions akin to those of our examples.
\end{example}

\begin{figure}
    \begin{center}
    \begin{tikzpicture}[scale = 0.40, align = left]
        \draw[red, thick][->] (1, 6) -- (4, 9);
        \draw[red, thick][->] (14, 9) -- (17, 6);
        \draw[blue, thick][->] (6, 10) -- (8, 10);
        \draw[blue, thick][->] (10, 10) -- (12, 10);
        \draw[blue, thick][->] (1, 4) -- (4, 1);
        \draw[blue, thick][->] (14, 1) -- (17, 4);
        \draw[red, thick][->] (6, 0) -- (8, 0);
        \draw[red, thick][->] (10, 0) -- (12, 0);
        \draw (0, 5) node{ 211 };
        \draw (5, 10) node{212};
        \draw (9, 10) node{213};
        \draw (13, 10) node{313};
        \draw (5, 0) node{311};
        \draw (9, 0) node{ 312 };
        \draw (13, 0) node{322};
        \draw (18, 5) node{323};
        \draw (2,8) node{$f_1$};
        \draw (2,2) node{$f_2$};
        \draw (7,11) node{$f_2$};
        \draw (11,11) node{$f_2$};
        \draw (16,8) node{$f_1$};
        \draw (7,-1) node{$f_1$};
        \draw (11,-1) node{$f_1$};
        \draw (16,2) node{$f_2$};
    \end{tikzpicture}
    \end{center}
    \caption{ \label{fig:July14aaa} The crystal graph with highest weight word $211$.}
    \end{figure}
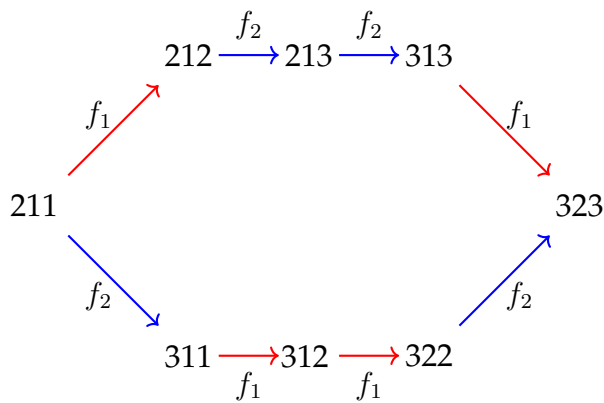

In this paper, we consider characters of coordinate rings, which are better known in commutative algebra as the (multigraded) \emph{Hilbert series} of embedded projective varieties.

\begin{example}[Standard graded case]\label{exa:stdgraded}
Suppose 
$I\subseteq R:={\mathbb C}[z_1,\ldots,z_k]$
is an ideal and $I$ is \emph{homogeneous}, that is, 
it is generated by polynomials in which each term is of the same total degree. (All our examples
from Section~\ref{subsec:motivating} have this property.) Then
\[R/I=\bigoplus_{d\geq 0} (R/I)_d\] is a graded vector space over ${\mathbb C}$, where the graded component $(R/I)_d$ consists of (classes of) those polynomials equivalent to some degree-$d$ homogeneous polynomial modulo $I$. For each term order $\prec$ on $R$, Gr\"obner theory provides a graded \emph{standard basis} for $R/I$ as a vector space. 
If ${\mathrm{init}}_{\prec} I$ is the initial ideal
of $I$ generated by leading terms of elements in $I$, the standard basis consists of all monomials in $R$ \emph{not} in 
${\mathrm{init}}_{\prec} I$.

$I$ is homogeneous if and only if the associated affine scheme $\mathfrak{X}$ is stable under the dilation action of ${\mathbb C}^\times=GL_1$. The character of the coordinate ring
$\complexes[\mathfrak{X}] = R/I$ is then $\sum_{d\geq 0} \dim_{\mathbb C} (R/I)_d t^d$, which is the (standard graded) Hilbert series of $\mathfrak{X}$.
This example generalizes: a larger torus action on $R/I$ corresponds to a multigrading on $R$. See Section~\ref{subsection:bipre}. The character of $R/I$ is then its multigraded Hilbert series; see 
Section~\ref{sec:prelim} and specifically Example~\ref{exa:char=Hilb}.
\end{example}

Suppose a linear algebraic group $G$ acts on $R/I$. We consider a natural question:
\begin{quotation}
What is a combinatorial counting rule for the multiplicities of the $G$-irreducible representations in $R/I$? 
\end{quotation}

Our guiding principle is that one achieves such rules 
by imposing a \emph{crystal structure} on the monomials of $R$ which descends to a crystal structure on a standard basis $B$ of $R/I$; we call this a \emph{Gr\"obner crystal structure} (GCS) on the triple $(R,I,\prec)$ (see Definition~\ref{def:GCS}). Now, under typical hypotheses, one can artificially impose a crystal structure on the standard basis, so in this sense, a GCS for $I$ always exists.\footnote{For example, if each graded component $(R/I)_d$ is a finite-dimensional polynomial representation of $GL_n$, then the standard monomials spanning $(R/I)_d$ are in \emph{some} multiset bijection with semistandard Young tableaux that preserves the finer grading from the maximal torus $T\subseteq GL_n$. Now use  the crystal structure on the tableaux (see Definition~\ref{def:tabbicops}) to induce a crystal structure on the monomials.}  However, we study a particular instantiation of the GCS thesis, which captures Examples
~\ref{exa:symmetric},~\ref{exa:det},~\ref{exa:matroid},~\ref{exa:doublebruhat}, ~\ref{exa:powers}, and their natural generalizations in a \emph{uniform} manner. In this setting, we use the bicrystal operators of
van Leeuwen and Danilov--Koshevoi \cite{bicrystal1, bicrystal2}, which are certain pullbacks of Kashiwara's operators along the Robinson--Schensted--Knuth (RSK) correspondence (Proposition~\ref{prop:pullback}). We believe that suitable modification of this construction, where, e.g., one uses a variation on RSK, will handle many other such cases, including Example~\ref{exa:veronese}
and~\ref{exa:quiver}. See further discussion in Sections~\ref{sec:non-comm} and~\ref{sec:concluding}.

\subsection{Summary of results; organization}
In Section~\ref{sec:bicrystalline}, after presenting basics from combinatorial commutative algebra, we introduce our main concept, \emph{bicrystalline ideals}. In our earlier paper \cite{AAA}, we showed that the defining ideals of (unions of) matrix Schubert varieties are bicrystalline and suggested that the bicrystalline property should be more general. However, in \emph{ibid.}~we presented neither additional examples nor non-examples -- we now rectify the situation. We reformulate the bicrystalline notion using \emph{test sets} (Definition~\ref{def:test}). Our first result
shows that construction of a test set gives a finite check for the bicrystalline property 
(Theorem~\ref{thm:firstmain}). We also prove the existence of minimal test sets (Theorem~\ref{thm:minimality}).

Section~\ref{sec:alg} presents an algorithm (Theorem~\ref{thm:main}) to decide if an ideal is bicrystalline by constructing test sets (Theorem~\ref{thm:testsetalg}). The algorithm, and its proof of correctness, offer a general technique to prove a given family of ideals is bicrystalline.

In Section~\ref{sec:prelim} we provide background in representation theory and tableau combinatorics.
This prepares for Section~\ref{sec:LR}, where we give a combinatorial rule (Theorem~\ref{thm:LRrule})
for determining the multiplicities in the irreducible decomposition of a bicrystalline ideal or the associated coordinate ring. This reformulates and extends the rule given in \cite{AAA} by replacing use of the \emph{Filtered RSK algorithm} with a modified ballot condition on pairs of semistandard tableaux $(P,Q)$ associated to a standard monomial under the classical RSK correspondence.
We illustrate the rule for examples of matrix matroid ideals  \cite{FNR},
powers of determinantal ideals (see, e.g., \cite{Trung, BC, BC03}),  and matrix Hessenberg ideals~\cite{Goldin.Precup}.

Together, Theorem~\ref{thm:firstmain} and Theorem~\ref{thm:LRrule} provide a proverbial ``one-two punch'', giving rules for the desired representation multiplicities in many instances. Sections~\ref{subsec:bidi},~\ref{sec:Knutsonideals}, and~\ref{sec:in-RSK} demonstrate our method's applicability to certain large families of ideals.

In Section~\ref{subsec:bidi}, by following the decision algorithm of Section~\ref{sec:alg} we construct explicit test sets to prove that  ideals we call \emph{Gr\"obner-determinantal ideals} are indeed bicrystalline (Theorem~\ref{thm:generalizeddetideals}). We proceed to give a simplified version of Theorem~\ref{thm:LRrule} for these ideals when one side of the action is by a full general linear group (Theorem~\ref{thm:simplifiedGrobnerdet}). This is applicable to the matrix matroid ideal of Example~\ref{exa:matroid}; see Example~\ref{ex:matriodLR}.

In Section~\ref{sec:Knutsonideals}, we use the results of Section~\ref{subsec:bidi} and theorems of A. Knutson \cite{Knutson} to
show (Theorem~\ref{thm:knutsonbicrystal}) that a large family of 
\emph{Knutson determinantal ideals} are Gr\"obner-determinantal ideals and therefore bicrystalline. This family includes classical determinantal ideals, Schubert determinantal ideals, and matrix double Bruhat ideals. Thus, Theorem~\ref{thm:knutsonbicrystal} generalizes the setting of Example~\ref{exa:det} and the main application from \cite{AAA}. Appendix~\ref{theappendix} is an elementary proof of a result of Knutson that we need for these conclusions. It describes the leading terms of the ``basic minors'' of the \emph{Kazhdan--Lusztig ideals} defined in \cite{WY:governing}. 

In Section~\ref{sec:in-RSK} we study the context of Example~\ref{exa:powers}, ideals spanned by the bitableaux 
of~\cite{DRS}. As highlighted by work of W.~Bruns--A.~Conca
\cite{BC}, a subclass of these ideals, called \emph{in-KRS}, are of particular significance. We prove that all such ideals that carry a $GL_n$-action, including (symbolic) powers of determinantal ideals, are bicrystalline.

Section~\ref{sec:non-comm} gives two vignettes regarding the GCS thesis in non-commutative settings.

Section~\ref{sec:concluding} offers concluding remarks, including perspectives for future work.

\section{Bicrystalline ideals}\label{sec:bicrystalline}
\subsection{Notation and preliminaries} \label{subsection:bipre}
Let ${\sf Mat}_{m,n}$ be the space of $m\times n$ matrices with entries in ${\mathbb C}$. Identify ${\mathbb C}[{\sf Mat}_{m,n}]$ with a polynomial ring in the $m\times n$ matrix of variables 
$Z = [z_{ij}]$. View $\complexes[{\sf Mat}_{m, n}]$ momentarily as a vector space, forgetting the multiplication operation. Then for any ideal $I\subseteq\complexes[{\sf Mat}_{m, n}]$ there is an isomorphism of vector spaces:
\begin{equation}\label{eqn:basicdecomp}
    \complexes[{\sf Mat}_{m, n}]\cong I\oplus \complexes[{\sf Mat}_{m, n}]/I.
\end{equation}
The decomposition \eqref{eqn:basicdecomp} is modeled combinatorially by monomials and the theory of Gr\"obner bases. Throughout this paper, we will identify monomials 
\[{\bf m}= z^M := \prod_{i, j}z_{ij}^{M_{ij}}\in \complexes[{\sf Mat}_{m, n}]\] 
by their exponent matrices $M = [M_{ij}]$ in the space ${\sf Mat}_{m,n}(\integers_{\geq 0})$ of matrices with entries in $\integers_{\geq 0} = \{0,1,2,\dots\}$. As we work with polynomials over a field exclusively, we assume without loss of generality that all monomials have scalar coefficient $1$.

Fix a choice of \emph{term order} $\prec$ on the monomials in $\complexes[{\sf Mat}_{m, n}]$. Our reference for Gr\"obner basis theory is \cite{CLO}. In most of our examples, $\prec$ is an \emph{antidiagonal term order}
or a \emph{diagonal term order}, meaning that it is some term order which picks the antidiagonal (respectively, diagonal)
term  of any minor of $Z$ as the lead term. There are many (anti)diagonal term orders. Much of our analysis and many of our examples are valid for
all such (anti)diagonal term orders. In those cases, we denote any of them by 
$\antidiag$ and $\diag$, respectively.

 The \emph{initial term} ${\mathrm{init}}_{\prec} f$ of $f\in \complexes[{\sf Mat}_{m, n}]$ is its largest monomial with respect to $\prec$. The \emph{initial ideal} and set of \emph{standard monomials} for $I$ are, respectively, 
\[{\mathrm{init}}_{\prec} I = \langle {\mathrm{init}}_{\prec} \ f: f\in I\rangle, \text{ \ and \ } 
{\mathrm{Std}}_{\prec}I=\{z^M: z^{M}\not\in {\mathrm{init}}_{\prec} \ I\}.\]
The key fact is that $\mathrm{Std}_\prec I$ is a vector space basis for $\complexes[{\sf Mat}_{m, n}]/I$ (see, e.g., \cite[pg. 158]{Miller.Sturmfels}). Taking exponent matrices of the standard and non-standard monomials for $I$, we obtain the following combinatorial model for \eqref{eqn:basicdecomp}:
\begin{align}
    \complexes[{\sf Mat}_{m, n}]&\cong(\init_\prec I)\oplus\mathrm{span}_\complexes(\mathrm{Std}_\prec I)\label{eqn:degendecomp},\\
    {\sf Mat}_{m, n}(\integers_{\geq0}) &= \{M:z^M\in\init_\prec I\}\sqcup\{M:z^M\in\mathrm{Std}_\prec I\}.\label{eqn:combdecomp}
\end{align}

Since we often use exponent matrices rather than monomials, as in \eqref{eqn:combdecomp}, we also define
\[\monos_\prec I := \{M\in{\sf Mat}_{m, n}(\integers_{\geq 0}): z^M\in\init_\prec I\}.\]
It follows immediately from the definitions that for all $M\in{\sf Mat}_{m, n}(\integers_{\geq 0})$, $\prec$, and $I$:
\begin{equation}\label{eqn:stdmonos}
M\notin\monos_\prec I \iff z^M\in\mathrm{Std}_\prec I.
\end{equation}
A set ${\mathcal G}=\{g_1,\ldots,g_r\}$ of elements of $I$
form a \emph{Gr\"obner basis} for $I$ with respect to $\prec$ if 
\[{\mathrm{init}}_{\prec} I=\langle {\mathrm{init}}_{\prec}(g_i)\, : \, 1\leq i\leq r\rangle.\]
Every ideal $I$ has a Gr\"obner basis, and Gr\"obner bases can be computed algorithmically from any generating set for $I$.

Now, the product of general linear groups 
\[{\bf GL}:=GL_m\times GL_n\]
acts on ${\sf Mat}_{m,n}$ via the right action
\begin{equation}\label{eqn:matrixaction}
    M\cdot (g, h)= g^{-1}M(h^{-1})^t,
\end{equation}
where $(g, h)\in\mathbf{GL}$, $M\in{\sf Mat}_{m, n}$, and ${}^t$ denotes matrix transpose. This right action induces a left $\mathbf{GL}$-action on $\complexes[{\sf Mat}_{m, n}]$:
\footnote{Our odd-looking choice of $\mathbf{GL}$-action on ${\sf Mat}_{m, n}$ is made so that all elements of $\complexes[{\sf Mat}_{m, n}]$ have positive multidegree under the torus multigrading given by this induced left action.
See also Section~\ref{subsec:reptheoryprelims}.}
\[(g, h)\cdot f(Z) := f(Z\cdot(g, h)\inv) = f(gZh^t)\quad \forall f\in\complexes[{\sf Mat}_{m, n}].\]

By restriction, $\complexes[{\sf Mat}_{m, n}]$ also carries an action of 
\[{\bf L}:=L(m)\times L(n)\] 
where $L(m)$ and $L(n)$ are 
\emph{Levi subgroups} of $GL_m$ and $GL_n$ respectively. That is, $L(m)$ is a direct sum of invertible block diagonal matrices. At one extreme, if $L(m)$ has one block of size $m$ then $L(m)=GL_m$. At the other end, if all blocks in $L(m)$ have size one then $L(m)=T_m$ is the maximal torus of invertible diagonal $m\times m$ matrices.
A \emph{Levi datum} consists of two sets of integers
\begin{equation}
\label{eqn:seqdef}
{\bf I}=\{0=i_0<i_1<\ldots< i_r=m\} \text{\ and \ } {\bf J}=\{0=j_0<j_1<\ldots <j_s=n\}.
\end{equation}
For each ${\bf I}$ one has a Levi subgroup $L_{\bf I}\leq GL_m$, where 
\[L_{\bf I}:=GL_{i_1-i_0}\times GL_{i_2-i_1}\times
\cdots \times GL_{i_r-i_{r-1}},\]
e.g.,  $\left[
\begin{smallmatrix}
* & * & 0 \\
* & * & 0 \\
0 &  0 & *
\end{smallmatrix}
\right]$ in the case of $GL_2\times GL_1 \leq GL_3$ (which corresponds to the set $\mathbf{I} = \{0, 2, 3\}$).
Similarly one defines $L_{\bf J}\leq GL_n$.

Any Levi group ${\bf L}$ contains the maximal torus $T_m\times T_n$ consisting of pairs of invertible diagonal matrices of size $m$ and $n$ respectively. 
The action of this torus induces a grading on ${\mathbb C}[{\sf Mat}_{m, n}]$ that assigns to the variable $z_{ij}$ the \emph{multidegree} $\vec \gamma_i+\vec \gamma_{m+j}\in {\mathbb Z}_{\geq 0}^{m+n}$, where $\vec \gamma_i$ is the standard basis vector. 
This is a multigrading in the sense of \cite[Definition~8.1]{Miller.Sturmfels}, which defines the concept as a homomorphism ${\mathrm{deg}}_A$ from the semigroup of exponent vectors to an abelian semigroup $A$.
Here, 
\[{\mathrm{deg}}_{\mathbb{Z}^{m+n}_{\geq 0}}\colon{\mathbb Z}^{mn}_{\geq 0}\to {\mathbb Z}^{m+n}_{\geq 0}\]
 is given by
\[\vec{\beta}_{ij}\mapsto \vec{\gamma}_i+\vec{\gamma}_{m+j}.\]
An abelian semigroup homomorphism  $\phi\colon A\to B$, 
induces a coarser grading 
\[{\mathrm{deg}}_B( - ):=\phi({\mathrm{deg}}_A(-)).\] 
The inclusion of a subtorus ${\bf T}\hookrightarrow T_m\times T_n$ induces such a coarsening homomorphism $\phi$. For example, consider the embedding of the $1$-dimensional subtorus
\[{\bf T}= \left\{\left(
\begin{bmatrix}t & 0 & \dots & 0\\
0 & t & \dots & 0\\
\vdots & \vdots & \ddots & \vdots\\
0 & 0 & \dots & t\end{bmatrix}
,\mathrm{Id}_n\right):t\in\complexes^\times\right\}\hookrightarrow T_m\times T_n.\]
Then $B = {\mathbb{Z}}_{\geq 0}$ and the induced coarsening $\phi:\integers^{m+n}_{\geq 0}\to\integers_{\geq 0}$ yields the standard grading on $\complexes[{\sf Mat}_{m, n}]$ where each variable has degree $1$.
For the ${\mathbb Z}^{m+n}_{\geq 0}$-multigrading, or any coarsening, the vector space ${\mathbb C}[{\sf Mat}_{m, n}]$ decomposes into a direct sum of graded components. 

Suppose now that $I\subseteq\complexes[{\sf Mat}_{m, n}]$ is a $\mathbf{L}$-stable ideal. 
Then both $I$ and $\complexes[{\sf Mat}_{m, n}]/I$ are $\mathbf{L}$-representations, and the vector space decomposition \eqref{eqn:basicdecomp} still holds as an isomorphism of $\mathbf{L}$-representations. 
The vector spaces $\init_\prec I$ and $\mathrm{span}_\complexes(\mathrm{Std}_\prec I)$ are not $\mathbf{L}$-representations in general, but they are always $T_m\times T_n$-representations, and the decomposition in \eqref{eqn:degendecomp} holds as an isomorphism of $T_m\times T_n$ representations. 
The crystal operators introduced in the next section will ultimately allow us to recover the $\mathbf{L}$-representation structure of \eqref{eqn:basicdecomp} from the $T_m\times T_n$-representation structure and combinatorics of \eqref{eqn:degendecomp} and \eqref{eqn:combdecomp}.  

\subsection{Main definitions}\label{subsec:maindef}
We recall the four 
\emph{bicrystal operators} 
\begin{equation}\label{eqn:thefour}
f_{i}^{\sf row}, e_i^{\sf row}, f_j^{\sf col}, e_j^{\sf col}: {\sf Mat}_{m,n}(\integers_{\geq 0})\to {\sf Mat}_{m,n}(\integers_{\geq 0})\cup \{\varnothing\}
\end{equation}
of M.~van Leeuwen and V.~I.~Danilov--G.~A.~Koshevoi \cite{bicrystal1, bicrystal2}. We start with $f_i^{\sf row}$.
Given $M\in {\sf Mat}_{m,n}(\integers_{\geq 0})$, its \emph{row word}, denoted ${\sf row}(M)$, is obtained by reading the nonzero entries of $M$ down columns, from left to right, and recording $M_{rc}$ copies of the row index $r$ for each entry $(r,c)$. 
Fix $1\leq i\leq m-1$. Compute the bracket sequence ${\sf bracket}_i({\sf row}(M))$ as in Example~\ref{exa:wordcrystal}, by replacing each $i$ with {\tt )} and each $i+1$
with  {\tt (}. Look for the rightmost unpaired {\tt )}; if this does not exist, output~$\varnothing$.\footnote{We follow the usual meaning of paired and unpaired brackets by working ``inside-out''. 
Identify any adjacent {\tt ()}, remove them; these are declared to be paired.
Continue this process until no such adjacent pairs remain. Any brackets that remain are declared unpaired.}
Otherwise this {\tt)} is associated to some nonzero entry $(r,c)$ in $M$. Now
$f_i^{\sf row}(M)$ is the matrix obtained by subtracting $1$ from $M_{r,c}$ and adding $1$ to $M_{r+1,c}$. Similarly, $e^{\sf row}_i(M)$
is defined by looking at the leftmost unpaired {\tt (} , associated to some $M_{rc}>0$, and doing the replacements
\[M_{rc}\mapsto M_{rc}-1, M_{r-1,c}\mapsto M_{r-1,c}+1.\] 
Finally, 
\[f_j^{\sf col}(M):=(f_j^{\sf row}(M^t))^t \text{\ and  
$e_j^{\sf col}(M):=(e_j^{\sf row}(M^t))^t$}.\]

\begin{example}\label{ex:bicrystalops}
    Let $M = \left[\begin{smallmatrix}
        1 & 1 & 2\\
        2 & 3 & 1
    \end{smallmatrix}\right]$. $M$ has row word $1221222112$. We compute $f_1^{\sf row}(M)$ by first computing ${\sf bracket}_1({\sf row}(M))$, replacing each $1$ in ${\sf row}(M)$ with a {\tt )} and each $2$ with a {\tt (}. This bracket sequence is {\tt \textcolor{red}{)}(()((())(}; the rightmost unmatched {\tt )} is highlighted in red. Since the 
    {\tt \textcolor{red}{)}} comes from $M_{11}$, $f_1^{\sf row}(M) =  \left[\begin{smallmatrix}
        0 & 1 & 2\\
        3 & 3 & 1
    \end{smallmatrix}\right]$.
\end{example}

\begin{remark}\label{remark:moving} In Section~\ref{bigproofabc}, it will be convenient to refer to  $f_{i}^{\sf row}(M)$ as \emph{moving from $(i,j)$} if the effect is to subtract $1$ from
$M_{i,j}$ and add $1$ to $M_{i+1,j}$. For instance, in Example~\ref{ex:bicrystalops}, $f_1^{\sf row}(M)$ moves
from $(1,1)$. We say $e_{i}^{\sf row}(M)$ is \emph{moving to $(i,j)$} if the operator
acts by subtracting $1$ from
$M_{i+1,j}$ and adding $1$ to $M_{i,j}$. We use analogous language for $f_{j}^{\sf col}$ and $e_{j}^{\sf col}$. 
\end{remark}

\begin{definition}\label{def:admissibleops}
  For a Levi datum $({\bf I}, {\bf J})$, the set of \emph{admissible bicrystal operators} is
  \[\{e^{\sf row}_i,f^{\sf row}_i,e_j^{\sf col},f_j^{\sf col} : i\not\in {\bf I}, j\not\in {\bf J}\}.\]
\end{definition}

\begin{definition}
A set ${\mathcal S}\subseteq {\sf Mat}_{m,n}({\mathbb Z}_{\geq 0})$ of matrices is
$({\mathbf I},{\mathbf J})$-\emph{bicrystal closed}  if, for any admissible crystal operator $\varphi$ and any $M\in {\mathcal S}$, 
$\varphi(M)\in {\mathcal S}\cup \{\varnothing\}$.
\end{definition}

\begin{definition}[{\cite[Definition 1.9]{AAA}}]\label{def:bicrystalline}
A $L_{\bf I}\times L_{\bf J}$-stable ideal $I\subseteq \complexes[{\sf Mat}_{m, n}]$ is $({\bf I},{\bf J},\prec)$-\emph{bicrystalline} if there exists a term order $\prec$ such that the set 
\[(\monos_\prec I)^c=\{M\in{\sf Mat}_{m, n}(\integers_{\geq 0}):z^M\in\mathrm{Std}_{\prec}I\}\] 
of exponent matrices of standard monomials is $(\mathbf{I}, \mathbf{J})$-bicrystal closed.\end{definition}

A collection of admissible bicrystal operators defines a graph structure on ${\sf Mat}_{m, n}(\integers_{\geq 0})$, similar to the graph on words shown in Figure~\ref{fig:July14aaa}. A subset $\mathcal{S}\subseteq {\sf Mat}_{m, n}$ is $(\mathbf{I},\mathbf{J})$-bicrystal closed if and only if every connected component of this graph is contained entirely in $\mathcal{S}$ or its complement. In particular, a Levi-stable ideal $I\subseteq{\sf Mat}_{m, n}$ is bicrystalline if and only if the set-theoretic decomposition \eqref{eqn:combdecomp} makes sense as a decomposition of crystal graphs. Lemma~\ref{lemma:SandSc} and Proposition~\ref{prop:reformulate} below formalize these notions.

\begin{lemma}\label{lemma:SandSc}
${\mathcal S}\subseteq {\sf Mat}_{m,n}({\mathbb Z}_{\geq 0})$
 is $(\mathbf{I}, \mathbf{J})$-bicrystal closed if and only if its complement ${\mathcal S}^c$ is
 $(\mathbf{I}, \mathbf{J})$-bicrystal closed.
\end{lemma}
\begin{proof}
This is immediate because the $f_i$ and $e_i$ operators are essentially inverses of one another by definition. More precisely,
\[f_{i}^{\sf row}(M)\neq \varnothing \implies e_i^{\sf row}(f_i^{\sf row}(M))=M\]
and
\[e_{i}^{\sf row}(M)\neq \varnothing \implies f_i^{\sf row}(e_i^{\sf row}(M))=M,\]
with the same statements holding when ``${\sf row}$'' is replaced by ``${\sf col}$''.
\end{proof}

\begin{proposition}\label{prop:reformulate}
A $L_{\bf I}\times L_{\bf J}$-stable ideal $I$ is $({\mathbf I},{\mathbf J}, \prec)$-bicrystalline if and only if $\monos_{\prec} I$ is $({\mathbf I},{\mathbf J})$-bicrystal closed.
\end{proposition}
\begin{proof}
Immediate from combining Lemma~\ref{lemma:SandSc} with \eqref{eqn:stdmonos}. \end{proof}

We have shown that when $I$ is bicrystalline, both $\monos_\prec I$ and its complement inside ${\sf Mat}_{m, n}(\integers_{\geq0})$ admit an explicit crystal structure. In Section~\ref{sec:prelim} we recall how a crystal-theoretic decomposition of \eqref{eqn:combdecomp} reflects the representation-theoretic decomposition of \eqref{eqn:basicdecomp}, leading to our explicit combinatorial rules for the irreducible multiplicities of $\complexes[{\sf Mat}_{m, n}]/I$ and $I$ in Theorem~\ref{thm:LRrule}. Until then, we focus on determining whether or not a given Levi-stable ideal is bicrystalline.

\begin{example}[A non-bicrystalline ideal]\label{exa:not-bicrystalline}
Let 
\[I=\langle z_{11}z_{23}-z_{13}z_{21}\rangle \subset {\mathbb C}[{\sf Mat}_{2,3}].\] 
This ideal has a 
$GL_2\times T_3$ action. Now, 
\[z^{\left[\begin{smallmatrix} 0 & 1 & 1 \\ 1 & 0 & 0\end{smallmatrix}\right]}\in {\mathrm{init}}_{\antidiag}I=\langle z_{13}z_{21}\rangle\]  but 
\[f_1^{{\sf row}}\left(\left[\begin{matrix} 0 & 1 & 1 \\ 1 & 0 & 0\end{matrix}\right]\right)=
\left[\begin{matrix} 0 & 1 & 0 \\ 1 & 0 & 1\end{matrix}\right]\not\in
{\mathrm{init}}_{\antidiag}I,\]
and hence $I$ is not $(\{0,2\}, \{0,1,2,3\}, \antidiag)$-bicrystalline by Proposition~\ref{prop:reformulate}.
The other initial ideal of $I$ is 
\[{\mathrm{init}}_{\diag}(I)=\langle z_{11}z_{23}\rangle\] 
and $z^{\left[\begin{smallmatrix} 1 & 0 & 0 \\ 0 & 0 & 1\end{smallmatrix}\right]}$ witnesses that
it is not $(\{0,2\}, \{0,1,2,3\}, \diag)$-bicrystalline either.
\end{example}

\begin{example}[Another non-bicrystalline ideal]\label{ex:nonbicrystallineMRV}
    Let $$I = \left\langle z_{11}, z_{41},\begin{vmatrix}
        z_{11} & z_{12} & z_{13}\\
        z_{21} & z_{22} & z_{23}\\
        z_{31} & z_{32} & z_{33}
    \end{vmatrix},\begin{vmatrix}
        z_{21} & z_{22} & z_{23}\\
        z_{31} & z_{32} & z_{33}\\
        z_{41} & z_{42} & z_{43} 
    \end{vmatrix}\right\rangle\subset \mathbb{C}[{\sf Mat}_{4,4}].$$ This ideal carries a $(GL_1\times GL_2\times GL_1)\times (GL_1\times GL_2\times GL_1)$ action. The variety defined by this ideal is an example of a \textit{matrix Richardson variety} (see Example~\ref{exa:MRI} and the accompanying footnote). Using Algorithm~\ref{thebigalg}, one calculates that $I$ is not bicrystalline under any term order. Moreover, by contrast with the ideal of Example~\ref{exa:not-bicrystalline}, there is no pair of permutations $\sigma,\tau\in {\mathfrak S}_4$ such that the ideal $$I_{\sigma,\tau} = \left\langle z_{\sigma(1)\tau(1)},z_{\sigma(4)\tau(1)},\begin{vmatrix}
        z_{\sigma(1)\tau(1)} & z_{\sigma(1)\tau(2)} & z_{\sigma(1)\tau(3)}\\
        z_{\sigma(2)\tau(1)} & z_{\sigma(2)\tau(2)} & z_{\sigma(2)\tau(3)}\\
        z_{\sigma(3)\tau(1)} & z_{\sigma(3)\tau(2)} & z_{\sigma(3)\tau(3)}
    \end{vmatrix},\begin{vmatrix}
        z_{\sigma(2)\tau(1)} & z_{\sigma(2)\tau(2)} & z_{\sigma(2)\tau(3)}\\
        z_{\sigma(3)\tau(1)} & z_{\sigma(3)\tau(2)} & z_{\sigma(3)\tau(3)}\\
        z_{\sigma(4)\tau(1)} & z_{\sigma(4)\tau(2)} & z_{\sigma(4)\tau(3)}
    \end{vmatrix}\right\rangle$$ is bicrystalline for any non-torus Levi group $L_{\bf I}\times L_{\bf J}$ acting on $I_{\sigma,\tau}$.
\end{example}

In most examples of this paper, ideals are bicrystalline under $\antidiag$. Here is a naturally arising example where that is not the case:

\begin{example}[Bicrystalline only for $\diag$]\label{ex:diagbicrystal}  Consider the following ideal, which cuts out a ``fat point'' in ${\sf Mat}_{2, 2}$:
\begin{align*}
    I &=  \left\langle z_{11}^2, z_{11}z_{12}, z_{12}^2, z_{11}z_{21}, z_{21}^2, z_{21}z_{22}, z_{12}z_{22}, z_{22}^2, z_{11}z_{22} + z_{12}z_{21}\right\rangle\subseteq
    {\mathbb C}[{\sf Mat}_{2,2}].
\end{align*} 
In the notation of Definition~\ref{def:StevenSamideals}, $I$ is the $GL_2\times GL_2$-stable ideal $I_{\ydiags{2}}$. The generators of $I$ form a Gr{\"o}bner basis under any term order. Now
\[\left[\begin{matrix}
    1 & 1\\
    0 & 0
\end{matrix}\right]\in \monos_{\antidiag}(I) \text{ \ but \ } f_1^{\sf row}\left(\left[\begin{matrix}
    1 & 1\\
    0 & 0
\end{matrix}\right]\right) = \left[\begin{matrix}
    1 & 0\\
    0 & 1
\end{matrix}\right]\not\in \monos_{\antidiag}I,\]
thus witnessing that $I$ is not $(\{0,2\},\{0,2\}, \antidiag)$-bicrystalline. 

For fat point ideals, ${\mathrm{Std}}_{\prec}I$ is a finite set. One sees that $I$ is $(\{0,2\},\{0,2\},\diag)$-bicrystalline
from Definition~\ref{def:bicrystalline} by checking that the set of
six monomials that comprise ${\mathrm{Std}}_{\diag} I$ is $(\{0,2\},\{0,2\})$-bicrystal closed. 
\end{example}

\begin{example}[Degenerate case]
If $L_{\bf I}\times L_{\bf J}=T_m\times T_n$, i.e., $({\bf I},{\bf J})=(\{0,1,\dots,m\},\{0,1,\dots,n\})$ then, since there are no admissible operators, every $T_m\times T_n$-stable ideal is bicrystalline with respect to any $\prec$.
\end{example}

Deciding if $I$ is bicrystalline amounts, \emph{a priori}, to checking an infinite set of conditions. To address this decision
problem, we introduce \emph{test sets}. Theorem~\ref{thm:firstmain} 
shows that they provide a finite certificate for (non)bicrystallinity. In Section~\ref{sec:alg}, we describe an algorithm to construct test sets for arbitrary ideals, resolving the decision problem.

\begin{definition}[Test sets]\label{def:test}
Let $\varphi$ be an admissible bicrystal operator for $({\bf I},{\bf J})$. A finite set 
\[\mathcal{M}(I,\prec,\varphi)\subseteq \monos_\prec I\] 
of nonnegative integer matrices is a \emph{test set} for $(I,\prec,\varphi)$ if for every $M\in\monos_\prec I$ such that $\varphi(M)\neq \varnothing$,
there exists $N\in \mathcal{M}(I,\prec,\varphi)$ such that $\varphi(N)\neq\varnothing$, $z^N$ divides $z^M$, and $z^{\varphi(N)}$ divides $z^{\varphi(M)}$. 
\end{definition}

\begin{example}[Powers of the irrelevant ideal]\label{exa:irrelevantid}
The only $GL_m\times GL_n$ invariant monomial ideals $I\subseteq {\mathbb C}[{\sf Mat}_{m,n}]$ are 
$I={\mathfrak m}^d$ where 
\[{\mathfrak m}=\langle z_{ij}:1\leq i\leq m, \, 1\leq j\leq n\rangle\] 
is the irrelevant ideal.\footnote{Since for any generic matrix pair $(g, h)\in GL_m\times GL_n$ and degree-$d$ monomial $z^M$, $(g, h)\cdot z^M$ contains all monomials of degree $d$. Then
the proof of the $n=1$ case in \cite[Corollary~2.2]{Miller.Sturmfels} generalizes verbatim.} 
$I$ is minimally generated by the collection of degree $d$ monomials and, trivially, these generators form a Gr\"obner basis with respect to any term order $\prec$. 
Clearly, their exponent matrices form a test set ${\mathcal M}(I,\prec,\varphi)$ with respect to any admissible operator
\[\varphi\in \{f_i^{\sf row},e_i^{\sf row},f_j^{\sf col},e_j^{\sf col}:i\in [m-1],j\in [n-1]\}.\]
\end{example}

In general, the leading terms of a Gr\"obner basis do not form a test set:

\begin{example}[Space of singular matrices]\label{exa:singularmatrices}
 Let $I\subseteq {\mathbb C}[{\sf Mat}_{n,n}]$ be the principal ideal generated by the determinant $\det$ of the
 generic $n\times n$ matrix~$Z$. For each term ${\bf m} = z^M$ of $\det$, pick $\prec_{\mathbf{m}}$ so that 
 $\pm{\bf m}={\mathrm{init}}_{\prec_{\mathbf{m}}}(\det)$. If $\prec_{\mathbf{m}}$ is not $\antidiag$, there exist two consecutive rows $i,i+1$ such that the $1$'s of $M$ in these rows are placed northwest to southeast. It is then easy to check that 
$f_{i}^{\sf row}(M)\not\in\monos_{\prec_{\mathbf{m}}} I$. Thus $I$ is not bicrystalline for these $\prec_{\mathbf{m}}$ by Proposition~\ref{prop:reformulate}.

For $\antidiag$, ${\bf m}=z_{n1}z_{n-1,2}\cdots z_{1n}$. In this case, $\varphi(M) = \varnothing$ for any admissible $\varphi$. Thus $\{M\}$ does not form a test set for any $\varphi$, since, e.g., the monomial $z^N = z_{n1}z_{n-1,2}\cdots z_{2,n-1}z_{1n}^2$ lies in $\init_\prec I$ and $f_1^{\sf row}(N)\neq \varnothing$. Nevertheless, $I$ \emph{is} bicrystalline for $\antidiag$, and there is a finite test set to establish this (see Theorem~\ref{thm:generalizeddetideals}). 

To be fully concrete, let $n=2$. Then  
\[I=\left\langle\left|\begin{matrix} z_{11} & z_{12}\\ z_{21} & z_{22}\end{matrix}\right|\right\rangle\subseteq {\mathbb C}[{\sf Mat}_{2,2}].\] 
Explicit computation produces the following test set:
\begin{equation}\label{eqn:mintestset}
{\mathcal M}(I,\antidiag,f_1^{\sf row} )=\left\{ \begin{bmatrix}1 & 1\\ 1 & 0\end{bmatrix}, \begin{bmatrix}0 & 2\\ 1 & 0\end{bmatrix}\right\}.
\end{equation}
\end{example}

\begin{theorem}\label{thm:firstmain}
Fix a collection of test sets $\{{\mathcal M}(I,\prec,\varphi):\varphi \text{ \ is $({\bf I},{\bf J})$-admissible}\}$ for a $L_\mathbf{I}\times L_\mathbf{J}$-stable ideal $I$.
Then $I$ is $({\mathbf I}, {\mathbf J}, \prec)$-bicrystalline if and only if for every
$({\bf I},{\bf J})$-admissible $\varphi$,
\begin{equation}
\label{eqn:May25abc}
\varphi(N)\in \monos_{\prec}I\cup\{\varnothing\}, \text{ \ for every $N\in \mathcal{M}(I,\prec,\varphi)$.}
\end{equation}
\end{theorem}
\begin{proof} ($\Rightarrow$) If $I$ is $(\mathbf{I}, \mathbf{J}, \prec)$-bicrystalline, then by Proposition~\ref{prop:reformulate} the set $\monos_\prec I$ is $(\mathbf{I}, \mathbf{J})$-bicrystal closed. 
The condition \eqref{eqn:May25abc} follows immediately, since each test set $\mathcal{M}(I, \prec, \varphi)$ is a subset of $\monos_\prec I$ by definition.

\noindent
($\Leftarrow$) Conversely, suppose \eqref{eqn:May25abc} holds.  By Proposition~\ref{prop:reformulate}, it suffices to show that $\monos_{\prec}I$ is $({\bf I}, {\bf J})$-bicrystal closed. Let $M\in \monos_{\prec}I$ be given, and fix an admissible operator $\varphi$. If $\varphi(M)=\varnothing$ there is nothing to check, so we may assume
$\varphi(M)\neq \varnothing$. By Definition~\ref{def:test}, there exists 
\[N\in \mathcal{M}(I,\prec,\varphi) \subseteq \monos_{\prec}I\] 
such that $z^N|z^M$, $\varnothing\neq \varphi(N)$, and  $z^{\varphi(N)}|z^{\varphi(M)}$. Since we are assuming
$\varphi(N)\in \monos_{\prec}I$, we see that $\varphi(M)\in \monos_{\prec}I$ as desired. 
\end{proof}

\begin{theorem}[The minimal test set]\label{thm:minimality} Let ${\varphi}$ be an $({\bf I},{\bf J})$-admissible operator.
\begin{itemize}
\item[(I)] $\mathcal{M}(I,\prec,\varphi)$ is minimal (with respect to containment) if and only if 
for all $N\in \mathcal{M}(I,\prec,\varphi)$, $\varphi(N) \neq \varnothing$ and there does not exist any $N'\neq N$ in $\mathcal{M}(I,\prec,\varphi)$ such that $z^{N'} | z^N$ and $z^{\varphi(N')} | z^{\varphi(N)}$. 
\item[(II)] The collection of all test sets for a given $(I,\prec,\varphi)$, partially ordered by containment, contains a unique minimal element $\mathcal{M}_{\mathrm{min}}(I,\prec,\varphi)$.
\end{itemize}
\end{theorem}

\begin{proof}
(I): Let $\mathcal{M}(I,\prec,\varphi)$ be a minimal test set. If there exists an $N\in\mathcal{M}(I,\prec,\varphi)$ such that $\varphi(N) = \varnothing$, then $\mathcal{M}(I,\prec,\varphi)\smallsetminus\{N\}$ is, by definition, also a test set, a contradiction. Moreover, if there exists 
\[N,N'\in\mathcal{M}(I,\prec,\varphi) \text{\ such that $N'\neq N$, $z^{N'} | z^{N}$, and $z^{\varphi(N')}|z^{\varphi(N)}$},\] then for any 
\[M\in \monos_\prec I \text{ with $z^N | z^M$ and $z^{\varphi(N)}|z^{\varphi(M)}$},\] 
we have that $z^{N'} | z^M$ and $z^{\varphi(N')} | z^{\varphi(M)}$. So, $\mathcal{M}(I,\prec,\varphi)\smallsetminus \{N\}$ is a test set, a contradiction. Thus, $\mathcal{M}(I,\prec,\varphi)$ must have the desired property.   

Conversely, assume $\mathcal{M}(I,\prec,\varphi)$ has the stated property. Let 
\[\mathcal{M}'(I,\prec,\varphi) = \mathcal{M}(I,\prec,\varphi)\smallsetminus \{N\}\] for some $N\in \mathcal{M}(I,\prec,\varphi)$. There exists no 
\[N'\in \mathcal{M}(I,\prec,\varphi) \text{\ such that $z^{N'} | z^N$ and $z^{\varphi(N')}|z^{\varphi(N)}$}.\] 
Hence $\mathcal{M}'(I,\prec,\varphi)$ is not a test set. Thus, since $N$ was arbitrary, $\mathcal{M}(I,\prec,\varphi)$ is minimal.

(II):  Algorithm~\ref{thebigalg} and Theorem~\ref{thm:testsetalg}, stated and proved in Section~\ref{sec:alg}, establish that
at least one test set exists. If there is only one test set, we are done. Otherwise, to obtain a contradiction, suppose there are two different minimal test sets $\mathcal{M}(I,\prec,\varphi),\mathcal{M}'(I,\prec,\varphi)$. 
Then there exists 
$N\in \mathcal{M}(I,\prec,\varphi) \smallsetminus \mathcal{M}'(I,\prec,\varphi)$.
Since $\mathcal{M}'(I,\prec,\varphi)$ is a test set, there exists some 
    \begin{equation}
    \label{eqn:June22bcd}
    N'\in \mathcal{M}'(I,\prec,\varphi) \text{ with $z^{N'} | z^N$ and $z^{\varphi(N')} | z^{\varphi(N)}$}.
    \end{equation} 
    Since $N\not\in \mathcal{M}'(I,\prec,\varphi)$, 
    \begin{equation}\label{eqn:June22abc}
    N\neq N'.
    \end{equation} 

Likewise, there exists some
\begin{equation}\label{eqn:June22bcd2}
N''\in \mathcal{M}(I,\prec,\varphi) \text{ with $N'' | N'$ and $z^{\varphi(N'')} | z^{\varphi(N')}$}.
\end{equation} 
By \eqref{eqn:June22bcd} and \eqref{eqn:June22bcd2} combined, 
\begin{equation}
\label{eqn:June22bcd3}
z^{N''}|z^N  \text{ \ and $z^{\varphi(N'')}|z^{\varphi(N)}$.}
\end{equation}
Since we assumed $\mathcal{M}(I,\prec,\varphi)$ is minimal, 
by \eqref{eqn:June22bcd3} and (I),  we know $N'' = N$. Thus by \eqref{eqn:June22bcd2} we have $z^N | z^{N'}$ and,
by \eqref{eqn:June22bcd}, $z^{N'} | z^N$. Together, we see that $N' = N$, 
contradicting~\eqref{eqn:June22abc}.
\end{proof}

\begin{example}
Continuing Example~\ref{exa:singularmatrices}, let $\varphi=f^{\sf row}_1$. 
Any 
\[M=\begin{bmatrix} a & b\\ c &d\end{bmatrix} \in \monos_{\antidiag} I\] has $bc\neq 0$.
If we assume $\varphi(M)\neq \varnothing$ then either $a>0$ or ($a=0$ and $b>c$). In the first case, the associated test set element needed is $N_1=\left[\begin{smallmatrix} 1 & 1\\ 1 &0\end{smallmatrix}\right]$ and in the second case it is
$N_2=\left[\begin{smallmatrix} 0 & 2\\ 1 &0\end{smallmatrix}\right]$. Hence
\[{\mathcal M}_{{\mathrm{min}}}(I,\antidiag,f^{\sf row}_1)=\{N_1,N_2\},\] 
the test set from \eqref{eqn:mintestset}.
\end{example}

Under certain conditions, the union of two test sets is a test set. We record this fact for later use in Section~\ref{sec:in-RSK}.

\begin{proposition}\label{prop:gbuniontestsets}
    Let $J,K\subseteq\mathbb{C}[{\sf Mat}_{m,n}]$ be ideals. Fix a term order $\prec$ and admissible bicrystal operator $\varphi$, let $\mathcal{M}(J,\prec,\varphi), \mathcal{M}(K,\prec,\varphi)$ be test sets, and let $I := J+K$. If
    \[\init_\prec I = \init_\prec J+\init_\prec K,\] 
     then $\mathcal{M}(J,\prec,\varphi)\cup \mathcal{M}(K,\prec,\varphi)$ is a test set. Thus $I$ is $({\bf I}, {\bf J}, \prec)$-bicrystalline if $J$ and $K$ are.
\end{proposition}
\begin{proof}
    Let $M\in \monos_\prec I$. Then either $M\in \monos_\prec J$ or $M\in \monos_\prec K$. So, 
    assuming $\varphi(M)\neq \varnothing$,
    either there exists some $N\in\mathcal{M}(J,\prec,\varphi)$ with $\varphi(N)\neq\varnothing$, $z^N | z^M$, and $z^{\varphi(N)} | z^{\varphi(M)}$, or there exists some $N'\in\mathcal{M}(K,\prec,\varphi)$ with $\varphi(N')\neq\varnothing$, $z^{N'} | z^{M}$, and $z^{\varphi(N')} | z^{\varphi(M)}$. 
 Thus $\mathcal{M}(J,\prec,\varphi)\cup \mathcal{M}(K,\prec,\varphi)$ is a test set. 
 By Theorem~\ref{thm:firstmain} and the fact that $I = J+K$ is $L_{\bf I}\times L_{\bf J}$-stable,
 $I$ is $({\bf I}, {\bf J}, \prec)$-bicrystalline
  if $J$ and $K$ are.
\end{proof}

\section{The bicrystalline property is decidable}\label{sec:alg}

This section provides an effective algorithm to construct a test set for an arbitrary ideal $I$ and term order $\prec$. It allows us to decide if $I$ is bicrystalline for any given term order. 

\begin{theorem}[Bicrystalline algorithm]\label{thm:main}
Given generators $\mathcal{G} = \{g_1,\dots, g_r\}$ for an ideal $I \subseteq \mathbb{C}[{\sf Mat}_{m,n}]$ and Levi datum $(\mathbf{I}, \mathbf{J})$,
there exists a finite algorithm to decide whether there exists a term order $\prec$ such that 
$I$ is $(\mathbf{I}, \mathbf{J}, \prec)$-bicrystalline, or to decide whether $I$ is $(\mathbf{I}, \mathbf{J}, \prec)$-bicrystalline for a given term order $\prec$.
\end{theorem}

Our  algorithm and its proof of correctness provide a method
to give non-computational proofs that given ideals or families of ideals are bicrystalline.

\subsection{The algorithms}

We use three subroutines to prove Theorem~\ref{thm:main}. The first two are standard, while the third  (Theorem~\ref{thm:testsetalg}) is our main contribution.\footnote{Code is available at \url{https://github.com/LiberMagnum/grobnercrystals}.}

A Gr\"obner basis ${\mathcal G} = \{g_1, \dots, g_k\}$ for an ideal $I \subseteq \mathbb{C}[{\sf Mat}_{m,n}]$ is called \emph{reduced} with respect to a term order $\prec$ if:
\begin{itemize}
    \item[(i)] Each $g_i$ is monic, i.e., the coefficient of ${\mathrm{init}}_{\prec}(g_i)$ is $1$;
    \item[(ii)] No term of $g_i$ lies in $\langle {\mathrm{init}}_{\prec}(g_j) : j \neq i \rangle$.
\end{itemize}
The reduced Gr\"obner basis for $I$ with respect to $\prec$ is unique, and we denote it by ${\mathcal G}_{\mathrm{red}}$. It is computable from any generating set of $I$ using Buchberger's algorithm \cite[Section 1.3]{CLO}.

\noindent
\begin{myalgorithm}[Levi-stability]\label{stablealg} 
\gap \normalfont

\noindent 
\emph{Input:} Generators $\mathcal{G} = \{g_1,\dots, g_k\}$ for an ideal $I \subseteq \mathbb{C}[{\sf Mat}_{m,n}]$, and Levi datum $(\mathbf{I}, \mathbf{J})$ as in~\eqref{eqn:seqdef}.

\smallskip
\noindent
\emph{Output:} {\tt true} if $I$ is $(L_{\mathbf{I}} \times L_{\mathbf{J}})$-stable, and {\tt false} otherwise.

\medskip
\begin{enumerate}
\item[0.] Let $E^r_{ij},E^c_{ij}$ be the $m \times m$ (respectively, $n\times n$) elementary matrices with $1$'s on the diagonal and in position $(i,j)$, and $0$'s elsewhere.

\item[1.] Compute the reduced Gr\"obner basis ${\mathcal G}_{\mathrm{red}}$ for $I$ with respect to any term order.

\item[2.] If any $g\in\mathcal{G}_{\mathrm{red}}$ is not homogeneous with respect to the $\mathbb{Z}^{m+n}_{\geq 0}$-multigrading (defined in Section~\ref{subsection:bipre}) induced by the action of $T_m \times T_n \leq (L_{\mathbf{I}} \times L_{\mathbf{J}})$, output {\tt false}.

\item[3.] For each $g_\ell \in \mathcal{G}$ and each $E^r_{ij}\in L_{\bf I}$, compute $(E^r_{ij},\mathrm{Id}_n) \cdot g_\ell \mod {\mathcal G}_{\mathrm{red}}$ (where $\mathrm{Id}_n$ is the $n\times n$ identity matrix) using the division algorithm. For each $g_\ell \in \mathcal{G}$ and each $E^c_{ij}\in L_{\bf J}$, compute $(\mathrm{Id}_m,E^c_{ij}) \cdot g_\ell \mod {\mathcal G}_{\mathrm{red}}$ using the division algorithm. If any result is nonzero, output {\tt false}.

\item[4.] Output {\tt true}.
\end{enumerate}
\end{myalgorithm}

\begin{proposition}\label{lemma:stabilityalg}
Algorithm~\ref{stablealg} correctly decides if an ideal $I\subseteq {\mathbb C}[{\sf Mat}_{m, n}]$
is stable under the action of $L_\mathbf{I}\times L_\mathbf{J}$.
\end{proposition}
\begin{proof}
For any $f\in I$,
    \[f = \sum_{\ell=1}^k f_\ell g_{\ell}, \text{\ \ 
    where $f_{\ell}\in \mathbb{C}[{\sf Mat}_{m,n}]$.}\] 
    Now, 
    \begin{align*}
    (E^r_{ij},\mathrm{Id}_n)\cdot f\in I, \ \forall f\in I & \iff (E^r_{ij},\mathrm{Id}_n)\cdot (f_{\ell}g_{\ell})\in I \\
    \ &\iff ((E^r_{ij},\mathrm{Id}_n)\cdot f_{\ell})((E^r_{ij},\mathrm{Id}_n)\cdot g_{\ell})\in I.
    \end{align*}
    The analogous statements hold for $(\mathrm{Id}_m,E^c_{ij})$. Hence $I$ is closed under the action of $(E^r_{ij},\mathrm{Id}_n)$ (respectively, $(\mathrm{Id}_m,E^c_{ij})$) if and only if $(E^r_{ij},\mathrm{Id}_n)\cdot g_{\ell}\in I$ (respectively, $(\mathrm{Id}_m,E^c_{ij})\cdot g_{\ell}\in I$) for all generators of $I$. 

    The ideal $I$ is stable under the action of
   $T_m\times T_n \leq L_\mathbf{I}\times L_\mathbf{J}$
     if and only if it has a generating set $\tilde{\mathcal{G}}$ whose elements are homogeneous with respect to the $\integers^{m+n}_{\geq 0}$-multigrading on $\complexes[{\sf Mat}_{m, n}]$. Applying Buchberger's algorithm to $\tilde{\mathcal{G}}$ shows that $I$ is stable under the torus action if and only if its reduced Gr\"obner basis is homogeneous with respect to this multigrading. We may therefore check the torus-stability of $I$ by computing its reduced Gr\"obner basis from $\mathcal{G}$ and checking whether or not these generators are homogeneous.

Since Steps 2 and 3 of Algorithm~\ref{stablealg} check that $I$ is closed under the action of any invertible diagonal matrix and any pair of elementary matrices in $L_{\bf I}\times L_{\bf J}$, correctness follows since every element of $L_{\bf I}\times L_{\bf J}$ is a pair $(A,B)$ where $A,B$ are products of elementary matrices and an element of $T_m$ or $T_n$, respectively.
\end{proof}

  Although there are infinitely many term orders on $\C[{\sf Mat}_{m, n}]$, it is well-known that any particular ideal $I$ has only finitely many distinct initial ideals. An algorithm to traverse these initial ideals (i.e., traverse the \emph{Gr\"obner fan}) is implemented in 
  A.~Jensen's {\sf Gfan} software \cite{gfan}, based on the algorithms in the papers \cite{FukudaGFans} and \cite{MoraGFan}. That is, one has:

\begin{theorem}[\cite{FukudaGFans, MoraGFan}]\label{lemma:gfanalg}
    There is an algorithm that takes a set of generators $\mathcal{G} = \{g_1,\dots, g_r\}$ for an ideal $I$ as input and outputs generators for each of the finitely many initial ideals for $I$.
\end{theorem}

We now present our main new algorithm:

\begin{myalgorithm}[Test set algorithm]\label{thebigalg}
\gap \normalfont

\noindent 
\emph{Input:} Generators $\mathcal{G} = \{g_1,\dots, g_r\}$ for an ideal $I \subseteq \mathbb{C}[{\sf Mat}_{m,n}]$. A term order $\prec$. An admissible operator $\varphi$ as in Definition~\ref{def:admissibleops}. 

\smallskip
\noindent
\emph{Output:} A finite set of matrices ${\mathcal M}(I,\prec,\varphi)\subseteq {\sf Mat}_{m,n}(\mathbb{Z}_{\geq 0})$.

\medskip

\begin{enumerate}
\item[0.] Compute the reduced Gr\"obner basis ${\mathcal G}_{\mathrm{red}}$ for $I$ with respect to $\prec$.

\item[1.] For each $g \in {\mathcal G}_{\mathrm{red}}$, let $M(g)\in {\sf Mat}_{m,n}(\integers_{\geq 0})$ denote the exponent matrix of its initial term (assumed to have coefficient $1$ without loss of generality).

\item[2.] Define $\Sigma_g$ as follows:
\begin{itemize}
    \item If $\varphi \in \{f_i^{\sf row}, e_i^{\sf row}\}$, let $\Sigma_g$ be the sum of the entries in rows $i$ and $i+1$ of $M(g)$.
    \item If $\varphi \in\{f_j^{\sf col}, e_j^{\sf col}\}$, let $\Sigma_g$ be the sum of the entries in columns $j$ and $j+1$ of $M(g)$.
\end{itemize}

\item[3.] For each $g\in\mathcal{G}_{\sf red}$, initialize ${\mathcal C}_g = \emptyset$. For each integer $0 \leq d \leq \Sigma_g+1$:
\begin{itemize}
    \item If $\varphi \in \{f_i^{\sf row}, e_i^{\sf row}\}$, compute all weak compositions of $d$ into $2n$ parts. For each such composition $c$, form a matrix $A$ by placing the first $n$ parts of $c$ in row $i$ of $A$, the remaining $n$ parts in row $i+1$ of $A$, and $0$'s elsewhere. Add $A$ to ${\mathcal C}_g$.
    \item If $\varphi \in\{f_j^{\sf col},e_j^{\sf col}\}$, compute all weak compositions of $d$ into $2m$ parts. For each such composition $c$, form a matrix $A$ by placing the first $m$ parts of $c$ in column $j$ of $A$, the remaining $m$ parts in column $j+1$, and $0$'s elsewhere. Add $A$ to ${\mathcal C}_g$.
\end{itemize}

\item[4.] Initialize ${\mathcal M}(I,\prec,\varphi) = \emptyset$. For each $g \in {\mathcal G}_{\mathrm{red}}$ and each $A \in {\mathcal C}_g$, set
\[
{\mathcal M}(I,\prec,\varphi) := {\mathcal M}(I,\prec,\varphi) \cup \{ M(g) + A \}.
\]

\item[5.] Output ${\mathcal M}(I,\prec,\varphi)$.
\end{enumerate}
\end{myalgorithm}

\begin{theorem}\label{thm:testsetalg}
The output of Algorithm~\ref{thebigalg} is a test set for $(I,\prec,\varphi)$.
\end{theorem}

\begin{example}
If $I=\langle 0\rangle$ is the zero ideal, then ${\mathcal G}={\mathcal G}_{\sf red}=\emptyset$. Thus, ${\mathcal M}(I,\prec,\varphi)=\emptyset$ for any $\varphi$.
\end{example}

\begin{example}\label{ex:bigalgtight}
    Let $m = 2$, $n = 3$, and $$I = \left\langle g_1 = z_{13}^2, \; g_2 = z_{13}z_{23}, \; g_3 = z_{23}^2\right\rangle.$$ We apply Algorithm~\ref{thebigalg} with input $\mathcal{G} = \{g_1,g_2,g_3\}$, $\antidiag$, and $\varphi = e_1^{\sf row}$. Here 
    $\mathcal{G}={\mathcal G}_{\mathrm{red}}$ is already the reduced Gr{\"o}bner basis for $I$. Now, $$M(g_1) = \begin{bmatrix}
        0 & 0 & 2\\
        0 & 0 & 0
    \end{bmatrix}, \; M(g_2) = \begin{bmatrix}
        0 & 0 & 1\\
        0 & 0 & 1
    \end{bmatrix}, \; M(g_3) = \begin{bmatrix}
        0 & 0 & 0\\
        0 & 0 & 2
    \end{bmatrix}$$ 
    and 
    \[\Sigma_{g_1} = \Sigma_{g_2} = \Sigma_{g_3} = 2.\] 
    The output is a set of $196$ monomials:
    $$\mathcal{M}(I,\prec,\varphi) = \{M : \deg(z^M)\leq 5 \text{ \ and \ }
    z^{M(g_i)} | z^M \text{ \ for some $i$}\}.$$ 
    Now, let 
    \[M = \left[\begin{matrix}
        0 & 0 & 2\\
        1 & 2 & 0
    \end{matrix}\right] =M(g_1) + \left[\begin{matrix}
        0 & 0 & 0\\
        1 & 2 & 0
    \end{matrix}\right]\in \monos_{\prec}I.\] 
    There does not exist 
    $N\in \monos_{\prec}I$ such that $\deg(z^N) < 5$, $z^N | z^M$, and $z^{\varphi(N)} | z^{\varphi(M)}$. Thus, by Definition~\ref{def:test} any test set for $(I,\prec,\varphi)$ 
    contains this $M$.
    That is, Algorithm~\ref{thebigalg} does not produce test sets in general if the degree bound $1+\Sigma_{g}$
    in Step~3 is lowered. \end{example}

\begin{remark}\label{remark:nonminimality} 
The test sets generated by Algorithm~\ref{thebigalg} are usually non-minimal. For instance in 
Example~\ref{ex:bigalgtight}, 
\[\# {\mathcal M}_{{\mathrm{min}}}(I,\prec,\varphi)=11.\] 
By Theorem~\ref{thm:minimality}, the unique minimal test set may be computed by constructing any (possibly non-minimal) test set $\mathcal{M}(I,\prec,\varphi)$ and removing all $M\in \mathcal{M}(I,\prec,\varphi)$ for which either $\varphi(M) = \varnothing$ or there exists a different
$M'\in \mathcal{M}(I,\prec,\varphi)$ with $z^{M'} | z^{M}$ and $z^{\varphi(M')}|z^{\varphi(M)}$. 
\end{remark}

Before proving Theorem~\ref{thm:testsetalg}, we show that it implies Theorem~\ref{thm:main}.

\noindent
\emph{Proof of Theorem~\ref{thm:main}:}
	First, apply Algorithm~\ref{stablealg}  to determine whether or not $I$ is $L_\mathbf{I}\times L_\mathbf{J}$-stable. If not, we output {\tt false}. If it is, and no term order is given, apply the algorithm of Theorem~\ref{lemma:gfanalg} 
	to compute the finite set of all initial ideals for $I$. For each initial ideal $J$ of $I$, choose a term order $\prec$ such that 
	$J = \init_\prec I$. 
	For each term order and each bicrystal operator $\varphi$ associated to 
	$L_\mathbf{I}\times L_\mathbf{J}$, apply Algorithm~\ref{thebigalg} to construct test sets $\mathcal{M}(I,\prec,\varphi)$. (If a term order is given, one bypasses the application of Theorem~\ref{lemma:gfanalg}.) By Theorem~\ref{thm:firstmain}, $I$ is $(\mathbf{I}, \mathbf{J}, \prec)$-bicrystalline 
	if and only if for each admissible operator $\varphi$ and $M\in \mathcal{M}(I,\prec,\varphi)$,  $\varphi(M)\in\monos_\prec I\cup\{\varnothing\}$. Since there are only finitely many test sets, and each test set is finite by definition, 
	this property can be checked in finite time.\qed

\subsection{Proof of Theorem~\ref{thm:testsetalg}} \label{bigproofabc} 
We verify that $\mathcal{M}(I,\prec,\varphi)$ satisfies Definition~\ref{def:test}. Clearly,  $\#\mathcal{M}(I,\prec,\varphi)<\infty$. By Step 4 of Algorithm~\ref{thebigalg}, 
$\mathcal{M}(I,\prec,\varphi)\subseteq \monos_{\prec}I$. 
Next, suppose $M\in \monos_{\prec} I$ satisfies $\varphi(M)\neq \varnothing$. We must show there exists $N\in \mathcal{M}(I,\prec,\varphi)$ such that:
\begin{itemize}
\item[(T1)] $z^N$ divides $z^M$, and 
\item[(T2)] $z^{\varphi(N)}$ divides $z^{\varphi(M)}$ (with $\varphi(N)\neq\varnothing$).
\end{itemize}

First, suppose $\varphi=f_i^{\sf row}$ and $\varphi(M)$ moves from $(i,j)$ to $(i+1,j)$, as defined in
Remark~\ref{remark:moving}. Fix $g\in\mathcal{G}$ so that $z^{M(g)}|z^{M}$ (this can be done since $z^M\in\init_\prec I$).

To construct rows $i,i+1$ of $N$, we use two-row \emph{bracket tableaux}, defined by placing $M_{i,\ell}$ many~{\tt )} into the box $(i,\ell)$ and $M_{i+1,\ell}$ many {\tt (} into the box $(i+1,\ell)$. For 
instance,
\[
\begin{matrix}
i\\
i+1
\end{matrix} \ 
\begin{tabular}{|c|c|c|c|c|c|c|}
\hline
$0$ & $3$ & $1$\\ \hline
$3$ & $0$ & $0$\\ \hline
\end{tabular}
\mapsto
\begin{tabular}{|c|c|c|c|c|c|c|}
\hline
\ & {\tt )))} & {\tt )}\phantom{{\tt ))}}\\ \hline
{\tt (((} & \phantom{{\tt (((}} & \phantom{{\tt (((}} \\ \hline
\end{tabular}.
\]
Suppose the
bracket tableaux of $M$ and $M(g)$, respectively, are:
\[D_M=
  \begin{tabular}{|c|c|c|c|c|c|c|}
    \hline
    \ & {\tt )))} & {\tt )} \\ \hline
    {\tt (((} & \ & \phantom{{\tt (((}} \\ \hline
  \end{tabular} \text{\ \ and \ } D_{M(g)}=
  \begin{tabular}{|c|c|c|c|c|c|c|}
    \hline
    \ & {\tt )}\phantom{{\tt ((}} & {\tt )} \\ \hline
    {\tt ((}\phantom{\tt (} & \ & \phantom{{\tt (((}} \\ \hline
  \end{tabular}.
\]

We pause to give intuition for the following construction. 
It would be nice to take $N=M(g)$ since then (T1) is automatic, 
but (T2) usually fails for this choice of $N$. 
On the other hand, one can set $N=M$ and both (T1) and (T2) will hold. However, the tension is to
have (T1) and (T2) hold simultaneously under a fixed constraint on the total number of {\tt (} and {\tt )} added to $D_{M(g)}$.

Since $z^{M(g)} | z^M$, we embed $D_{M(g)}$ into $D_M$ 
by marking, using square brackets {\tt{[}} and {\tt{]}}, those parentheses in $D_M$ that also appear in $D_{M(g)}$:
\[D_M\mapsto D_M^\prime=
  \begin{tabular}{|c|c|c|c|c|c|c|}
    \hline
    \ & {\tt ]))} & {\tt ]} \\ \hline
    {\tt ([[} & \ & \phantom{{\tt (((}} \\ \hline
  \end{tabular}
\]
Our \emph{placement rule} for this embedding is that in row $i+1$, the {\tt [} are placed rightmost in each box, and in row $i$ the {\tt ]} are placed leftmost in each box, as done above.

Determine the positions of matched {\tt (} and {\tt )} in $D_M$; we mean that the bracket sequence is obtained by reading $D_M$ down columns from left to right and matchings are determined as usual (see Section~\ref{subsec:maindef}).
If a matched pair of brackets in $D_M'$
are a 
\begin{itemize}
\item {\tt (} matching with a {\tt ]}, turn that {\tt (} into a {\tt <};
\item  {\tt [} matching with a {\tt )}, turn that {\tt )} into a {\tt >}. 
\end{itemize}
Denote the resulting diagram by $D^{M(g)}_M$. Furthermore, let 
$D_N^\prime$ be the tableau obtained by deleting all {\tt (} and {\tt )} from $D^{M(g)}_M$, and lastly form $D_N$ from $D_N^\prime$ by turning
\begin{itemize}
\item all {\tt [} , {\tt <} to {\tt (}, and 
\item all {\tt ]} , {\tt >} to {\tt )}.
\end{itemize}
Continuing our example:
\[D^{M(g)}_M=
  \begin{tabular}{|c|c|c|c|c|c|c|}
    \hline
    \ & {\tt ]>)} & {\tt ]} \\ \hline
    {\tt ([[} & \ & \phantom{{\tt (((}} \\ \hline
  \end{tabular}
\mapsto
D^\prime_N=
  \begin{tabular}{|c|c|c|c|c|c|c|}
    \hline
    \ & {\tt ]>}\phantom{{\tt )}} & {\tt ]} \\ \hline
    {\tt [[}\phantom{{\tt (}} & \ & \phantom{{\tt (((}} \\ \hline
  \end{tabular} \mapsto D_N=
  \begin{tabular}{|c|c|c|c|c|c|c|}
    \hline
    \ & {\tt ))}\phantom{{\tt )}} & {\tt )} \\ \hline
    {\tt ((}\phantom{{\tt (}} & \ & \phantom{{\tt (((}} \\ \hline
  \end{tabular}.
\] 

There are two cases: either $M(g)_{ij} = M_{ij}$ or $M(g)_{ij} < M_{ij}$.

\noindent
\textit{Case 1:} ($M(g)_{ij} = M_{ij}$) Define $N$ by
\[
N_{k\ell} = \begin{cases}
    \#\text{$\{${\tt (,)}$\}$ in box }(k,\ell)\text{ of }D_N, & k\in \{i,i+1\}\\
    M(g)_{k\ell}, & k\not\in\{i,i+1\}
\end{cases}.\] 
\noindent
\textit{Case 2:} ($M(g)_{ij} < M_{ij}$) Define $N'$ by
$$N'_{k\ell} = \begin{cases}
    N_{k\ell}, & (k,\ell)\neq (i,j)\\
    N_{k\ell}+1, & (k,\ell) = (i,j)
\end{cases}.$$ 

\begin{claim}\label{claim:May20xyz}
$N$ and $N^\prime$ appear in the output ${\mathcal M}(I,\prec,\varphi)$ of Algorithm~\ref{thebigalg}.
\end{claim}
\noindent
\emph{Proof of Claim~\ref{claim:May20xyz}:} By construction, 
\[z^{M(g)} | z^{N} \text{ \ and $N = M(g) + A$,}\] 
where $A$ is some non-negative integer matrix that is $0$ outside of rows $i,i+1$. Moreover, since in $D_N^\prime$ each {\tt [}, {\tt ]} can be matched by at most one {\tt <}, {\tt >}, we have
\begin{align*} 
\deg(A) = \#\{\text{{\tt <}, {\tt >} in rows $i,i+1$ of $D_N^\prime$}\}
 \leq \Sigma_g  < \Sigma_g + 1.
\end{align*}

Also, $N' = N + A'$, where $$A' = \begin{cases}
    A_{k\ell}, & (k,\ell)\neq (i,j)\\
    A_{k\ell} + 1, & (k,\ell) = (i,j)
\end{cases}.$$ So, \[\deg(A') = \deg(A) + 1 \leq \Sigma_g + 1,\] 
as required.\qed

By construction, (T1) holds for $N$ and $N^\prime$, i.e., $z^N|z^M$ and $z^{N'}|z^M$. It remains to show:

\begin{claim}\label{claim:May20ccc} \emph{(T2)} holds for $N$ and $N^\prime$, i.e.,
$z^{\varphi(N)}|z^{\varphi(M)}$, $z^{\varphi(N')}|z^{\varphi(M)}$, and $\varphi(N),\varphi(N')\neq\varnothing$.
 \end{claim}
\noindent
\emph{Proof of Claim~\ref{claim:May20ccc}:} We show that both $f_i^{\sf row}(N)$ and $f_i^{\sf row}(N^\prime)$ 
move from $(i,j)$ to $(i+1,j)$. Given (T1), once we establish this assertion, (T2) is immediate.

We show (a) that there is a {\tt )} in box $(i,j)$ of $D_{N}$ (respectively $D_{N^\prime}$), and (b) that it corresponds to  the rightmost unmatched {\tt )} in ${\sf bracket}_i({\sf row}(N))$ (respectively ${\sf bracket}_i({\sf row}(N'))$).

\noindent
\emph{Case 1:} For (a), since \emph{Case 1} assumes $M(g)_{ij} = M_{ij}$ and $f_i^{\sf row}(M)$ moves from $(i, j)$, we see
\[M(g)_{ij}\geq 1.\] 
Thus, in ${D}_{M(g)}$ there is a {\tt )} in box $(i,j)$. This corresponds to some
{\tt ]} in $D_M^\prime$. Since no {\tt ]} are eliminated in the conversion from $D_M^\prime$ to 
$D_N^\prime$, that {\tt ]} appears in box $(i,j)$ of $D_N^\prime$. Hence, a {\tt )} appears in box $(i,j)$ of 
$D_N$, as desired.

For (b), we now show that the rightmost {\tt )} appearing in box $(i,j)$ of $D_N$ is the rightmost unmatched
{\tt )} in ${\sf bracket}_i({\sf row}(N))$. Since $M_{ij} = M(g)_{ij}$, this {\tt )} in $D_N$ came from a {\tt ]} in $D_N^\prime$ and corresponds to the rightmost {\tt )} in box $(i,j)$ of $D_M$. By assumption, it is this very {\tt )} in box $(i,j)$ of $D_M$ that is the rightmost unpaired {\tt )} in ${\sf bracket}_i({\sf row}(M))$. Therefore, this {\tt )} in its incarnation as {\tt ]} in $D_M^\prime$ and $D^{M(g)}_M$ remains the rightmost unpaired {\tt ]}. In the step  
\begin{equation}
\label{eqn:Sept5aaa}
D^{M(g)}_M\mapsto D_N^\prime
\end{equation}
where {\tt (}, {\tt )} are deleted, removing matched pairs of {\tt )} cannot destroy this rightmost unpaired property. Likewise, removing any unmatched {\tt )} cannot destroy this rightmost unpaired property, as any such {\tt )} must lie to the left of the rightmost bracket in box $(i,j)$. This completes the proof of Case 1.  

\noindent
\emph{Case 2:} Since $N'_{ij} = N_{ij} + 1$, there is at least one {\tt )} in box $(i,j)$ of $D_{N'}$ (which equals $D_N$ with an extra {\tt )} placed in box $(i,j)$), proving (a). 

For (b), we now show that the rightmost {\tt )} appearing in box $(i,j)$ of $D_{N'}$ is the rightmost unmatched
{\tt )} in ${\sf bracket}_i({\sf row}(N'))$. Every {\tt{(}} in $D_{N'}$ occupying a box $(i+1,j')$, $j' < j$, must be matched by some {\tt )} in $D_{N'}$ that existed in $D_N$. This is because, since there is an unmatched {\tt )} in box $(i,j)$ of $D_M$, every {\tt [} and {\tt <} in a box $(i+1,j')$, $j' < j$, of $D_M^{M(g)}$ must be matched by some {\tt ]} or {\tt >} in $D_M^{M(g)}$. Moreover, each {\tt ]}, {\tt >} matching these {\tt [}, {\tt <} must lie in a box $(i,j'')$, $j''\leq j$. Every such {\tt [} and {\tt <} remains matched in $D_N'$ by the arguments in \emph{Case 1} (the sentences about
\eqref{eqn:Sept5aaa}), so in $D_N$ every corresponding {\tt (} is matched. Adding a {\tt )} in box $(i,j)$ in our final conversion from $D_N$ to $D_{N'}$ cannot change any of these matchings, so there is at least one unmatched {\tt )} in box $(i,j)$ of $D_{N'}$. The same argument as in \emph{Case 1} shows that box $(i,j)$ indeed contains the rightmost unmatched {\tt )}, proving the claim.
\qed 

This concludes the proof of correctness when $\varphi = f_i^{\sf row}$. Similar proofs show correctness for the other bicrystal operators. For $\varphi = e_i^{\sf row}$, the construction of $N$ in \emph{Case 1} is identical to the construction for $f_{i}^{\sf row}$. In \emph{Case 2}, add $1$ to entry $(i+1,j)$ instead of $(i,j)$ to obtain $N'$. 
The remainder of the argument goes through by swapping ``right'' for ``left''.
The constructions for $f_j^{\sf col},e_j^{\sf col}$ are transpose to those for $f_i^{\sf row}, e_i^{\sf row}$.\qed

\section{Combinatorial Representation Theory Preliminaries}\label{sec:prelim} 

We now turn our attention towards Theorem~\ref{thm:LRrule}, where we reformulate the combinatorial rule given in~\cite{AAA} for computing the multiplicities in the irreducible decomposition of a quotient $\mathbb{C}[{\sf Mat}_{m,n}]/I$ as a Levi representation in terms of generalized Littlewood--Richardson tableaux. This section reviews some necessary material from combinatorial representation theory; more in-depth explanations may be found in~\cite[Sections 2-5]{AAA}.

\subsection{Combinatorial preliminaries}
We review some tableau combinatorics; we refer to~\cite[Section 3]{AAA}, and the references therein, for more details. 

Let $\lambda$ be an integer partition, identified with its Young diagram in English convention.
If $\lambda\subset \nu$ are two partitions, positioned so that their northwest corners agree, $\nu/\lambda$ is their \emph{skew shape} consisting of the boxes of $\nu$ with those of $\lambda$ removed.

\begin{definition}
A \emph{semistandard Young tableau} $T$ of \emph{shape} $\nu/\lambda$ is a filling of the boxes of $\nu/\lambda$ with positive integer entries, such that the entries both weakly increase along rows from left to right and strictly increase along columns from top to bottom.
\end{definition} 

\begin{definition}
    The \emph{length} of partition $\lambda$, denoted $\ell(\lambda)$, is the number of parts of $\lambda$.
\end{definition}

Let ${\mathrm{SSYT}}(\nu/\lambda)$ be the set of all semistandard Young tableaux of shape $\nu/\lambda$ 
and let ${\mathrm{SSYT}}(\nu/\lambda,n)$ be the subset consisting of those tableaux that use entries from $[n]$. 

\begin{definition}
    The \emph{row insertion} of $x$ into $T\in {\mathrm{SSYT}}(\lambda)$ is another semistandard tableau denoted $T\leftarrow x$. If no entry in the first row of $T$ exceeds $x$, form $T\leftarrow x$ by adding $x$ at the end of the first row of $T$. Otherwise, let $y$ be the leftmost entry in the first row of $T$ strictly greater than $x$. Replace this $y$ with $x$, then insert $y$ into the second row of $T$ in the same manner. The tableau produced when this process eventually terminates is $T\leftarrow x$.
\end{definition}

\begin{example}
    Let $T = \ytabs{
    1 & 2 & 3 & 3\\ 
    3 & 5}$ and let $x = 2$. 
    \begin{itemize}
    \item   Inserting $2$ into the first row of $T$ bumps out a $3$, yielding $T^{(1)} = \ytabs{1 & 2 & 2 & 3\\ 3 & 5}$. 
    \item Reinserting the displaced $3$ into the second row bumps out the $5$ to give\\ $T^{(2)} = \ytabs{1 & 2 & 2 & 3\\ 3 & 3 }$. 
    \item Reinserting this $5$ in the previously empty third row gives $(T\leftarrow x)=\ytabs{1 & 2 & 2 & 3\\ 3 & 3 \\ 5 }$.
    \end{itemize}
\end{example}

\begin{definition}\label{def:insertiontab}
    The \emph{insertion tableau} of a word $w = w_1w_2\dots w_k$ is the tableau 
    \[{\sf tab}(w) := (((\emptyset\leftarrow w_1)\leftarrow w_2)\leftarrow\dots\leftarrow w_k).\]
\end{definition}

\begin{definition}
    The \emph{RSK map} sends  $M\in {\sf Mat}_{m,n}({\mathbb Z}_{\geq 0})$ to
    \[{\sf RSK}(M) := ({\sf tab}({\sf row}(M)),{\sf tab}({\sf col}(M))),\]
    where ${\sf row}(M)$ and ${\sf col}(M)$ are as defined in Section~\ref{subsec:maindef}.
\end{definition}

\begin{theorem}[RSK Correspondence]\label{thm:RSK}
    The map ${\sf RSK}$ defines a bijection 
    \[{\sf Mat}_{m,n}({\mathbb Z}_{\geq 0})\longrightarrow\bigsqcup_{\lambda}
    {\mathrm{SSYT}}(\lambda,m) \times
    {\mathrm{SSYT}}(\lambda,n).\]
\end{theorem}

Our description of the {\sf RSK} correspondence is unorthodox in that it is not evident that ${\sf RSK}(M)=(P,Q)$ 
is a pair of tableaux of the same shape. For a standard description see, e.g., \cite[Section 4.1]{Fulton} or \cite[Section 7.11]{ECII}.\footnote{Our description also swaps the $P$- and $Q$-tableaux from the conventions in these sources, this merely being a matter of transposing $M$.}

\begin{definition}
    The (column) \textit{reading word} of a tableau $T$, ${\sf word}(T)$, is the list of entries of $T$ read along columns, bottom-to-top, left-to-right. Define ${\sf revword}(T)$ to be ${\sf word}(T)$ with the entries listed backwards.
\end{definition}

\begin{definition}\label{eqn:Knuthequivalence}
\emph{Knuth equivalence} on words is defined by the relations
\begin{equation}\label{eqn:Knuth1}
prq\equiv rpq \text{ \ \ if $p\leq q<r$}
\end{equation}
and 
\begin{equation}\label{eqn:Knuth2}
 qpr\equiv qrp \text{\ \ if $p<q\leq r$.}
 \end{equation}
 \end{definition}
 
The following is a fundamental fact about Knuth equivalence and {\sf RSK}:
\begin{theorem}[{\cite[Proposition 2.1.1 and Lemma 2.3.2]{Fulton}}]
\label{thm:June24bbb}
${\sf word}({\sf tab}(w))\equiv w$.  
\end{theorem}

\begin{example}
Let $w=23124$. Then 
\[{\sf tab}(w)=\ytabb{1 & 2 & 4\\ 2 & 3}, \text{ \ ${\sf word}({\sf tab}(w))=21324$,
and ${\sf revword}({\sf tab}(w))=42312$
.}\]
Indeed, we have $2\underline{132}4\equiv 23124$ by applying \eqref{eqn:Knuth1} to the underlined three letters,
in agreement with Theorem~\ref{thm:June24bbb}.
\end{example}

This, too, is one of the main results about Knuth equivalence and {\sf RSK}:

\begin{theorem}[{\cite[Theorem 2.1]{Fulton}}]
\label{thm:Knuthstraight}
In every Knuth equivalence class ${\mathcal K}$, there is a unique word that is ${\sf word}(T)$ for a straight shape tableau
$T$. Moreover, $T={\sf tab}(w)$ for any $w\in {\mathcal K}$. 
\end{theorem}

\begin{definition}[Tableau crystal operators \cite{KashNak}]\label{def:tabbicops}
    Let $T\in \operatorname{SSYT}(\lambda,n)$ and recall the crystal structure on words from Example~\ref{exa:wordcrystal}. If $f_i({\sf word}(T))=\varnothing$, define $f_i(T):=\varnothing$. Otherwise, $f_i$ changes a single $i$ in ${\sf word}(T)$ to an $i+1$. Define $f_i(T)\in\operatorname{SSYT}(\lambda, n)$ to be the tableau obtained by changing the corresponding $i$ in $T$ to an $i+1$. Define $e_i(T)\in\operatorname{SSYT}(\lambda, n)\cup\{\varnothing\}$ analogously using $e_i({\sf word}(T))$.
\end{definition}

\begin{example}
Figure~\ref{fig:crystal} depicts the crystal graph for $\operatorname{SSYT}(\ydiags{2,1},3)$. The edges show the effect of the two lowering operators $f_1$,$f_2$ (the raising
operators $e_1$, $e_2$ go in the opposite direction).

\begin{figure}[ht]
           \begin{tikzpicture}[scale = 0.63, align = center]
          \draw[red, thick][->] (1, 6) -- (4, 9);
        \draw[red, thick][->] (14, 9) -- (17, 6);
        \draw[blue, thick][->] (6, 10) -- (8, 10);
        \draw[blue, thick][->] (10, 10) -- (12, 10);
        \draw[blue, thick][->] (1, 4) -- (4, 1);
        \draw[blue, thick][->] (14, 1) -- (17, 4);
        \draw[red, thick][->] (6, 0) -- (8, 0);
        \draw[red, thick][->] (10, 0) -- (12, 0);
        \draw (0, 5) node{\ytabb{1 & 1\\ 2}};
        \draw (5, 10) node{\ytabb{1 & 2\\ 2}};
        \draw (9, 10) node{\ytabb{1 & 3\\ 2}};
        \draw (13, 10) node{\ytabb{1 & 3\\ 3}};
        \draw (5, 0) node{\ytabb{1 & 1\\ 3}};
        \draw (9, 0) node{\ytabb{1 & 2\\ 3}};
        \draw (13, 0) node{\ytabb{2 & 2\\ 3}};
        \draw (18, 5) node{\ytabb{2 & 3\\ 3}};
        \draw (2,8) node{$f_1$};
        \draw (2,2) node{$f_2$};
        \draw (7,11) node{$f_2$};
        \draw (11,11) node{$f_2$};
        \draw (16,8) node{$f_1$};
        \draw (7,-1) node{$f_1$};
        \draw (11,-1) node{$f_1$};
        \draw (16,2) node{$f_2$};
        \end{tikzpicture}
        \ytableausetup{boxsize=.3em}
        \caption{The crystal graph for $\operatorname{SSYT}(\vcenter{\hbox{\ydiagram{2,1}}},3)$. }\label{fig:crystal}
\end{figure}
For instance, to compute $f_2\left(\,\ytabs{ 1 & 2 \\ 2}\,\right)$, look at 
${\sf word}\left(\,\ytabs{ 1 & 2 \\ 2}\,\right)=212$.
The bracket sequence from reading the $2$'s and $3$'s is {\tt ))}. The rightmost unpaired {\tt )} corresponds to the rightmost $2$. 
Hence that $2$ turns into a $3$, producing $f_2\left(\,\ytabs{ 1 & 2 \\ 2}\, \right)=\ytabs{ 1 & 3 \\ 2}$ as shown in Figure~\ref{fig:crystal}.
\end{example}

The next proposition relates Definition~\ref{def:tabbicops}  and the  bicrystal operators from Section~\ref{subsec:maindef} through ${\sf RSK}$. The proof is from the definitions, although we omit it here. It is implicit in \cite{bicrystal1, bicrystal2}; see also \cite[Proposition~4.31]{AAA} for an explicit argument.

\begin{proposition}\label{prop:pullback}
    Let $M\in{\sf Mat}_{m,n}(\mathbb{Z}_{\geq 0})$, ${\sf RSK}(M) = (P,Q)$. Then 
    \begin{align*}
        {\sf RSK}(f_i^{\sf row}(M)) = (f_i(P),Q), &\quad {\sf RSK}(f_j^{\sf col}(M)) = (P,f_j(Q)),\\
        {\sf RSK}(e_i^{\sf row}(M)) = (e_i(P),Q), &\quad {\sf RSK}(e_j^{\sf col}(M)) = (P,e_j(Q)).
    \end{align*}
\end{proposition}

\subsection{Representation theory}\label{subsec:reptheoryprelims}
We recall facts about representation theory of general linear groups, referring to \cite{FH, Fulton, Howe} and our earlier work \cite[Section~2]{AAA} for more details. 

Let $V$ be a (finite-dimensional) vector space, viewed as an affine space with the right $GL(V)$-action
\[\vec{v}\cdot g := g\inv\vec{v}.\]
Let $\mathfrak{X}$ be an affine subscheme of $V$ stable under the action of some linear algebraic subgroup $G\subseteq GL(V)$. Equivalently, $V$ is the spectrum of its coordinate ring $\complexes[V]\cong \mathrm{Sym}(V^*)$, and $\mathfrak{X} = \mathrm{Spec}(\complexes[V]/I)$ for some ideal $I\subseteq\complexes[V]$. The right $G$-action on $V$ translates into a left $G$-action on $\complexes[V]$:
\[g\cdot f(\vec v):=f(\vec{v}\cdot g\inv) = f(g\vec{v}), \text{ \ for $g\in G, \vec v\in V, f\in {\mathbb C}[V]$.}\]
The subscheme $\mathfrak{X}\subseteq V$ is stable under the right $G$-action if and only if the corresponding ideal $I\subseteq\complexes[V]$ is stable under the left $G$-action. In this case, the $G$ action on $\complexes[V]$ descends to an action on the coordinate ring $\complexes[\mathfrak{X}] := \complexes[V]/I$. Since $\mathfrak{X}$ is not assumed to be an affine sub\emph{variety} of $V$, the ideal $I$ need not be radical.

Now restrict to our main case of interest by letting $U$ and $W$ be vector spaces of dimensions $m$ and $n$ respectively, and setting $V := U\boxtimes W$. We take $G$ to be the linear algebraic subgroup 
\[GL(U)\times GL(W)\hookrightarrow GL(V),\] 
embedded by mapping the pair $(g, h)\in GL(U)\times GL(W)$ to the Kronecker product $g\otimes h$. Identifying $V$ with ${\sf Mat}_{m, n}$, the $G$-action on $V$ above is precisely the $\mathbf{GL}$-action on ${\sf Mat}_{m, n}$ from equation \eqref{eqn:matrixaction}.\footnote{Usually, one would identify ${\sf Mat}_{m, n}$ with $\mathrm{Hom}(W, U)\cong W^*\otimes U$ rather than $U\otimes W$. Our identification gives a more natural correspondence with the combinatorics of ${\sf RSK}$.} By restriction, we may also take $G$ to be any Levi subgroup $L_{\bf I}\times L_{\bf J}$ of $GL(U)\times GL(V)$. Each $L_{\bf I}$ is a direct product of general linear groups $GL_{k}$, as described in Section~\ref{subsection:bipre}.

The irreducible polynomial representations $V_{\lambda}(k)$ of $GL_k$ are indexed by integer partitions $\lambda$ with at most $k$ parts, i.e., $\ell(\lambda)\leq k$. For concreteness, we give a classical description of $V_{\lambda}(k)$ in terms of minors; see \cite[Exercise~15.57]{FH} which credits J.~Deruyts and earlier A.~Clebsch.
Let $Z$ be the generic $k\times k$ matrix. For each semistandard tableau $T$ of shape $\lambda$ we associate a homogenous polynomial $[T]\in {\mathbb C}[{\sf Mat}_{k, k}]$ as follows: the $\lambda_j'$ entries of the $j$-th column of $T$ will be the
row indices of a minor $\Delta_j$ whose column indices are $1,2,\ldots,\lambda_j'$. Let 
\[[T]=\Delta_1\cdot \Delta_2\cdots\Delta_{\lambda_1}.\]
Define the vector space
\begin{equation}
\label{eqn:Vkdef}
V_{\lambda}(k) := \operatorname{span}_{\mathbb{C}} \left\{ [T] \;\middle|\; T \text{ is semistandard of shape } \lambda \right\} \subseteq \mathbb{C}[{\sf Mat}_{k, k}].
\end{equation}
This is a representation of $GL_k$ via linear substitution $Z\mapsto g\cdot Z$ (coordinate-wise), called a
\emph{Schur module}.

\begin{example}
$GL_2$'s second fundamental representation $V_{\ydiags{2}}(2)$ is the ${\mathbb C}$-span of 
\[\left[\, \ytabb{1 & 1}\,\right] =z_{11}^2, 
  \left[\,\ytabb{1 & 2}\,\right] =z_{11}z_{21}, 
   \left[\,\ytabb{2 & 2}\,\right] =z_{21}^2\]
   inside ${\mathbb C}[{\sf Mat}_{2,2}]$. 

   The determinant representation $V_{\ydiags{1,1}}(2)$ is the ${\mathbb C}$-span of $\left[\,\ytabs{1 \\ 2}\,\right] =\left|\begin{smallmatrix} z_{11} & z_{12}\\ z_{21} & z_{22}\end{smallmatrix}\right|$.
\end{example}

Each $[T]$ spans a one-dimensional irreducible sub-representation of the maximal torus $T_k\leq GL_k$;
$[T]$ is a weight vector spanning a weight space with weight
$\prod_{i\in T} x_i$. The \emph{character} of $V_{\lambda}(k)$ is the generating series over all these weights, namely, the \emph{Schur polynomial}
\begin{equation}\label{eqn:June24xyz}
s_{\lambda}(x_1,\ldots,x_k):=\sum_{T} x^T, \, x^T:=\prod_{i\in T}x_i,
\end{equation}
where the sum is over semistandard tableau $T$ of shape $\lambda$.

The irreducible polynomial representations of $GL_m\times GL_{n}$
are of the form $V_{\lambda}(m)\boxtimes V_{\mu}(n)$ where the action is by 
\[(g,g')\cdot v\otimes w=(g\cdot v)\otimes (g'\cdot w).\]

\begin{example}
The irreducible $GL_2\times GL_2$ representation $V_{\ydiags{2}}(2)\boxtimes V_{\ydiags{1,1}}(2)$ has a concrete description as the span of 
\[\left[\,\ytabs{1 & 1}\,\right]\otimes \left[\,\ytabs{1\\ 2 }\,\right]=z_{11}^2\cdot
\left|\begin{smallmatrix} {\widetilde z}_{11} & {\widetilde z}_{12}\\ {\widetilde z}_{21} & {\widetilde z}_{22}\end{smallmatrix}\right|, \, 
  \left[\,\ytabs{1 & 2} \,\right]\otimes \left[\,\ytabs{1\\ 2}\,\right] =z_{11}z_{21}\cdot
\left|\begin{smallmatrix} {\widetilde z}_{11} & {\widetilde z}_{12}\\ {\widetilde z}_{21} & {\widetilde z}_{22}\end{smallmatrix}\right|,\,  
   \left[\,\ytabs{2 & 2}\,\right]\otimes \left[\,\ytabs{1\\ 2}\,\right] =z_{21}^2\cdot
\left|\begin{smallmatrix} {\widetilde z}_{11} & {\widetilde z}_{12}\\ {\widetilde z}_{21} & {\widetilde z}_{22}\end{smallmatrix}\right|,\]

inside 
\[{\mathbb C}[{\sf Mat}_{2,2}]\boxtimes {\mathbb C}[{\sf Mat}_{2,2}]\cong {\mathbb C}\left[\begin{smallmatrix}
z_{11} & z_{12} \\ z_{21} & z_{22}\end{smallmatrix}\right]\boxtimes {\mathbb C}\left[\begin{smallmatrix}
{\widetilde z}_{11} & {\widetilde z}_{12} \\ {\widetilde z}_{21} & {\widetilde z}_{22}\end{smallmatrix}\right]\cong
{\mathbb C}[z_{11},z_{12},z_{21},z_{22}, {\widetilde z}_{11},{\widetilde z}_{12},{\widetilde z}_{21},{\widetilde z}_{22}]
.\]
\end{example}
\begin{remark}
Basis vectors for 
\[V_\lambda(m)\boxtimes V_\mu(n)\subseteq\complexes[{\sf Mat}_{m, m}]\boxtimes\complexes[{\sf Mat}_{n, n}]\] 
are indexed by tableau-pairs of shape $(\lambda, \mu)$. In the case where $\lambda = \mu$, there is also a copy of $V_\lambda(m)\boxtimes V_\lambda(n)$ embedded inside $\complexes[{\sf Mat}_{m, n}]$. In Section~\ref{sec:in-RSK} we recall a ``bitableau'' vector space basis for $\complexes[{\sf Mat}_{m, n}]$. Although these basis vectors are also indexed by tableau-pairs and defined using a generalization of \eqref{eqn:Vkdef}, we emphasize that the bitableaux of shape $\lambda$ do \emph{not} form a vector space basis for 
\[V_\lambda(m)\boxtimes V_\lambda(n)\subseteq\complexes[{\sf Mat}_{m, n}].\] 
See \cite[Example 11.8.5]{BCRV}.
\end{remark}
$V_{\lambda}(m)\boxtimes V_{\mu}(n)$ has linear basis of weight vectors $[T]\otimes [U]$. Hence, its character is 
\[s_{\lambda}(x_1,\ldots,x_m)s_{\mu}(y_1,\ldots,y_n).\]

Similarly, the irreducible $L_{\bf I}\times L_{\bf J}$-representations  are denoted
\[V_{\underline\lambda|\underline \mu}=V_{\underline\lambda}\boxtimes V_{\underline \mu}\]
where $\underline\lambda:=(\lambda^{(1)},\lambda^{(2)},\ldots, \lambda^{(r)})$, each $\lambda^{(t)}$ is an integer partition 
with $\ell(\lambda^{(t)})\leq i_t-i_{t-1}$, and 
\[V_{\underline\lambda}=V_{\lambda^{(1)}}(i_1-i_0)\boxtimes V_{\lambda^{(2)}}(i_2-i_1)\boxtimes \cdots
\boxtimes V_{\lambda^{(r)}}(i_r-i_{r-1})\]
is an irreducible $L_{\bf I}$-representation. 
Similarly, one defines $\underline\mu$ and $V_{\underline\mu}$ with respect to ${\bf J}$. Since $L_\mathbf{I}\times L_\mathbf{J}$ is a reductive group, any finite-dimensional polynomial representation $V$ of it admits a decomposition of the form
\begin{equation}\label{eqn:irreddecomp}
V\cong_{L_\mathbf{I}\times L_\mathbf{J}}\bigoplus_{\underline\lambda|\underline\mu}\left(V_{\underline \lambda}\boxtimes V_{\underline\mu}\right)^{\oplus c^V_{\underline\lambda|\underline\mu}},
\end{equation}
for some nonnegative integers $c^V_{\underline\lambda|\underline\mu}$. 
By Schur's lemma, this decomposition is unique up to isotypic components $\left(V_{\underline \lambda}\boxtimes V_{\underline\mu}\right)^{\oplus c^V_{\underline\lambda|\underline\mu}}$. Hence,
we define the \emph{irreducible multiplicities} of $V$:
\begin{equation}
\label{eqn:irredmultdef}
c^V_{\underline\lambda|\underline\mu}:=\dim_{{\mathbb C}}{\mathrm{Hom}}_{L_{\bf I}\times L_{\bf J}}\left(V_{\underline \lambda}\boxtimes V_{\underline\mu}, V\right)\in {\mathbb Z}_{\geq 0}.
\end{equation} 
These irreducible multiplicities also appear in the unique expression for the character of $V$ as a sum of products of Schur polynomials:
\begin{equation}\label{eqn:chardef}
\chi_V = \sum_{\underline\lambda|\underline\mu}c^V_{\underline\lambda|\underline\mu}s_{\underline\lambda}(x_1,\dots, x_m)s_{\underline\mu}(y_1,\dots,y_n),
\end{equation}
where $s_{\underline\lambda}(x_1,\dots, x_m)$ is the weight generating function (i.e., the character) for $V_{\underline\lambda}$. 

When an ideal $I\subseteq\complexes[{\sf Mat}_{m, n}]$ is $L_\mathbf{I}\times L_\mathbf{J}$-stable, both $I$ and the coordinate ring $\complexes[\mathfrak{X}] = \complexes[{\sf Mat}_{m, n}]/I$ are polynomial $L_\mathbf{I}\times L_\mathbf{J}$-representations. 
Indeed, we have the representation-theoretic decomposition 
\[\complexes[{\sf Mat}_{m, n}]\cong_{L_\mathbf{I}\times L_\mathbf{J}} I\oplus \complexes[\mathfrak{X}]\] 
previously mentioned in \eqref{eqn:basicdecomp}. 
Although these representations are not technically finite-dimensional, they are a direct sum of finite-dimensional graded components and therefore still admit decompositions of the form \eqref{eqn:irreddecomp}. It is then natural to seek combinatorial rules for the multiplicities $c^I_{\underline\lambda|\underline\mu}$ or $c^{\complexes[\mathfrak{X}]}_{\underline\lambda|\underline\mu}$, which are equivalent to rules for the types of character formulas given in Section~\ref{subsec:motivating}.
Our rule,
Theorem~\ref{thm:LRrule}, provides a common generalization of two important settings:

\begin{example}[Characters and Hilbert series]\label{exa:char=Hilb}
If 
\[L_{\bf I}\times L_{\bf J}=T_m\times T_n\] 
is the maximal torus, each $V_{\underline\lambda|\underline\mu}$ is one-dimensional and each $\lambda^{(i)},\mu^{(j)}$ is a partition with at most one part, i.e., a nonnegative integer. Each $V_{\underline\lambda|\underline\mu}$ is spanned by a standard basis vector ${\bf m} = z^M\in {\mathrm{Std}}_{\prec} I$ such that the entries of $M$ in the $i$-th row sum to $\lambda^{(i)}$ and the entries in the $j$-th column sum to $\mu^{(j)}$. In representation-theoretic terms, this means that ${\bf m}$ is a weight vector with weight 
$\underline\lambda|\underline\mu$ spanning the weight space $V_{\underline\lambda|\underline\mu}$.
The character of $\complexes[\mathfrak{X}]$ is then the formal power series
\[\chi_{{\mathbb C}[{\mathfrak X}]}=\sum_{\underline\lambda|\underline\mu}c^{\complexes[\mathfrak{X}]}_{\underline\lambda|\underline\mu} x^{\underline\lambda}y^{\underline\mu},\]
where 
\[x^{\underline\lambda}=x_1^{\lambda^{(1)}}x_2^{\lambda^{(2)}}\cdots x_m^{\lambda^{(m)}} \text{\ and \  $y^{\underline\lambda}=y_1^{\lambda^{(1)}}y_2^{\lambda^{(2)}}\cdots y_n^{\lambda^{(n)}}$.}\]
Now, $\chi_{{\mathbb C}[{\mathfrak X}]}$ is the generating series for standard monomials with respect
to the multigrading induced by $T_m\times T_n$. That is, $\chi_{{\mathbb C}[{\mathfrak X}]}$ is, by definition, the \emph{multigraded Hilbert series} of ${\mathbb C}[{\mathfrak X}]$ (see \cite[Definition 8.14]{Miller.Sturmfels}), and the multiplicities \eqref{eqn:irredmultdef} are the values of its
(multigraded) Hilbert function.
\end{example}

\begin{example}[Littlewood--Richardson coefficients]\label{exa:Levi-branch}
Let ${\bf L}=(GL_k\times GL_{m-k})\times GL_n$ act on ${\mathfrak X}={\sf Mat}_{m,n}$ (corresponding to the zero ideal $I=\langle 0\rangle$). The ${\bf L}$-irreducible representations of $\complexes[{\sf Mat}_{m, n}]$ are of the form $V_{\lambda^{(1)}}(k)\boxtimes V_{\lambda^{(2)}}(m-k)\boxtimes V_{\mu}(n)$, and a standard branching formula \cite[Equation 5.7.2.1]{Howe} 
for decomposing 
a $GL_m$ representation into a $GL_k\times GL_{m-k}$ representation shows that 
\[c^{\complexes[{\sf Mat}_{m, n}]}_{\lambda^{(1)},\lambda^{(2)}|\mu}=c_{\lambda^{(1)},\lambda^{(2)}}^{\mu},\] the \emph{Littlewood--Richardson coefficient}.
It is in this sense that our rule for $c^{\complexes[\mathfrak{X}]}_{\underline\lambda|\underline\mu}$, namely Theorem~\ref{thm:LRrule}, is a generalized Littlewood--Richardson rule.
\end{example}

\subsection{Highest weight matrices and tableaux}\label{subsec:highestweight} Crystal graphs, such as those previously described on words, tableaux, and nonnegative integer matrices, 
have several properties that make them useful for computing irreducible representation multiplicities. In any $GL_k$-crystal graph (or in any $L_\mathbf{I}\times L_\mathbf{J}$-crystal graph), each vertex has a \emph{weight}, and each connected component contains a unique \emph{highest weight vertex} (a source vertex in the directed graph). The weight generating function for the vertices in a connected crystal graph with highest weight $\lambda$ is exactly the Schur polynomial $s_\lambda(x_1,\dots, x_k)$. Enumerating highest weight vertices in a crystal graph for some representation $V$ thus expresses the character $\chi_V$ in the form of (\ref{eqn:chardef}), yielding formulas for the irreducible multiplicities $c^V_{\underline\lambda|\underline\mu}$. For references and a more thorough exposition of the above facts, see \cite[Section 4]{AAA}.

In what follows, let $({\bf I},{\bf J})$ be a Levi datum and 
assume ${\sf RSK}(M)=(P,Q)$.

\begin{definition} A nonnegative integer matrix $M\in {\sf Mat}_{m,n}({\mathbb Z}_{\geq 0})$ is \emph{$({\bf I},{\bf J})$-highest weight} if, for every admissible raising operator $\varphi\in\{e_i^{\sf row},e_j^{\sf col}\}$, $\varphi(M) = \varnothing$. 
\end{definition}

\begin{definition}\label{def:latword}
   Fix integers $a\leq b$. A word 
    \[w = w_1\ldots w_k, \; w_\ell\in [a,b]\] 
    is  \textit{$[a,b]$-ballot} if for every $1\leq \ell\leq k$ and every $i\in [a,b-1]$, $i$ occurs in the initial segment $w_1\ldots w_{\ell}$ at least as many times as $i+1$ does. 
\end{definition}

Given a word $w$, let $w|_{[a,b]}$ be the subword that uses only the letters from $[a,b]$.

\begin{definition}\label{def:LRtab}
    Fix integers $a\leq b$. $T\in \operatorname{SSYT}(\nu/\lambda)$ is an $[a,b]$-\textit{Littlewood--Richardson (LR) tableau} if ${\sf revword}(T)|_{[a,b]}$ is $[a,b]$-ballot.
\end{definition}

\begin{definition}
    Let $T\in \operatorname{SSYT}(\nu/\lambda,\ell)$ and let 
    \[{\bf K} = \{0=k_0 < \ldots < k_t = \ell\}.\] 
    We say $T$ is ${\bf K}$-LR if for each $0 < \alpha \leq t$, $T$ is a $[k_{\alpha-1}+1,k_{\alpha}]$-LR tableau.
\end{definition}

\begin{definition}\label{def:supersemistandard}
    Fix a partition $\lambda$ and interval $[a, a']$ with $\ell(\lambda)\leq a'-a+1$. The \emph{supersemistandard tableau} of shape $\lambda$ on $[a, a']$ is the tableau $T_\lambda[a, a']\in\operatorname{SSYT}(\lambda)$ such that each row $i$ is filled entirely with the value $a+i-1$. The supersemistandard tableau on $[1, m]$ is simply denoted $T_\lambda$.
\end{definition}

\begin{remark}\label{rem:supersemistandard}
    For $\mathbf{I} = \{0, m\}$, $T_\lambda$ is the unique $\mathbf{I}$-LR tableau of shape $\lambda$ when it exists. More generally, for an arbitrary $\mathbf{I}$, any $\mathbf{I}$-LR tableau $P$ satisfies that $P|_{[1,i_1]} = T_\lambda[1,i_1]$.
\end{remark}

\begin{example}\label{exa:June24ggg}
Let $m=n=3$, ${\bf I}={\bf J}=\{0,2,3\}$, and 
\[M^{(1)}=\begin{bmatrix}
0 & 1 & 1 \\
1 & 0 & 0 \\
1 & 0 & 0
\end{bmatrix} \text{ \ and  \ }
M^{(2)}=\begin{bmatrix}
0 & 0 & 1 \\
1 & 1 & 0 \\
1 & 0 & 0
\end{bmatrix}.
\]
We have
\[{\sf RSK}(M^{(1)})=(P^{(1)},Q^{(1)})=\left(\, \ytabb{ 1 & 1\\ 2 & 3 }, 
\ytabb{ 1 & 1\\ 2 & 3 }\, \right)\]
and
\[{\sf RSK}(M^{(2)})=(P^{(2)},Q^{(2)})=\left(\,\ytabb{ 1 & 2\\ 2 \\ 3 }, 
\ytabb{ 1 & 1\\ 2 \\ 3 }\, \right).
\] 
$P^{(1)}$ and $Q^{(1)}$  are respectively 
${\bf I}$-LR and ${\bf J}$-LR.
Although $Q^{(2)}$ is ${\bf J}$-LR, $P^{(2)}$ is not ${\bf I}$-LR.
This means that $M^{(1)}$ is $({\bf I},{\bf J})$-highest weight (the admissible operators 
being $e_1^{\sf row}$ and $e_1^{\sf col}$). However, $M^{(2)}$ is not  $({\bf I},{\bf J})$-highest weight since
\[e_1^{\sf row}(M^{(2)})=\begin{bmatrix}
1 & 0 & 1 \\
0 & 1 & 0 \\
1 & 0 & 0
\end{bmatrix}\neq\varnothing. 
\]
These calculations illustrate Proposition~\ref{prop:LRtab} below.
\end{example}

\begin{proposition}\label{prop:LRtab}
Let ${\sf RSK}(M) = (P, Q)$. Then $M$ is $({\bf I},{\bf J})$-highest weight if and only if $P$ is ${\bf I}$-LR and $Q$ is ${\bf J}$-LR.
\end{proposition}
\begin{proof}
    $M$ is $({\bf I},{\bf J})$-highest weight if and only if for all admissible $i,j$, 
    \[e_i^{\sf row}(M) = \varnothing \text{\ and \ } e_j^{\sf col}(M) = \varnothing.\] 
   By Proposition~\ref{prop:pullback}, $M$ is $({\bf I},{\bf J})$-highest weight if and only if for all admissible $i,j$, 
    \[e_i(P) = \varnothing \text{\ and \ }  e_j(Q) = \varnothing.\]

    Fix any $i$ such that $e_i^{\sf row}$ is admissible. Let ${\sf revword}(P) = w_1\ldots w_k$ and let ${\sf word}(P) = w'_1\ldots w'_k$.

    If $e_i(P) = \varnothing$, then every $i+1$ in ${\sf word}(P)$ must be matched with some $i$ to its right. So, in ${\sf revword}(P)$, every $i+1$ is preceded by the $i$ matched with it in ${\sf word}(P)$. Since every $i$ can match at most one $i+1$, for any $\ell$, the number of $i$'s in $w_1\ldots w_\ell$ is at least the number of $(i+1)$'s. Thus, ${\sf revword}(P)$ is ballot for $[i,i+1]$. Since ${\sf revword}(P)$ is $[i,i+1]$-ballot for every admissible $i$, $P$ is ${\bf I}$-LR.

    Conversely, assume ${\sf revword}(P)$ is ballot for $[i,i+1]$. If $e_i(P)\neq\varnothing$, then ${\sf word}(P)$ contains some $i+1$ that is not matched with any $i$ to its right. So, every $i$ to the right of this $i+1$ is matched with some different $i+1$, implying that there is some $\ell$ for which $w'_\ell w'_{\ell+1}\ldots w'_k$ contains more $i+1$s than $i$s. This is impossible, since ${\sf revword}(P)$ was assumed to be ballot. So, $e_i(P)\neq\varnothing$ for each admissible $i$, and thus $P$ is highest weight.

    Identical arguments hold for $Q$, proving the claim.
\end{proof}

\section{The ballot rule for the irreducible multiplicities}\label{sec:LR}

\subsection{Statement of the rule; examples}\label{subsec:LRreformulation} 
We give a new combinatorial rule, Theorem~\ref{thm:LRrule}, for the multiplicities
$c^I_{\underline\lambda|\underline\mu}$ or $c^{R/I}_{\underline\lambda|\underline\mu}$ when $I\subseteq R = \mathbb{C}[{\sf Mat}_{m,n}]$ is $(\mathbf{I},\mathbf{J}, \prec)$-bicrystalline. This new rule reformulates and extends the rule given in~\cite[Main Theorem 1.11]{AAA}. Also, it is stated in a more explicit form, in
terms of generalized Littlewood--Richardson tableaux.

\begin{definition} Let $T$ be a semistandard tableau using the entries $[a,b]$ where $1\leq a\leq b$ are integers.
The $[a,b]$-\emph{content} of $T$ is an integer composition $\mu=(\mu_1,\mu_2,\ldots,\mu_{b-a+1})$ 
where $\mu_i$ equals the number of entries of $T$ equal to $a+i-1$, for $1\leq i\leq b-a+1$.
\end{definition}

\begin{definition}
   A tableau-pair $(P,Q)$ is \textit{$({\bf I},{\bf J})$-LR of content $(\underline{\lambda},\underline{\mu})$} if 
    \begin{itemize}
    \item $P$ is ${\bf I}$-LR and has $[i_{\alpha-1}+1,i_{\alpha}]$-content $\lambda^{(\alpha)}$ for 
    $1\leq \alpha \leq r$, and
    \item $Q$ is ${\bf J}$-LR and has $[j_{\beta-1}+1,j_{\beta}]$-content $\mu^{(\beta)}$ for
    for each $1\leq \beta \leq s$.
    \end{itemize}
    Let ${\mathcal{LR}}\left({\bf I}, {\bf J}, \underline\lambda, \underline\mu\right)$ be the set of these pairs of tableaux.
    Let ${\mathcal{LR}}\left({\bf I},\underline\lambda\right)$ and  
    ${\mathcal{LR}}\left({\bf J},\underline\mu\right)$ respectively denote the sets of $P$ and $Q$ tableaux of the above kinds.
\end{definition}

\begin{example}
Let $(P^{(1)},Q^{(1)})$ be as in Example~\ref{exa:June24ggg}. This pair is $(\{0,2,3\},\{0,2,3\})$-LR of content $\left(\left(\ydiags{2,1},\ydiags{1}\right), \left(\ydiags{2,1},\ydiags{1}\right)\right)$.
\end{example}

\begin{theorem}[The multiplicity rule]\label{thm:LRrule}
If $I\subseteq R = {\mathbb C}[{\sf Mat}_{m,n}]$ is an $({\bf I},{\bf J}, \prec)$-bicrystalline ideal, then
\begin{equation}\label{eqn:therule}
c^{R/I}_{\underline\lambda|\underline\mu}=\#\left\{M\not\in \monos_\prec I : {\sf RSK}(M)\in 
{\mathcal{LR}}\left({\bf I}, {\bf J}, \underline\lambda, \underline\mu\right)\right\},
\end{equation}
and
\begin{equation}\label{eqn:therule2}
c^{I}_{\underline\lambda|\underline\mu}=\#\left\{M\in\! \monos_\prec I : {\sf RSK}(M)\in 
{\mathcal{LR}}\left({\bf I}, {\bf J}, \underline\lambda, \underline\mu\right)\right\}.
\end{equation}
\end{theorem}

\begin{remark}
    We provide some intuition for Theorem~\ref{thm:LRrule}. Section~\ref{subsection:bipre} opens with the observation \eqref{eqn:basicdecomp} that for any ideal $I\subseteq R = \complexes[{\sf Mat}_{m, n}]$, $R\cong I\oplus R/I$ as vector spaces. Via Gr\"obner theory, this vector space decomposition is related to the set-theoretic decomposition \eqref{eqn:combdecomp}: 
    \[{\sf Mat}_{m, n}(\integers_{\geq 0}) = \monos_\prec I \sqcup (\monos_\prec I)^c.\] 
    Theorem~\ref{thm:LRrule} should be viewed as a Levi-equivariant upgrade of the connection between \eqref{eqn:basicdecomp} and \eqref{eqn:combdecomp}. When $I$ is Levi-stable, \eqref{eqn:basicdecomp} holds as an isomorphism of representations. When $I$ is bicrystalline, \eqref{eqn:combdecomp} holds as an isomorphism of crystal graphs, and Theorem~\ref{thm:LRrule} uses these crystal graphs to read off the irreducible representations appearing in \eqref{eqn:basicdecomp}.
\end{remark}

\begin{example}[Graphical matroids]\label{exa:graphicalmatroid}
Example~\ref{exa:matroid} discussed matrix matroid ideals~\cite{FNR}. 
A source of such ideals comes from graphical matroids. 
Given a finite simple graph $G=(V,E)$  with vertices $V=[m]$, the 
\emph{graphical matroid} associated to $G$ is the collection of vectors 
\[\{\vec e_i-\vec e_j: \{i,j\} \in E \text{\ and $i<j$}\}\subseteq {\mathbb R}^m.\]
 
Let $I$ be the ideal of the matrix graphical matroid variety $\mathfrak{X}_G$ associated to the graph $G$ below; the variety is the $GL_4\times T_5$ orbit of the associated matrix $M_G$ whose columns are labelled, left to right, $a,b,c,d,e$: 
\[G=\vcenter{\hbox{\begin{tikzpicture}[scale=2]
  \node[circle, draw, inner sep=1pt, minimum size=0.4cm] (1) at (0,1) {$1$};
  \node[circle, draw, inner sep=1pt, minimum size=0.4cm] (2) at (1,1) {$2$};
  \node[circle, draw, inner sep=1pt, minimum size=0.4cm] (4) at (1,0) {$4$};
  \node[circle, draw, inner sep=1pt, minimum size=0.4cm] (3) at (0,0) {$3$};

  \draw (1) -- (2) node[midway, above] {\small $b$};
  \draw (1) -- (3) node[midway, left] {\small $a$};
  \draw (2) -- (3) node[midway, sloped, above] {\small $c$};
  \draw (2) -- (4) node[midway, right] {\small $e$};
  \draw (3) -- (4) node[midway, below] {\small $d$};
\end{tikzpicture}}}, \ \ 
M_G=\begin{bmatrix}
1 & 1 & 0 & 0 & 0\\
0 & -1 & 1 & 0 & 1 \\
-1 & 0 & -1 & 1 & 0\\
0 & 0 & 0 & -1 & -1
\end{bmatrix}
\]
$I$ has a Gr\"obner basis under $\antidiag$ given by
eight cubics and a quartic, with
\begin{multline}\nonumber
{\mathrm{init}}_{\antidiag} I=\langle z_{13}z_{22}z_{31},\, z_{23}z_{32}z_{41}, \, z_{25}z_{34}z_{43}, \, z_{13}z_{32}z_{41}, \,
 z_{13}z_{22}z_{41},\, z_{15}z_{34}z_{43},\, z_{15}z_{24}z_{43},\,\\
  z_{15}z_{24}z_{33},\,
 z_{15}z_{24}z_{32}z_{41}\rangle.
 \end{multline}
 Algorithm~\ref{thebigalg} verifies that $I$ is $(\{0,4\},\{0,1,2,3,4,5\}, \antidiag)$-bicrystalline.

We compute $c^{\complexes[\mathfrak{X}_G]}_{\ydiags{2,2,2}|\ydiags{2},\ydiags{1},\ydiags{1},\ydiags{1},\ydiags{1}}$ using Theorem~\ref{thm:LRrule}. We find 
semistandard tableaux $Q$ of shape $\vcenter{\hbox{\ydiags{2,2,2}}}$ with content $(2, 1, 1, 1, 1)$ such that
\[{{\sf RSK}^{-1}(T_{\ydiags{2,2,2}},Q)} = M\] 
for some $z^M\in {\mathrm{Std}}_{\antidiag}I$ (because ${\bf I}=\{0,4\}$ and the only ${\bf I}$-LR tableau of shape $\ydiags{2,2,2}$ is $T_{\ydiags{2,2,2}}$). The choices for $Q$ are

\[\ytabb{1 & 1 \\ 2 & 3 \\ 4 & 5},\,
\ytabb{1 & 1 \\ 2 & 4 \\ 3 & 5} \ .
\]
One sees that $c^{\complexes[\mathfrak{X}_G]}_{\ydiags{2,2,2}|\ydiags{2},\ydiags{1},\ydiags{1},\ydiags{1},\ydiags{1}}=1$ since
\[{\sf RSK}^{-1}\left(\, \ytabb{ 1 &1 \\ 2 & 2 \\ 3 & 3} \, , \, 
\ytabb{1 & 1 \\ 2 & 3 \\ 4 & 5}\, \right)=
\begin{bmatrix}
0 & 0 & 0 & 1 & 1\\
0 & 1 & 1 & 0 & 0\\
2 & 0 & 0 & 0 & 0\\
0 & 0 & 0 & 0 & 0
\end{bmatrix}\notin \monos_{\antidiag}I,\]
whereas $c^{I_G}_{\ydiags{2,2,2}|\ydiags{2},\ydiags{1},\ydiags{1},\ydiags{1},\ydiags{1}}=1$ because
\[\ \ \ \ {\sf RSK}^{-1}\left(\, \ytabb{ 1 &1 \\ 2 & 2 \\ 3 &3 } \, , \, 
\ytabb{1 & 1 \\ 2 & 4 \\ 3 & 5}\, \right)=
\begin{bmatrix}
0 & 0 & 1 & 0 & 1\\
0 & 1 & 0 & 1 & 0\\
2 & 0 & 0 & 0 & 0\\
0 & 0 & 0 & 0 & 0
\end{bmatrix} \in \monos_{\antidiag}I.\]

In general, we ask: 
\begin{problem}\label{problem:matroidarebi?}
Which matrix matroid ideals are bicrystalline?
\end{problem}

When matroid varieties are bicrystalline, Theorem~\ref{thm:LRrule} provides a combinatorial rule for their multiplicities. These examples give a step toward such rules more generally.

\end{example}

\begin{example}[Square of a maximal minors ideal]
Let $I\subseteq R = {\mathbb C}[{\sf Mat}_{2,3}]$ be generated by the $2\times 2$ minors of a generic $2\times 3$ matrix
$\left[\begin{smallmatrix} z_{11} & z_{12} & z_{13} \\ z_{21} & z_{22} & z_{23}\end{smallmatrix}\right]$. Now, 
\[{\mathrm{init}}_{\antidiag}(I^2)=\langle z_{22}^2 z_{13}^2, \, z_{21}z_{22}z_{12}z_{13}, \, z_{21}^2z_{12}z_{13}, \,
z_{21}z_{22}z_{13}^2, \, z_{21}^2z_{13}^2, \, z_{21}^2z_{12}^2\rangle.\] 
One can use Algorithm~\ref{thebigalg} to verify that $I^2$ is $(\{0, 2\}, \{0, 3\}, \antidiag)$-bicrystalline. 
Corollary~\ref{cor:detpowerbicrystal} gives a proof in greater generality.\footnote{In general, ordinary and symbolic powers of an ideal differ (as in Example~\ref{exa:powers}), but here they are equal.} 
Since the Levi datum in this case corresponds to the entire group $\mathbf{GL}$, the matrices mapped into $\mathcal{LR}(\mathbf{I},\mathbf{J},\underline\lambda,\underline\mu)$ by ${\sf RSK}$ are particularly easy to describe. They are nonnegative integer matrices of the form $\left[\begin{smallmatrix}a & b & 0\\ b & 0 & 0\end{smallmatrix}\right]$, which are mapped by ${\sf RSK}$ to the highest-weight tableau-pair of shape $\lambda = (a+b, b)$. Such a matrix lies in $\monos_\antidiag I^2$ if and only if $b\geq 2$. Thus, by Theorem~\ref{thm:LRrule} the character of $R/I^2$ is
\[\sum_{\lambda:\ydiags{2,2}\not\subset \lambda} s_{\lambda}(x_1,x_2)s_{\lambda}(y_1,y_2,y_3).\]
\end{example}

\begin{example}[Nilpotent matrix Hessenberg variety]\label{exa:Hessenberg} We follow an example of R.~Goldin--M.~Precup 
\cite[Example~2.9]{Goldin.Precup}. The ideal $I\subseteq R = {\mathbb C}[{\sf Mat}_{4,2}]$ is defined by $3\times 3$ minors of $\left[\begin{smallmatrix}
0 & z_{11} & z_{12}\\
0 & z_{21} & z_{22}\\
0 & z_{31} & z_{32}\\
z_{11} & z_{41} & z_{42}
\end{smallmatrix}\right]$.\footnote{As explained in \cite{Goldin.Precup}, this corresponds to the nilpotent matrix Hessenberg variety associated to the
nilpotent matrix 
$\left[\begin{smallmatrix} 1 & 0 & 0 & 1 \\ 0 & 1 & 0 & 0 \\ 0 & 0 & 1 & 0 \\ 0& 0 & 0 &1\end{smallmatrix}\right]$ and 
Hessenberg function $h=(2,4,4,4)$. Our conventions differ from \cite{Goldin.Precup}; our matrix of variables is flipped vertically from theirs.} It cuts out a union of two matrix Schubert varieties 
and is stable under the action of $(GL_1\times GL_2\times GL_1)\times T_2$.

Under $\antidiag$, the defining generators form a Gr\"obner basis, with 
\[{\mathrm{init}}_{\antidiag} I=\langle z_{11}z_{12}z_{21}, z_{11}z_{12}z_{31}, z_{11}z_{22}z_{31} \rangle.\]
Algorithm~\ref{thebigalg} verifies that $I$ is bicrystalline for $\antidiag$ and the given Levi datum.

Now, $c^{R/I}_{\ydiags{1},\, \ydiags{2,1},\,\ydiags{2}\, |\, \ydiags{2}, \, \ydiags{4}}=2$; the two matrices counted by
Theorem~\ref{thm:LRrule} are:
\begin{align*}
\begin{bmatrix}
0 & 1 \\ 
0 & 2 \\
1 & 0 \\
1 & 1
\end{bmatrix} \stackrel{{\sf RSK}}{\mapsto} \left(\,\ytabb{1 & 2 & 2 & 4\\ 3 & 4},
\ytabb{1 & 1 & 2 & 2\\ 2 & 2}\,
\right), \\
\begin{bmatrix}
0 & 1 \\ 
1 & 1 \\
1 & 0 \\
0 & 2
\end{bmatrix} \stackrel{{\sf RSK}}{\mapsto} \left(\,\ytabb{1 & 2 & 4 & 4 \\ 2 & 3},
\ytabb{1 & 1 & 2 & 2 \\ 2 & 2}
\,\right).
\end{align*}

Using $\diag$ instead, the defining generators of $I$ still form a Gr\"obner basis, so
\[{\mathrm{init}}_{\diag}I =\langle z_{11}^2z_{22}, z_{11}^2z_{32}, z_{11}z_{21}z_{32} \rangle.\]
However, $I$ is not bicrystalline with respect to $\diag$ since, e.g.,
\[\begin{bmatrix}1 & 0\\ 0 & 0\\ 1 & 1\\ 0 & 0\end{bmatrix}\notin\monos_{\diag}I\ \ \text{but}\ \ e_2^{\sf row}\left(\begin{bmatrix}1 & 0\\ 0 & 0\\ 1 & 1\\ 0 & 0\end{bmatrix}\right)=
\begin{bmatrix}1 & 0\\ 1 & 0\\ 0 & 1\\ 0 & 0\end{bmatrix}\in \monos_{\diag}I.\] 
Indeed, attempting to apply Theorem~\ref{thm:LRrule} to $I$ under $\diag$ yields incorrect results. For instance, in the above computation, one would include the third matrix 
\begin{align*}
\begin{bmatrix}
1 & 0 \\ 
0 & 2 \\
1 & 0 \\
0 & 2
\end{bmatrix} \stackrel{{\sf RSK}}{\mapsto} \left(\,\ytabb{1 & 2 & 2 & 4 & 4 \\ 3},
\ytabb{1 & 1 & 2 & 2 & 2 \\ 2}
\,\right),
\end{align*}
concluding incorrectly that $c^{R/I}_{\ydiags{1},\, \ydiags{2,1},\,\ydiags{2}\, |\, \ydiags{2}, \, \ydiags{4}}>2$.
Therefore the bicrystalline hypothesis in Theorem~\ref{thm:LRrule} cannot be dispensed with.

What other matrix Hessenberg varieties are bicrystalline with respect to the largest Levi group that acts on them? Some of these varieties are unions of matrix Schubert varieties \cite[Proposition 7.2]{Goldin.Precup}, hence bicrystalline by Theorem~\ref{thm:knutsonbicrystal}, but the problem is open in general. 
For bicrystalline matrix Hessenberg varieties, is there a reformulation of Theorem~\ref{thm:LRrule} purely in terms of the ``Hessenberg data'' of the matrix operator and $h$? \end{example}

\begin{example}[Alternating sign matrix (ASM) varieties]\label{exa:ASM} Let 
\begin{equation}
\label{eqn:June24aaa}
I=\left\langle z_{11}, \left|\begin{matrix}
z_{11} & z_{12}\\
z_{21} & z_{22}
\end{matrix}\right|\right\rangle\subseteq {\mathbb C}[{\sf Mat}_{2,2}].
\end{equation}
This is an \emph{alternating sign matrix variety} as defined by
A.~Weigandt \cite{Weigandt}, corresponding to the matrix $\left[\begin{smallmatrix} 0 & 1 & 0 \\ 1 &-1 & 1\\ 0 & 1 & 0
\end{smallmatrix}\right]$ (after a change of coordinates). It is a reduced union of two matrix Schubert varieties. In general, all ASM varieties are unions of matrix Schubert varieties. Therefore, one can apply Theorem~\ref{thm:LRrule} to compute the desired irreducible representation multiplicities for all ASM ideals; we originally proved this in \cite[Theorem 1.14]{AAA}
but it also follows from Theorem~\ref{thm:knutsonbicrystal}. Is there a rule for the multiplicities
in terms of the data of the indexing ASM?
\end{example}

W.~Graham asked if, without a (proved) Gr\"obner basis for $I$, one can still ascertain information about the irreducible multiplicities of $\complexes[{\sf Mat}_{m, n}]/I$. We formulate the question as follows. Suppose $I\subseteq R = {\mathbb C}[{\sf Mat}_{m,n}]$, $({\bf I},{\bf J})$ is a Levi datum, and $\prec$ is a term order.
Assume $I$ is $L_{\bf I}\times L_{\bf J}$ stable and  
\[{\mathcal G}^{\mathrm{fake}}=\{g_1,\ldots,g_f\}\] 
is a collection of elements of $I$. Define
\[{\mathrm{fakeinit}}_{\prec} I:=\langle \init_{\prec}(g_i):1\leq i\leq f\rangle,\quad {\sf fakeMat}_\prec I :=\{M\in{\sf Mat}_{m,n}(\integers_{\geq 0}):z^M\in\mathrm{fakeinit}_\prec I\}.\]
Then ${\sf fakeMat}_\prec I\subseteq\monos_\prec I$. Suppose that ${\sf fakeMat}_\prec I$ is $(\mathbf{I}, \mathbf{J})$-bicrystal closed.

\begin{problem}
Prove or disprove:
\[c^{I}_{\underline\lambda|\underline\mu}\geq \#\left\{M\in{\sf fakeMat}_\prec I: {\sf RSK}(M)\in {\mathcal{LR}}({\bf I},{\bf J},\underline\lambda,\underline\mu)\right\}.\]
\end{problem}

\subsection{Proof of Theorem~\ref{thm:LRrule}} 
For a semistandard tableau $P$, let $P|_{[a,b]}$ be the (skew) subtableau consisting only of entries from $[a,b]$. To prove Theorem~\ref{thm:LRrule}, we need the following well-known lemma:
\begin{lemma}[{\cite[Lemma 3.2.3]{Fulton}}]\label{lemma:June25r123}
If $w$ and $w'$ are words on $[1, m]$ such that $w\equiv w'$, then for any subinterval $[a, a']\subseteq[1, m]$ we have
\[w|_{[a, a']}\equiv w'|_{[a, a']},\]
where $w|_{[a,b]}$ is the subword of $w$ using only entries in $[a,b]$ as in Section~\ref{subsec:highestweight}. 
In particular,
\begin{equation}
\label{eqn:June25aaa}
{\sf word}({\sf tab}(w)|_{[a,b]})={\sf word}({\sf tab}(w))|_{[a,b]}\equiv w|_{[a,b]}.
\end{equation}
\end{lemma}

\noindent
\emph{Proof of Theorem~\ref{thm:LRrule}:} Recall the supersemistandard tableaux from Definition~\ref{def:supersemistandard}. The following fact is well-known \cite[Lemma~5.2.1]{Fulton}. A skew semistandard 
tableau $T$ is Littlewood--Richardson of content $\mu$ if and only if 
${\sf tab}({\sf word}(T))=T_{\mu}$. 
It is immediate from this fact that 
$(P,Q)={\sf RSK}(M)\in {\mathcal{LR}}({\bf I},{\bf J},\underline\lambda,\underline\mu)$ if and only if 
\begin{equation}
\label{eqn:June25tt1}
{\sf tab}({\sf word}(P|_{[i_{\alpha-1}+1, i_{\alpha}]}))=T_{\lambda^{(\alpha)}}[i_{\alpha-1}+1,i_{\alpha}], \ 1\leq\alpha\leq r,
\end{equation}
and
\begin{equation}
\label{eqn:June25tt2}
{\sf tab}({\sf word}(Q|_{[j_{\beta-1}+1, j_{\beta}]}))=T_{\mu^{(\beta)}}[j_{\beta-1}+1,j_{\beta}], \ 1\leq\beta\leq s.
\end{equation}
By Lemma~\ref{lemma:June25r123} and the fact that ${\sf RSK}(M)=(P,Q)$,
\[{\sf word}(P|_{[i_{\alpha-1}+1, i_{\alpha}]})\equiv {\sf row}(M)|_{[i_{\alpha-1}+1,i_{\alpha}]}\]
 and
 \[{\sf word}(Q|_{[j_{\beta-1}+1, j_{\beta}]})\equiv {\sf col}(M)|_{[j_{\beta-1}+1,j_{\beta}]}.\]
Therefore, by
Theorem~\ref{thm:June24bbb} and Theorem~\ref{thm:Knuthstraight}, \eqref{eqn:June25tt1} and \eqref{eqn:June25tt2} are respectively  equivalent to 
\begin{equation}
\label{eqn:June25uu1}
{\sf row}(M)|_{[i_{\alpha-1}+1,i_{\alpha}]}\equiv {\sf word}(T_{\lambda^{(\alpha)}}[i_{\alpha-1}+1,i_{\alpha}]), \ 1\leq\alpha\leq r,
\end{equation}
and
\begin{equation}
\label{eqn:June25uu2}
{\sf col}(M)|_{[j_{\beta-1}+1,j_{\beta}]}\equiv {\sf word}(T_{\mu^{(\beta)}}[j_{\beta-1}+1,j_{\beta}]), \ 1\leq\beta\leq s.
\end{equation}
By Theorem~\ref{thm:Knuthstraight} again, 
\eqref{eqn:June25uu1} and \eqref{eqn:June25uu2} are equivalent to 
\begin{equation}
\label{eqn:June25vv1}
{\sf tab}({\sf row}(M)|_{[i_{\alpha-1}+1,i_{\alpha}]})=
T_{\lambda^{(\alpha)}}[i_{\alpha-1}+1,i_{\alpha}], \ 1\leq\alpha\leq r,
\end{equation}
and
\begin{equation}
\label{eqn:June25vv2}
{\sf tab}({\sf col}(M)|_{[j_{\beta-1}+1,j_{\beta}]})= T_{\mu^{(\beta)}}[j_{\beta-1}+1,j_{\beta}], \ 1\leq\beta\leq s.
\end{equation}
Lastly, by Proposition~\ref{prop:LRtab}, $M$ is $({\bf I},{\bf J})$-highest weight if and only if ${\sf RSK}(M)=(P,Q)$ is $({\bf I},{\bf J})$-LR. 

In conclusion, the set of $z^M\in\mathrm{Std}_{\prec}I$ enumerated on the right hand side of \eqref{eqn:therule} are
those such that $M$ is highest-weight and \eqref{eqn:June25vv1} and \eqref{eqn:June25vv2} hold. This is precisely the rule of \cite[Main Theorem 1.11]{AAA}, which states that $c^{R/I}_{\underline\lambda|\underline\mu}$ counts the number of standard monomials $z^M\in\mathrm{Std}_\prec I$ such that ${\sf filterRSK}_{{\bf I}|{\bf J}}(M)$ (as defined in \cite[Main Definition~1.5]{AAA}) is the highest weight tableau tuple $(T_{\underline\lambda}|T_{\underline\mu})$ 
(defined in \cite[Definition~1.10]{AAA}). This proves the rule \eqref{eqn:therule}. The rule \eqref{eqn:therule2} for $c^I_{\underline\lambda|\underline\mu}$ then follows from \eqref{eqn:therule} and the isomorphism \eqref{eqn:basicdecomp} of Levi-representations.
\qed

\begin{example}[Comparison with Example~1.15 of \cite{AAA}]\label{exa:comparison}
Let 
\[I=\left\langle z_{11}, \left|\begin{matrix} z_{11} & z_{12} & z_{13}\\
z_{21} & z_{22} & z_{23}\\
z_{31} & z_{32} & z_{33}
\end{matrix}\right|\right\rangle \subseteq R = {\mathbb C}[{\sf Mat}_{4,4}].\]
Here $({\bf I},{\bf J})=(\{0,1,4\},\{0,1,4\})$. Suppose $\underline\lambda=\underline\mu=(\ydiags{1}, \ydiags{2,1})$.
The two $({\bf I},{\bf J})$-LR tableau pairs of content $(\underline\lambda,\underline\mu)$ are:
\[\left(\, \ytabb{1 & 2 \\ 2 \\ 3 }, \ytabb{1 & 2 \\ 2 \\ 3}\, \right) \text{\ and \ }
\left(\, \ytabb{1 & 2 \\ 2 & 3 }, \ytabb{1 & 2 \\ 2 & 3}\, \right).
\]
By applying ${\sf RSK}^{-1}$, the corresponding
matrices are, respectively,
\[\begin{bmatrix}
0 & 0 & 1 & 0\\
0 & 2 & 0 & 0\\
1 & 0 & 0 & 0\\
0 & 0 & 0 & 0 
\end{bmatrix}\leftrightarrow z_{13}z_{22}^2z_{31} \text{ \ and  \ } 
\begin{bmatrix}
0 & 1 & 0 & 0\\
1 & 0 & 1 & 0\\
0 & 1 & 0 & 0\\
0 & 0 & 0 & 0 
\end{bmatrix}\leftrightarrow z_{12}z_{21}z_{23}z_{32}.
\]
As seen in \cite[Main Theorem 1.14]{AAA}, $I$ is $({\bf I}, {\bf J},\antidiag)$-bicrystalline (this also holds by Theorem~\ref{thm:knutsonbicrystal}). These two monomials are in ${\mathrm{Std}}_{\antidiag}I$. 
Hence, $c^{R/I}_{\underline\lambda|\underline\mu}=2$. Using \cite[Main Theorem 1.11]{AAA} instead, the two monomials above are identified as the only standard monomials $z^M\in\mathrm{Std}_\antidiag I$ such that 
\begin{equation}\label{eqn:June23aaa}
{\sf tab}({\sf row}(M)|_{[1,1]}),{\sf tab}({\sf col}(M)|_{[1,1]})
=\ytabb{1} \text{ \ and \ }
{\sf tab}({\sf row}(M)|_{[2,4]}),
{\sf tab}({\sf col}(M)|_{[2,4]})
=\ytabb{2 & 2\\ 3}.
\end{equation}

The list of tableaux appearing in \eqref{eqn:June25vv1} and \eqref{eqn:June25vv2} is insufficient to reconstruct a unique monomial $z^M\in\mathrm{Std}_\antidiag I$. 
Multiple standard monomials may give the same tuple of such tableaux, as seen in \eqref{eqn:June23aaa}. The rule of  Theorem~\ref{thm:LRrule} also improves on our rule from \cite{AAA} by removing this many-to-one issue: via the bijectivity of {\sf RSK}, each highest-weight standard monomial can be reconstructed from an $(\mathbf{I},\mathbf{J})$-LR tableau-pair.
\end{example}

\section{Gr\"obner-determinantal ideals}\label{subsec:bidi}
We define \emph{Gr\"obner-determinantal ideals} and apply our theory to them. 
Theorem~\ref{thm:generalizeddetideals} states all such ideals are bicrystalline (with respect to appropriate Levi groups). This is proved using Theorem~\ref{thm:firstmain}. In the following section, we show that this class of ideals includes Schubert determinantal ideals and more. 
We also give a simplification of Theorem~\ref{thm:LRrule} for Gr\"obner-determinantal ideals stable under particularly large Levi actions.

\subsection{Definition and bicrystallinity}
 As usual, let $Z=[z_{ij}]$ be an $m\times n$ matrix of variables. For simplicity, identify contiguous rectangular submatrices of $Z$ by their row and column indices; for example, the submatrix of $Z$ using rows $2$, $3$, $4$ and columns $1, 2$ will be denoted $[2, 4]\times [1, 2]$. Fix a set of submatrices
 \[{\mathcal U}=\{U_1,U_2,\ldots,U_k\}\] 
 where 
 \[U_i = [a_i, a'_i]\times[b_i, b'_i] \text{\ for $1\leq a_i \leq a'_i\leq m$ and $1\leq b_i\leq b'_i\leq n$.}\]
 For each $1\leq i\leq k$, let ${\mathcal G}_i$ be the set of all $d_i\times d_i$ minors of $U_i$, where
 ${\mathcal D}=\{d_1,\ldots,d_k\}$ is a set of positive integers. Let 
 \[{\mathcal G} = \bigcup_{i=1}^k {\mathcal G}_i,  \text{ \ and \ } I_{\mathcal G}=\langle {\mathcal G} \rangle\subseteq \complexes[{\sf Mat}_{m, n}]\] 
 be the ideal generated by $\mathcal{G}$. Without loss, we assume that our description of $I_\mathcal{G}$ is \emph{irredundant}, meaning that at least one $d_i\times d_i$ minor from each $U_i$ is not generated by the other minors in $\mathcal{G}$.
\begin{definition}\label{def:gendetideals}
   Call ideals $I_\mathcal{G}$ of the above form \emph{contiguous determinantal ideals}. Such an
   ideal $I_{\mathcal G}$ is furthermore \textit{Gr\"obner-determinantal} if ${\mathcal G}$ is a Gr{\" o}bner basis for $I$ under some choice of $\antidiag$.
\end{definition}

\begin{lemma}\label{lemma:grobnerdetidealstable}
    Let $I_{\mathcal G}$ be a Gr\"obner-determinantal ideal and let $(\mathbf{I}, \mathbf{J})$ be a Levi datum. If 
    \[a_i-1, a'_i\in \mathbf{I} \text{\  and \  $b_i-1, b'_i\in \mathbf{J}$ for all $i\in[k]$,}\]
     then $I_\mathcal{G}$ is $L_{\mathbf{I}}\times L_{\mathbf{J}}$-stable.
\end{lemma}
\begin{proof}
    Immediate from the construction of $I_\mathcal{G}$.
\end{proof}

We prove the following theorem by applying Algorithm~\ref{thebigalg} to construct explicit test sets for $I_\mathcal{G}$, then using Theorem~\ref{thm:firstmain} to verify the bicrystalline property. 

\begin{theorem}\label{thm:generalizeddetideals}
	If $I_{\mathcal G}$ is a $L_{\bf I}\times L_{\bf J}$-stable Gr\"obner-determinantal ideal, then $I_{\mathcal G}$ is $(\mathbf{I},\! \mathbf{J},\! \antidiag)$-bicrystalline. 
\end{theorem}
\begin{proof}
    For each admissible operator $\varphi$, Algorithm~\ref{thebigalg} produces a test set $\mathcal{M}(I_\mathcal{G}, \antidiag, \varphi)$. We focus on the case where $\varphi = f_i^{\sf row}$, as the four cases are almost identical. By construction, $\mathcal{M}(I_\mathcal{G}, \antidiag, f_i^{\sf row})$ consists of elements of the form $M(g)+A$, where $g\in\mathcal{G}$ and $A$ is a nonnegative integer matrix concentrated in rows $i$ and $i+1$ with 
    $|A|\leq \Sigma_g+1 = 3$.
     
    To show that $I_\mathcal{G}$ is $(\mathbf{I}, \mathbf{J}, \antidiag)$-bicrystalline via Theorem~\ref{thm:firstmain}, we must verify that each element $M\in\mathcal{M}(I_\mathcal{G}, \antidiag, f_i^{\sf row})$ satisfies 
    \[f_i^{\sf row}(M)\in\monos_{\antidiag} I_\mathcal{G}\cup\{\varnothing\}.\]
Let 
    \[M\in\mathcal{M}(I_\mathcal{G}, \antidiag, f_i^{\sf row})\] 
    be arbitrary such that $f_i^{\sf row}(M)\neq\varnothing$ and let $g\in\mathcal{G}$ be a minor such that $z^{M(g)}|z^M$. Say $g$ is a $d\times d$ minor of some contiguous submatrix $U = [a, a']\times[b, b']$ in the definition of $I_\mathcal{G}$. We denote the row and column sets of $g$ by 
    \[R_g = \{r_1<\dots<r_d\} \text{\ and $C_g = \{c_1>\dots>c_d\}$}\] 
    respectively, so 
    \[z^{M(g)} = \prod_{i=1}^d z_{r_i c_i}.\] 
    Say $f_i^{\sf row}(M)$ moves from $(i, j)$. Then $z^{f_i^{\sf row}(M)}$ is still divisible by $z^{M(g)}$ unless there exists some $k\in[d]$ such that $(i, j) = (r_k, c_k)$ and $M_{i, j} = 1$. When this occurs, we identify another $d\times d$ minor of $U$ whose antidiagonal term divides $z^{f_i^{\sf row}(M)}$. There are two cases.

    \noindent\emph{Case 1:} ($r_{k+1}\neq i+1$). Let $g'\in\mathcal{G}$ be the minor with 
    \[R_{g'} = (R_g\smallsetminus\{i\})\cup\{i+1\} \text{ and $C_{g'} = C_g$.}\] 
    Then $z^{f_i^{\sf row}(M)}$ is divisible by $z^{M(g')}$.

    \noindent\emph{Case 2:} ($r_{k+1} = i+1$). Since $M_{r_{k+1}, c_{k+1}} > 0$ and $f_i^{\sf row}(M)$ moves from $(i, c_k)$, there must exist a column index $c\in[c_{k+1}, c_k]$ such that  
    $M_{i,c}>0$ by the pairing procedure in ${\sf bracket}_i({\sf row}(M))$.
    Let $g''\in\mathcal{G}$ be the minor with 
    \[R_{g''} = R_g \text{ and $C_{g''} = (C_g\smallsetminus\{c_k\})\cup\{c\}$.}\] 
    Then $M^{f_i^{\sf row}(M)}$ is divisible by $z^{M(g'')}$  (See Figure~\ref{fig:testgrobdet}).

    Thus in all cases $f_i^{\sf row}(M)\in\monos_{\antidiag}I_\mathcal{G}\cup\{\varnothing\}$ and the proof is complete.
\end{proof}

\begin{figure}
        \centering
        \begin{subfigure}{0.20\textwidth}
        \centering
        $\begin{bmatrix}
            0 & 0 & 0 & 0 & 1\\
            0 & 0 & 0 & 0 & 0\\
            0 & 1 & 0 & 1 & 0\\
            1 & 0 & 0 & 0 & 0
        \end{bmatrix}$
        \caption*{$M$}
        \end{subfigure}
	\begin{subfigure}{0.20\textwidth}
        \centering
        $\begin{bmatrix}
            0 & 0 & 0 & 0 & 1\\
            0 & 0 & 0 & 0 & 0\\
            0 & 1 & 0 & 0 & 0\\
            1 & 0 & 0 & 1 & 0
        \end{bmatrix}$
        \caption*{$f_3^{\sf row}(M)$}
        \end{subfigure}
        \begin{subfigure}{0.20\textwidth}
        \centering
        $\begin{bmatrix}
            0 & 0 & 0 & 0 & 1\\
            0 & 0 & 0 & 0 & 0\\
            0 & 0 & 0 & 1 & 0\\
            1 & 0 & 0 & 0 & 0
        \end{bmatrix}$
        \caption*{$M(g)$}
        \end{subfigure}
        \begin{subfigure}{0.20\textwidth}
        \centering
        $\begin{bmatrix}
            0 & 0 & 0 & 0 & 1\\
            0 & 0 & 0 & 0 & 0\\
            0 & 1 & 0 & 0 & 0\\
            1 & 0 & 0 & 0 & 0
        \end{bmatrix}$
        \caption*{$M(g'')$}
        \end{subfigure}
    \caption{Case 2 in the proof of Theorem~\ref{thm:generalizeddetideals}, with $(i, j) = (3, 4)$.}\label{fig:testgrobdet}
\end{figure}

\begin{example}[Matrix Richardson ideal]\label{exa:MRI}
    Let $m = n = 4$ and set \[I = \left\langle z_{11}, \left|\begin{matrix}
        z_{11} & z_{12} & z_{13}\\
        z_{21} & z_{22} & z_{23}\\
        z_{31} & z_{32} & z_{33}
    \end{matrix}\right|, \left|\begin{matrix}
        z_{21} & z_{22} & z_{23} & z_{24}\\
        z_{31} & z_{32} & z_{33} & z_{34}\\
        z_{41} & z_{42} & z_{43} & z_{44}\\
        z_{51} & z_{52} & z_{53} & z_{54}
    \end{matrix}\right|\right\rangle.\]
 The generators form a Gr{\"o}bner basis under $\antidiag$, so $I$ is 
 Gr\"obner-determinantal. By Theorem~\ref{thm:generalizeddetideals},
  $I$ is $(\{0, 1, 3, 5\}, \{0, 1, 3, 4, 5\}, \antidiag)$-bicrystalline.\footnote{Let $\pi:GL_n\to B_{-}\backslash GL_n$ be the projection map to the complete flag variety and suppose $X_w=\overline{B_{-}\backslash B_{-}wB}$ is a Schubert variety. The \emph{matrix Richardson variety} is the closure of $\pi^{-1}(X_w)$ in ${\sf Mat}_{n,n}$. The ideal in question cuts out one such example. In general, a Gr\"obner basis or description of $\init_\antidiag I$ for matrix Richardson ideals is unknown; \cite{Knutson.Miller, Klein.Weigandt} answer cases of this problem.  They are not all Gr\"obner-determinantal ideals, see Example~\ref{ex:nonbicrystallineMRV}.
    Which matrix Richardson ideals are bicrystalline?}
\end{example}

\subsection{A uniform simplification of Theorem~\ref{thm:LRrule}}\label{subsec:LRruleapplications}
Recall the simple formula \eqref{eqn:detcauchy} for $c_{\lambda|\mu}^{R/I_k}$ when $I_k$ is a classical determinantal ideal
(see Example~\ref{exa:det}). Gr\"obner-determinantal ideals $I_{\mathcal G}$  generalize classical determinantal ideals: they
have initial ideals generated by antidiagonals of minors,
carry large group actions, and
are bicrystalline. When do the standard monomial conditions for $I_{\mathcal{G}}$ translate into simple tableaux conditions via ${\sf RSK}$, generalizing \eqref{eqn:detcauchy}? We give such a translation when $I_{\mathcal{G}}$ is stable under $L_{\bf I}\times GL_n$ or $GL_m\times L_{\bf J}$.

\begin{definition}\label{def:wordwidth}
    Let $w = w_1\dots w_k$ be a word on $[1, m]$. The \emph{$[a, a']$-width} of $w$, denoted ${\sf width}_{[a, a']}(w)$, is the length $d$ of a longest decreasing subsequence $w_{i_1} > w_{i_2} > \dots > w_{i_d}$ ($i_1 < 
    \dots < i_d$) in the restriction $w|_{[a, a']}$ of $w$ to the subinterval $[a, a']\subseteq[1, m]$.
\end{definition}

\begin{example}\label{exa:matrixwidth}
    Let $M = \left[\begin{smallmatrix}
        1 & 0 & 1\\
        0 & 1 & 1\\
        1 & 1 & 1
    \end{smallmatrix}\right]$, let $w = {\sf row}(M) = 1323123$, and let $[a,a'] = [2,3]$. The following underlined subsequences of ${\sf row}(M)$ are longest decreasing subsequence of $w|_{[2,3]}$:
    \[{\underline 3}{\underline 2}323, \; {\underline 3}23{\underline 2}3, \; 32{\underline 3}{\underline 2}3.\] Since these subsequences have length $2$, ${\sf width}_{[2,3]}(w) = 2$. 

    Notice that, for instance, the subsequence $32{\underline 3}{\underline 2}3$ comes from the following underlined antidiagonal entries of $M$:
    \[
    M = \begin{bmatrix}
        1 & 0 & 1\\
        0 & 1 & \underline{1}\\
        1 & \underline{1} & 1
    \end{bmatrix}.
    \] 
    By \emph{antidiagonal} we mean matrix positions $(i_1,j_1),\ldots,(i_d,j_d)$ satisfying $i_1>\dots>i_d$ and $j_1<\dots<j_d$. \end{example}
    
    In general, the following is clear:
    
    \begin{lemma}\label{lemma:matrixwidth}
    For $M\in{\sf Mat}_{m,n}(\mathbb{Z}_{\geq 0})$ and any $[a,a']\subseteq [1,m]$, ${\sf width}_{[a,a']}({\sf row}(M))$ is the length of the longest sequence of non-zero antidiagonal entries of $M$ with row indices in $[a, a']$.  Likewise, for any $[b,b']\subseteq [1,n]$, ${\sf width}_{[b,b']}({\sf col}(M))$ is the length of the longest sequence of non-zero antidiagonal entries of $M$ with column indices in $[b,b']$.
\end{lemma}

The following lemma is a version of a classical result due to Schensted. It forms the foundation for translating standard monomial conditions to tableaux via ${\sf RSK}$.

\begin{lemma}[Schensted]\label{lemma:Schenstedlemma}
    Let $w,w'$ be two words on $[1,m]$ with $w\equiv w'$. Let $[a,a']\subseteq [1,m]$. Then ${\sf width}_{[a,a']}(w) = {\sf width}_{[a,a']}(w')$. 
\end{lemma}
\begin{proof}
    If $w$ and $w'$ are Knuth equivalent, then the restrictions $w|_{[a, a']}$ and $w'|_{[a, a']}$ are also Knuth equivalent by Lemma~\ref{lemma:June25r123}. The lemma statement is thus equivalent to the claim that the longest decreasing subsequences of Knuth equivalent words have the same length, which follows from \cite[Lemma 3.1.2, Exercise 3.1]{Fulton}.
\end{proof}

\begin{definition}\label{def:tabantidiag}
    Let $T\in\operatorname{SSYT}(\lambda, m)$ be a tableau, with restriction $T|_{[a, a']}$ to the subinterval $[a, a']\subseteq[1, m]$. 
    An \emph{$[a, a']$-antidiagonal} of length $d$ in $T$ is a sequence of $d$ boxes from distinct rows of $T|_{[a, a']}$, each box weakly east and strictly north of the previous, such that when the boxes are read from bottom to top, their fillings strictly decrease.
\end{definition}

\begin{example}
    Let 
    $$T = \ytabb{
        1 & 2 & 2 & *(CornflowerBlue)3 & 5 & 5 & 5\\
        2 & 4\\
        *(CornflowerBlue)4
    }, \; T' = \ytabb{
        *(CornflowerBlue) 1 & 2 & 2 & 3 & 5 & 5 & 5\\
        *(CornflowerBlue) 2 & 4\\
        *(CornflowerBlue) 4
    }$$ The highlighted elements of $T$ form a $[3,4]$-antidiagonal. Similarly, the highlighted elements of $T'$ form a $[1,4]$-antidiagonal. However, the red-highlighted elements of the tableau below do \emph{not} form an antidiagonal: $$\ytabb{
        1 & 2 & 2 & 3 & *(Lavender)4 & 5 & 5\\
        2 & 5\\
        *(Lavender) 3
    }.$$ 
\end{example}

\begin{definition}\label{def:tabwidth}
    Let $T\in\operatorname{SSYT}(\lambda, m)$ be a tableau, with restriction $T|_{[a, a']}$ to the subinterval $[a, a']\subseteq[1, m]$. The \emph{$[a,a']$-width} of $T$, denoted ${\sf width}_{[a,a']}(T)$, is the length of the longest $[a,a']$-antidiagonal of $T$.
\end{definition}

\begin{example}\label{exa:tabwidth}
    If $T\in \operatorname{SSYT}(\lambda,m)$, then any column of $T$ is a $[1,m]$-antidiagonal. So, ${\sf width}_{[1,m]}(T)$ is always the length of the longest column of $T$, i.e., $\ell(\lambda)$.
\end{example}

\begin{example}
     Let 
    $$T = \ytabb{
        1 & 2 & 2 & *(CornflowerBlue)3 & 5 & 5 & 5\\
        2 & 4\\
        *(CornflowerBlue)4
    }.$$ The $[3,4]$-width of $T$ is $2$, as witnessed by the antidiagonal highlighted in blue.
\end{example}

\begin{lemma}\label{lemma:tabwidth}
    Let $P\in {\operatorname{SSYT}}(\lambda,m)$ and $[a,a']\subseteq [1,m]$. Then 
    \[{\sf width}_{[a,a']}(P) = {\sf width}_{[a,a']}({\sf word}(P)).\]
\end{lemma}
\begin{proof}
    This follows from the definitions of ${\sf word}(P)$ and $[a,a']$-antidiagonal, since the entries of any antidiagonal in $P$ form a decreasing sequence in ${\sf word}(P)$ and vice versa. 
\end{proof}

\begin{proposition}\label{prop:goodschensted}
    Let $M\in {\sf Mat}_{m,n}(\mathbb{Z}_{\geq 0})$ and let ${\sf RSK}(M) = (P,Q)$. Then, for any $[a,a']\subseteq [1,m]$ and $[b,b']\subseteq [1,n]$, 
    \[{\sf width}_{[a,a']}(P) = {\sf width}_{[a,a']}({\sf row}(M)), \; {\sf width}_{[b,b']}(Q) = {\sf width}_{[a,a']}({\sf col}(M)).\]
\end{proposition}
\begin{proof}
    Recall that ${\sf row}(M)\equiv {\sf word}(P)$ by Theorem~\ref{thm:June24bbb}. Thus  Lemma~\ref{lemma:Schenstedlemma} implies that 
    \[{\sf width}_{[a,a']}({\sf row}(M)) = {\sf width}_{[a,a']}({\sf word}(P)).\] The first claimed equality follows by Lemma~\ref{lemma:tabwidth}; taking transposes gives the second.
\end{proof}

\begin{theorem}\label{thm:simplifiedGrobnerdet}
    Let $I_{\mathcal{G}}\subseteq R = \complexes[{\sf Mat}_{m, n}]$ be a Gr\"obner-determinantal ideal defined by the set of submatrices $\mathcal{U} = \{[a_i, a'_i]\times[1, n]:1\leq i\leq k\}$ and rank conditions $\mathcal{D} = \{d_i:1\leq i\leq k\}$. Then $I_{\mathcal{G}}$ is $(L_{\mathbf{I}}\times GL_n)$-stable for some $\mathbf{I}$ and 
    \[c^{R/I_{\mathcal{G}}}_{\underline{\lambda}|\mu} = \# \{P\in{\mathcal{LR}}(\mathbf{I},\underline\lambda)
    \text{ of shape }\mu: \forall i\in[k],\ 
    {\sf width}_{[a_i, a'_i]}(P) < d_i\},\]
    \[c^{I_{\mathcal{G}}}_{\underline{\lambda}|\mu} = \# \{P\in{\mathcal{LR}}(\mathbf{I},\underline\lambda)
    \text{ of shape }\mu: \exists i\in[k],\ 
    {\sf width}_{[a_i, a'_i]}(P) \geq d_i\}.\]
    Similarly, if $\mathcal{U} = \{[1, m]\times[b_i, b'_i]:1\leq i\leq k\}$, then $I_{\mathcal{G}}$ is $(GL_m\times L_{\mathbf{J}})$-stable for some $\mathbf{J}$ and
    \[c^{R/I_{\mathcal{G}}}_{\lambda|\underline\mu} = \# \{Q\in{\mathcal{LR}}(\mathbf{J},\underline\mu)
    \text{ of shape }\lambda:
    \forall i\in[k],\ 
    {\sf width}_{[b_i, b'_i]}(Q) < d_i\},\]
    \[c^{I_{\mathcal{G}}}_{\lambda|\underline{\mu}} = \# \{Q\in{\mathcal{LR}}(\mathbf{J},\underline\mu)
    \text{ of shape }\lambda: \exists i\in[k],\ 
    {\sf width}_{[b_i, b'_i]}(Q) \geq d_i\}.\]
\end{theorem}
\begin{proof}
    By taking transposes, it suffices to prove the first pair of statements. The $(L_{\mathbf{I}}\times GL_n)$-stability claim is immediate from the definition of $I_{\mathcal{G}}$. By Theorem~\ref{thm:generalizeddetideals}, $I_{\mathcal{G}}$ is bicrystalline under $\antidiag$. Thus Theorem~\ref{thm:LRrule} applies, giving the formula
    \[c^{R/I_\mathcal{G}}_{\underline\lambda, \mu} = \#\{M\in\monos_\antidiag I_{\mathcal{G}}:{\sf RSK}(M)\in\mathcal{LR}(\mathbf{I},\{0, n\}, \underline\lambda, \mu)\}.\]
    There is a unique $\{0, n\}$-LR tableau $T_\mu$ of each shape $\mu$ (see Remark~\ref{rem:supersemistandard}). Thus ${\sf RSK}(M)=(P, Q)\in\mathcal{LR}(\mathbf{I},\{0, n\}, \underline\lambda, \mu)$ if and only if $Q = T_\mu$ and $P\in{\mathcal{LR}}(\mathbf{I},\underline\lambda)$. Since $P$ is then necessarily of shape $\mu$ (by Theorem~\ref{thm:RSK}), it follows that
    \[c^{R/I_\mathcal{G}}_{\underline\lambda, \mu} = \#\{P\in{\mathcal{LR}}(\mathbf{I},\underline\lambda):{\sf RSK}\inv((P,T_\mu))\in\monos_\antidiag I_{\mathcal{G}}\}.\]

    We claim that 
    \[{\sf RSK}\inv((P, T_\mu))\in\monos_\antidiag I_\mathcal{G} \iff {\sf width}_{[a_i, a'_i]}(P) < d_i
    \text{\ for each $i\in[k]$.}\] Let $M ={\sf RSK}\inv((P, T_\mu))$. 
    Then for each $i\in[k]$, 
    \[{\sf width}_{[a_i, a'_i]}(P)<d_i \iff {\sf width}_{[a_i,a'_i]}({\sf row}(M)) < d_i,\]
     by Proposition~\ref{prop:goodschensted}. But the interpretation of ${\sf width}_{[a_i, a'_i]}({\sf row}(M))$ in Lemma~\ref{lemma:matrixwidth} shows that this bound holds if and only if $z^M$ is not divisible by the lead term of any $d_i\times d_i$ minor of $[a_i, a'_i]\times[1, n]$ under $\antidiag$. This completes the proof of the first statement. The proof of the second statement is identical, using the formula for $c^{I_{\mathcal{G}}}_{\underline\lambda|\mu}$ from Theorem~\ref{thm:LRrule}.
\end{proof}

\begin{remark}
	The upcoming Theorem~\ref{thm:knutsonbicrystal} proves (among other things) that all contiguous determinantal ideals defined by sets of submatrices of the form $[a, a']\times[1, n]$ or $[1, m]\times[b, b']$ are in fact Gr\"obner-determinantal. Thus the Gr\"obner-determinantal condition in the hypotheses of Theorem~\ref{thm:simplifiedGrobnerdet} is always satisfied.
\end{remark}

\begin{remark}
    In~\cite[Definition 4.1]{Almousa.Gao.Huang}, A. Almousa--S. Gao--D. Huang define \emph{generalized antidiagonals} of rectangular tableaux $T$ with respect to an interval $S$ and positive integer $r$. This definition is then used to describe the standard monomials of positroid varieties in the Grassmannian. Their definition makes sense for tableaux of arbitrary shapes and agrees with our Definition~\ref{def:tabantidiag}: a tableau $T$ contains a generalized antidiagonal for $S\leq r$ in their terminology if and only if ${\sf width}_S(T)> r$. We expect that our Theorem~\ref{thm:simplifiedGrobnerdet} can also be derived from their description of standard monomials for positroid varieties, generalizing the connection between the standard monomial theory of the Grassmannian and that of classical determinantal varieties in $\complexes[{\sf Mat}_{m, n}]$ (see \cite[Section 3.2]{BCRV}).
\end{remark}

\begin{example}[Classical determinantal ideals, revisited]\label{exa:detidealredux}
    Suppose $I_{\mathcal{G}}\subseteq R = \complexes[{\sf Mat}_{m, n}]$ is a classical determinantal ideal $I_d$ (i.e., $\mathcal{U} = \{[1, m]\times[1, n]\}$ and $\mathcal{D} = \{d\}$). Each $I_d$ is $GL_m\times GL_n$-stable, and they are Gr\"obner-determinantal because the defining minors are known to form a Gr\"obner basis under $\antidiag$ (see e.g. \cite{Abhyankar, Caniglia, Sturmfels}). By Theorem~\ref{thm:simplifiedGrobnerdet},
    \[c^{R/I_d}_{\lambda|\mu} = \# \{P\in{\mathcal{LR}}(\{0, m\},\lambda)
    \text{ of shape }\mu:
    {\sf width}_{[1, m]}(P) < d\}.\]
    Recall from Remark~\ref{rem:supersemistandard} that for each $\lambda$, the supersemistandard tableau $T_\lambda$ is the unique element of $\mathcal{LR}(\{0, m\},\lambda)$. Also, a tableau $P\in\operatorname{SSYT}(\lambda, m)$ satisfies ${\sf width}_{[1, m]}(P) < d$ if and only if $\lambda$ has fewer than $d$ rows (see Example~\ref{exa:tabwidth}). The rule of Theorem~\ref{thm:simplifiedGrobnerdet} thus simplifies to the rule \eqref{eqn:detcauchy} from Example~\ref{exa:det}:
    \[c^{R/I_d}_{\lambda, \mu} = 
    \begin{cases}
        1 & \text{if } \lambda = \mu \text{ and } \ell(\lambda) < d,\\
        0 & \text{otherwise.}
    \end{cases}\]
    This derivation is essentially the reverse of B. Sturmfels' proof \cite{Sturmfels} that the defining minors of $I_d$ form a Gr\"obner basis under $\antidiag$. Sturmfels' argument combined knowledge of the values  $c^{R/I_d}_{\lambda, \mu}$ with ${\sf RSK}$ to show that the minors form a Gr\"obner basis. Here, we \emph{begin} with a Gr\"obner basis for $I_d$, using it along with ${\sf RSK}$ to obtain the formula for $c^{R/I_d}_{\lambda, \mu}$. See Example~\ref{ex:theOGex} for more on the relationship between our results and \cite{Sturmfels}.
\end{example}

\begin{example}\label{ex:matriodLR}
    We apply Theorem~\ref{thm:simplifiedGrobnerdet} to the matrix matroid ideal from Example~\ref{exa:matroid}. That ideal $I\subseteq R =
    \mathbb{C}[{\sf Mat}_{2,6}]$ is Gr\"obner-determinantal. It is defined by:
    \[{\mathcal U} = \left\{U_1 = \begin{bmatrix}
        z_{11} & z_{12} & z_{13}\\
        z_{21} & z_{22} & z_{23}
    \end{bmatrix}, U_2 = \begin{bmatrix}
        z_{14} & z_{15}\\
        z_{24} & z_{25}
    \end{bmatrix}, U_3 = \begin{bmatrix}
        z_{16}\\
        z_{26}
    \end{bmatrix}\right\},\] 
    \[{\mathcal D} = \{d_1 = 2, d_2 = 2, d_3 = 1\}.\]

Since $R/I$ is  a
    ${\bf L}=GL_2\times T_6$ representation, its irreducibles are indexed by $(\lambda|\mu_1,\ldots,\mu_6)$
    where $\ell(\lambda)\leq 2$ and $\mu_1,\ldots,\mu_6$ are one-row partitions.
    By Theorem~\ref{thm:simplifiedGrobnerdet}, we know that $c^{R/I}_{(\lambda|\mu_1,\mu_2,\mu_3,\mu_4,\mu_5,\mu_6)}$ is equal to the number of semistandard Young tableaux $Q$ of shape $\lambda$ containing no $[1,3]$-antidiagonals of length $2$, no $[4, 5]$-antidiagonals of length $2$, and no $6$'s. (Since each $\mu_i$ has $1$ part, any ballot conditions are trivially satisfied.)
    From this rule we derive an explicit, multiplicity-free formula for these constants,  as follows. 
    Let 
    \[S=\{(\lambda \, |\, \mu_1,\mu_2,\mu_3,\mu_4,\mu_5,\mu_6) : {\text{(I) and (II) below hold}}\},\]
    where:
     \begin{enumerate}
        \item[(I)] $\mu_6 = \emptyset$, and
        \item[(II)] $\lambda$ is a partition of the form $(|\mu_1| + |\mu_2| + |\mu_3| + |\mu_4| + |\mu_5| - k, k)$ for some $0\leq k\leq |\mu_4| + |\mu_5|$.
    \end{enumerate} 
    We claim:
\begin{equation}
\label{eqn:July1bbc}
c^{R/I}_{(\lambda|\mu_1,\mu_2,\mu_3,\mu_4,\mu_5,\mu_6)} = \begin{cases}
        1, \; (\lambda|\mu_1,\mu_2,\mu_3,\mu_4,\mu_5,\mu_6)\in S\\
        0, \; \text{else.}
    \end{cases}
\end{equation}
   For $c^{R/I}_{(\lambda|\mu_1,\mu_2,\mu_3,\mu_4,\mu_5,\mu_6)}>0$, (I) must hold since $Q$ has content $\underline\mu$ and contains no $6$s.
    
    To show Condition (II), first note that since $Q$ contains no $[1, 3]$-antidiagonals of length $2$, all $1$s, $2$s, and $3$s are in the first row of $Q$. Since $Q$ is semistandard, the top row of $Q$ must thus contain $\mu_1$ $1$'s followed by $\mu_2$ $2$'s and then $\mu_3$ $3$'s. If $Q$ has a second row, it must be filled only with $4$'s and $5$'s and so must have length at most $|\mu_4| + |\mu_5|$. Finally we are done since if 
    the second row of $\lambda$ has length $0\leq k\leq |\mu_4| + |\mu_5|$, there is a unique $Q$ of the desired 
    form: $Q$ must be the tableau such that ${\sf word}(Q)|_{[4,5]}$ weakly increases. 
For example, the unique such $Q$ for 
\[(\lambda|\mu_1,\mu_2,\mu_3,\mu_4,\mu_5,\mu_6) = ((6,2) \; | \; (2),(1),(2),(2),(1),(0))\] 
is 
    \[Q = \ytabb{
        1 & 1 & 2 & 3 & 3 & 5\\
        4 & 4
    }.\] 
\end{example}

\section{Knutson determinantal ideals are bicrystalline} \label{sec:Knutsonideals}
The goal of this section is to establish a large family of Gr\"obner-determinantal ideals that includes the 
matrix double Bruhat ideals discussed in Example~\ref{exa:doublebruhat}. In this section, set $\prec$ to be the antidiagonal (pure) lexicographic term order induced by following ordering of variables in $Z$:
\[z_{1n}\succ z_{1 (n-1)}\succ\dots \succ z_{11}\succ z_{2n}\succ\dots \succ z_{m1}.\]

\begin{definition}\label{def:KDI}
    A \emph{Knutson determinantal ideal} $I\subseteq\complexes[{\sf Mat}_{m, n}]$ is a contiguous determinantal ideal using any combination of the following four types of submatrices $U$ of $Z$:
    \begin{enumerate}
        \item[(I)] \emph{Northwest}-justified submatrices ($U = [1, a']\times[1, b']$),
        \item[(II)] \emph{Southeast}-justified submatrices ($U = [a, m]\times[b, n]$),
        \item[(III)] Consecutive \emph{columns} ($U = [1, m]\times[b, b']$), or
        \item[(IV)] Consecutive \emph{rows} ($U = [a, a']\times[1, n]$).
    \end{enumerate}
\end{definition}

\begin{example}
Examples~\ref{exa:det} and~\ref{exa:doublebruhat} are examples of Knutson determinantal ideals.\footnote{Knutson determinantal ideals are a particular family of \emph{Knutson ideals} as defined in \cite{ConcaV}. Knutson determinantal ideals are of special interest because they are Levi-stable.}
\end{example}

\begin{theorem}\label{thm:knutsonbicrystal}
    If $I$ is a Knutson determinantal ideal, then $I$ is Gr\"obner-determinantal under the lexicographic order $\prec$ above. Hence, by Theorem~\ref{thm:generalizeddetideals}, $I$ is $({\bf I},{\bf J}, \prec)$-bicrystalline for any 
 $({\bf I},{\bf J})$ such that $I$ is $L_{\bf I}\times L_{\bf J}$-stable.   \end{theorem}

We deduce Theorem~\ref{thm:knutsonbicrystal} from a special case of A.~Knutson's work on Frobenius splittings in \cite[Section 7.3]{Knutson}.  We thank Knutson for indicating this connection to us. In order to give the argument, we must briefly describe the relevant results in our notation.

Let ${\mathfrak S}_n$ denote the group of permutations of $[n]$. We will express permutations $v\in {\mathfrak S}_n$ either in one-line notation, or as a \emph{permutation matrix} $M_v$ that places a $1$ in matrix position $(i,v(i))$ for $1\leq i\leq n$ and $0$'s elsewhere.

\begin{definition}
    For a permutation $v\in\mathfrak{S}_n$, let $Z_v$ be the specialization of the generic $n\times n$ matrix $Z = [z_{ij}]$ obtained by setting each variable $z_{ij}$ with $j=v(i)$ to $1$, and each variable $z_{ij}$ with $j>v(i)$ or $i>v\inv(j)$ to $0$.
\end{definition}

\begin{definition}[{\cite[Section 3]{Fulton:duke}}]
The \emph{rank function} of a permutation $w$, denoted 
\[r_w:[n]\times [n]\to {\mathbb Z}_{\geq 0},\] 
maps each position $(i, j)$ to the number of $1$'s weakly northwest of it in the permutation matrix $M_w$. 
\end{definition}

\begin{definition}[{\cite[Section 3.2]{WY:governing}}]\label{def:KLideal}
    For $w, v\in\mathfrak{S}_n$, the \emph{Kazhdan--Lusztig ideal} $I_{v, w}\subseteq\complexes[Z_v]$ is 
    \[I_{v, w} = \langle(r_w(i, j)+1)\times(r_w(i, j)+1)\text{ minors of } Z_v : (i, j)\in[n]\times[n]\rangle.\]
\end{definition}

We refer the reader to the recent survey \cite{WY:survey} where the role of Kazhdan--Lusztig ideals in the study of singularities of Schubert varieties is explained. 

\begin{example}\label{exa:MSV} In Definition~\ref{def:KLideal}, replacing the use of $Z_v$ with the generic matrix $Z$ defines the \emph{Schubert determinantal ideal} $I_w$ of \cite{Fulton:duke}.
The associated (irreducible) variety is the \emph{matrix Schubert variety} \cite{Fulton:duke, Knutson.Miller}. Schubert determinantal ideals were the focus of our earlier paper \cite{AAA}. They are also Knutson determinantal ideals. 
\end{example}

\begin{definition}
    Given $v\in \mathfrak{S}_n$ and a position $(i, j)\in[n]\times[n]$, the \emph{antidiagonal drift} of $(i, j)$ is the quantity
    \[{\sf drift}_v(i, j) := i+j-(r_v(i, j)+1).\]
    The \emph{$k$th drifted antidiagonal} of $Z_v$ is the set
    \[D_v(k) := \{(i, j)\in[n]\times[n]: Z_v(i, j) = z_{ij},\ {\sf drift}_v(i, j) = k\}.\]
\end{definition}

\begin{example}\label{exa:Jun23aaa}
Let $v=31542$. Placing the values of $r_v(i,j)$ and ${\sf drift}_v(i,j)$ into a matrix gives:
\[
r_{v}=\begin{bmatrix}
0 & 0 & 1 & 1 & 1\\
1 & 1 & 2 & 2 & 2 \\
1 & 1 & 2 & 2 & 3\\
1 & 1 & 2 & 3 & 4\\
1 & 2 & 3 & 4 & 5
\end{bmatrix}, \ 
{\sf drift}_v=
\begin{bmatrix}
1 & 2 & 2 & 3 & 4\\
1 & 2 & 2 & 3 & 4 \\
2 & 3 & 3 & 4 & 4\\
3 & 4 & 4 & 4 & 4\\
4 & 4 & 4 & 4 & 4
\end{bmatrix}.
\]
The specialized matrix $Z_v$ is displayed below, along with the restriction ${\sf drift}'_v$ of the matrix ${\sf drift}_v$ to those positions $(i, j)$ such that $Z_v(i, j) = z_{ij}$.
\[
Z_{v}=\begin{bmatrix}
z_{11} & z_{12} & 1 & 0 & 0\\
1 & 0 & 0 & 0 & 0 \\
0 & z_{32} & 0 & z_{34} & 1\\
0 & z_{42} & 0 & 1 & 0\\
0 & 1 & 0 & 0 & 0
\end{bmatrix}, \ 
{\sf drift}'_v=
\begin{bmatrix}
1 & 2 & - & - & -\\
- & - & - & - & - \\
- & 3 & - & 4 & -\\
- & 4 & - & - & -\\
- & - & - & - & -
\end{bmatrix}.
\]
The $k$th drifted antidiagonal of $Z_v$ consists of the positions $(i, j)$ such that ${\sf drift}'_v(i, j) = k$.
\end{example}

\begin{definition}
    The \emph{$k$th basic submatrix} $Z_v^{(k)}$ of $Z_v$ is the northwest-justified $k\times k$ submatrix. Its determinant is the \emph{$k$th basic minor} 
    \[\Delta^{(k)}_v := \det Z_v^{(k)}.\] 
\end{definition}

\begin{theorem}[{\cite[Theorem 7]{Knutson}}]\label{thm:minorscover}
    Let $v\in\mathfrak{S}_n$, and let $\prec$ be the lexicographic order above. Then
    \begin{enumerate}
        \item[(I)] The lead term of $\Delta^{(k)}_v$ is 
            \[{\mathrm{init}}_\prec(\Delta^{(k)}_v) = \prod_{(i, j)\in D_v(k)}z_{ij}.\]
        \item[(II)] The defining generators for any sum 
        \[I = \sum_{w} I_{v, w}\] 
        of Kazhdan--Lusztig ideals $I_{v, w}$ (with $v$ fixed) form a Gr\"obner basis under $\prec$.
    \end{enumerate}
\end{theorem}

\begin{example}
    Continuing Example~\ref{exa:Jun23aaa}, let $v = 31542$. Referring to the diagram of $Z_v$ above, we compute the lead terms of the basic minors:
    \[\init_\prec(\Delta_v^{(1)}) = z_{11},\ \init_\prec(\Delta_v^{(2)}) = z_{12},\ \init_\prec(\Delta_v^{(3)}) = z_{32},\ \init_\prec(\Delta_v^{(4)}) = z_{34}z_{42}.\]
    Each variable in $Z_v$ appears in the lead term of exactly one basic minor, and the $z_{ij}$ appearing in $\init_\prec(\Delta_v^{(k)})$ are those with ${\sf drift}_v(i, j) = k$, in accordance with Theorem~\ref{thm:minorscover}(I).
\end{example}

\begin{remark}
    The definition of Kazhdan--Lusztig ideal used in \cite{Knutson} is more general than ours---in that reference, the ideal depends on a choice of Weyl group $W$, elements $w, v\in W$, and a reduced word $Q$ for $v$. 
    Our Kazhdan--Lusztig ideals $I_{v, w}$ correspond to the case where $w, v\in\mathfrak{S}_n$ are permutations and $Q$ is the ``Rothe word'' for $v$ formed by listing the variables of $Z_v$ in decreasing order according to $\prec$ and replacing each $z_{ij}$ with the simple transposition $s_{{\sf drift}_v(i, j)}$. 
    For example, the Rothe word for $v = 31542$ is $s_2s_1s_4s_3s_4$.
\end{remark}

Given part (I) of Theorem~\ref{thm:minorscover}, part (II) follows quickly by Theorems $4$ and $6$ in \cite{Knutson}. See the appendix for a self-contained proof of part (I) akin to arguments of L. Seccia in~\cite{Seccia}. Part (II) generalizes  \cite[Main Theorem 2.1]{WY}, which concerns single Kazhdan--Lusztig ideals rather than sums. The proof in 
\emph{loc.~cit.} is different from that found in \cite{Knutson}.

\begin{proof}[Proof of Theorem~\ref{thm:knutsonbicrystal}]
    We realize the sets of minors defining the four types of rank conditions in Definition~\ref{def:KDI} as the generators of Kazhdan--Lusztig ideals for appropriate choices of $v$ and $w$. Let $Z$ denote an $(m+n)\times(m+n)$ generic matrix and $\overline{Z}$ denote the northwest $m\times n$ submatrix $[1, m]\times[1, n]$. Fix $v\in\mathfrak{S}_{m+n}$ to be the permutation 
    \[v=(n+1)(n+2)\cdots(n+m)12\cdots n.\] 
    Then $Z_v$ and $\overline{Z}$ contain the same variables. 
    
    \begin{example}\label{exa:June23dd1} If $m = 3$ and $n = 4$, then
    \[
    Z_{5671234} = \left[
    \begin{array}{c c c c|c c c}
      z_{11} & z_{12} & z_{13} & z_{14} & 1 & 0 & 0\\
      z_{21} & z_{22} & z_{23} & z_{24} & 0 & 1 & 0 \\
      z_{31} & z_{32} & z_{33} & z_{34} & 0 & 0 & 1 \\
      \hline
      1 & 0 & 0 & 0 & 0 & 0 & 0\\
      0 & 1 & 0 & 0 & 0 & 0 & 0\\
      0 & 0 & 1 & 0 & 0 & 0 & 0 \\
      0 & 0 & 0 & 1 & 0 & 0 & 0 
    \end{array}
    \right].\]
    \end{example}
    
    Let $w\in\mathfrak{S}_{m+n}$ be a \emph{bigrassmannian} permutation, meaning $w$ and $w\inv$ each have exactly one descent. The $1$-line notation for such a permutation is always of the form
    \[w = (1\cdots k)(c+1\cdots c+r-k)(k+1\cdots c)(c+r-k+1\cdots m+n)\]
    for some $r, c, k$ satisfying 
    \[k < \min\{r, c\} \text{ and $r+c-k\leq m+n$.}\] 
    Moreover, for any such $r, c, k$ satisfying these conditions, there is a corresponding bigrassmannian permutation $w\in\mathfrak{S}_{m+n}$; this is well-known, see, e.g., \cite[Exercise 2.2.5]{Manivel} or
    \cite[Lemma~4.1]{RWY:coho}. For instance, if $m+n=12$, $k=3$, $r=6$, and $c=7$, the bigrassmannian permutation is
    $w=1 \ 2\  3\  8 \ 9 \ 10 \ 4 \ 5 \ 6  \ 7 \ 11 \ 12$.
    
The Kazhdan--Lusztig ideal $I_{v, w}$ is generated by the $(k+1)\times(k+1)$ minors of the submatrix $[1, r]\times[1, c]\subseteq Z_v$. Straightforwardly, the following different choices of $r$, $c$, and $k$ yield each of the four types of rank conditions in the statement of Theorem~\ref{thm:knutsonbicrystal}:
    \begin{enumerate}
        \item[(I)] It is immediate that the $(d+1)\times(d+1)$ minors of $[1, a']\times[1, b']\subseteq\overline{Z}$ are the $(k+1)\times(k+1)$ minors of $[1, r]\times[1, c]\subseteq Z_v$ for
        \[r = a', c = b', k=d.\]
        In Example~\ref{exa:June23dd1}, the $(d+1)\times (d+1)= 2\times 2$ minors of $[1,a']\times [1,b'] = [1,2]\times [1,3]$ result from using $r=2, c=3, k=1$; the bigrassmannian permutation generating this rank condition is $w= 1\ 4 \ 2 \ 3 \ 5 \ 6 \ 7$.
        \item[(II)] The $(d+1)\times(d+1)$ minors of $[a, m]\times[b, n]\subseteq\overline{Z}$ are the $(k+1)\times(k+1)$ minors of $[1, r]\times[1, c]$ for
        \[r = m+b-1, c = n+a-1, k = a+b+d-2.\]
        To see that this works, one notices that any nonzero $(d+1)\times (d+1)$ minor of $[1,r]\times [1,c]$ in $Z_v$ uses the $1$'s appearing in
        the blocks $[1,a-1]\times [n+1,n+a-1]$ and $[m+1,m+b-1]\times [1,b-1]$. 
        
        In Example~\ref{exa:June23dd1}, the $(d+1)\times (d+1)= 2\times 2$ minors of $[a,m]\times [b,n] = [2,3]\times [2,4]$ are obtained from $r=4, c=5, k=3$; the bigrassmannian permutation generating this rank condition is $w= 1 \ 2 \ 3\ 6\ 4\ 5\ 7$.
        \item[(III)] The $(d+1)\times(d+1)$ minors of $[1, m]\times[b, b']\subseteq\overline{Z}$ are the $(k+1)\times(k+1)$ minors of $[1, r]\times[1, c]$ for
        \[r = m+b-1, c = b', k = b+d-1.\]
        Here, the point is that any nonzero $(d+1)\times (d+1)$ minor of $[1, m]\times[b, b']$ uses the $1$'s in the block $[m+1,m+b-1]\times [1,b-1]$, producing the desired minors. Although larger nonzero minors may also appear, each such minor is generated by the $(d+1)\times(d+1)$ minors via cofactor expansion. Cofactor expansion also shows that the lead term of each larger minor is divisible by the leading term of one of the desired $(d+1)\times (d+1)$ minors. Consequently,
        the defining minors are Gr\"obner under $\prec$ if and only if the set of all these $(d+1)\times (d+1)$ minors are Gr\"obner for that term order, and so these additional minors are harmless to our argument.
        
        In Example~\ref{exa:June23dd1}, the $(d+1)\times (d+1)= 2\times 2$ minors of $[1,m]\times [b,b']=[1,3]\times [2,3]$ come from $r=4, c=3, k=2$; the bigrassmannian permutation is $w=1\ 2 \ 4 \ 5 \ 3 \ 6 \ 7$.
        \item[(IV)] The $(d+1)\times(d+1)$ minors of $[a, a']\times[1, n]\subseteq\overline{Z}$ are the $(k+1)\times(k+1)$ minors of $[1, r]\times[1, c]$ for 
        \[r = a', c = n+a-1, k = a+d-1.\]
        This time, any nonzero $(d+1)\times (d+1)$ minor uses the $1$'s of the block $[1,a-1]\times [n+1,n+a-1]$. The same notice from (III) about
        ``additional minors'' similarly applies in this case.
        
         In Example~\ref{exa:June23dd1}, the $(d+1)\times (d+1)= 2\times 2$ minors of $[a,a']\times [1,n]=[2,3]\times [1,4]$ arise when $r=3, c=5, k=2$; the bigrassmannian permutation is $w=1\ 2 \ 6 \ 3 \ 4 \ 5 \ 7$.      
    \end{enumerate}

    Thus any Knutson determinantal ideal $I$ is a sum of Kazhdan--Lusztig ideals $I_{v, w}$ for our fixed choice of $v$ and various bigrassmannian $w$. Hence the minors generating $I$ form a Gr\"obner basis under $\prec$ by Theorem~\ref{thm:minorscover}(II), so $I$ is a Gr\"obner-determinantal ideal.
\end{proof}

The next three examples show strict containments of families of determinantal ideals:
\[
\text{Knutson\ determinantal} 
\subsetneq\ \text{Gr\"obner-determinantal} 
\subsetneq\ \text{bicrystalline\ under\ } \antidiag. 
\]

\begin{example}[Knutson determinantal]\label{ex:KnutsonIdeal}
    Let $m = n = 6$ and let $I\subseteq \mathbb{C}[{\sf Mat}_{m,n}]$ be generated by the determinants of the submatrices $[1,2]\times [1,2],[1,4]\times [1,4]$ of $Z$. $I$ is a Knutson determinantal ideal and is 
    $(\{0,2,4,6\},\{0,2,4,6\},\prec)$-bicrystalline.
\end{example}

\begin{example}[Gr{\"o}bner-determinantal, not Knutson determinantal]\label{ex:grobneridealnotknutson}
    Let $m=n=6$ and let $I\subseteq \mathbb{C}[{\sf Mat}_{m,n}]$ be generated by the determinants of the submatrices $[1,2]\times [1,2], [2,5]\times [2,5], [5,6]\times [5,6]$ of $Z$. $I$ is not a Knutson determinantal ideal. However, the generators form a 
    $\prec$-Gr{\"o}bner basis, so $I$ is a Gr\"obner-determinantal ideal and is therefore $(\{0,1,2,4,5,6\},\{0,1,2,4,5,6\}, \prec)$-bicrystalline.
\end{example}

\begin{example}[Bicrystalline under $\prec$, not Gr\"obner-determinantal]
    Let $m=n=6$ and let $I\subseteq \mathbb{C}[{\sf Mat}_{m,n}]$ be generated by the determinants of the submatrices $[3,4]\times [3,4], [2,5]\times [2,5]$ of $Z$. $I$ is not Knutson determinantal. Moreover, the defining generators do not form a Gr{\"o}bner basis. Indeed, the reduced Gr{\"o}bner basis for $I$ under $\prec$ has lead terms:
    $$z_{34}z_{43}=\left[\begin{smallmatrix}
        0 & 0 & 0 & 0 & 0 & 0\\
        0 & 0 & 0 & 0 & 0 & 0\\
        0 & 0 & 0 & 1 & 0 & 0\\
        0 & 0 & 1 & 0 & 0 & 0\\
        0 & 0 & 0 & 0 & 0 & 0\\
        0 & 0 & 0 & 0 & 0 & 0
    \end{smallmatrix}\right], \; z_{24}z_{35}z_{43}z_{52}=\left[\begin{smallmatrix}
        0 & 0 & 0 & 0 & 0 & 0\\
        0 & 0 & 0 & 1 & 0 & 0\\
        0 & 0 & 0 & 0 & 1 & 0\\
        0 & 0 & 1 & 0 & 0 & 0\\
        0 & 1 & 0 & 0 & 0 & 0\\
        0 & 0 & 0 & 0 & 0 & 0
    \end{smallmatrix}\right], \; z_{24}z_{33}z_{35}z_{44}z_{52}=\left[\begin{smallmatrix}
        0 & 0 & 0 & 0 & 0 & 0\\
        0 & 0 & 0 & 1 & 0 & 0\\
        0 & 0 & 1 & 0 & 1 & 0\\
        0 & 0 & 0 & 1 & 0 & 0\\
        0 & 1 & 0 & 0 & 0 & 0\\
        0 & 0 & 0 & 0 & 0 & 0
    \end{smallmatrix}\right]$$
    However, Algorithm~\ref{thebigalg} shows $I$ is 
     $(\{0,1,2,4,5,6\},\{0,1,2,4,5,6\},\prec)$-bicrystalline.
\end{example}

We end this section by returning to the subfamily of matrix Schubert varieties.

\begin{example}[A spherical matrix Schubert variety]\label{ex:2143HilbSeries}
    Consider the ideal $J$ in $R = {\mathbb C}[{\sf Mat}_{3,3}]$ generated by $z_{11}$ and the $3\times 3$ minor. This is a Schubert determinantal ideal, as defined in Example~\ref{exa:MSV}.
   We will show that 
\[c^{R/J}_{(\lambda^{(1)},\lambda^{(2)}|\mu^{(1)},\mu^{(2)})}\in\{0,1\}.\] 
Our point is that, with some additional analysis, Theorem~\ref{thm:LRrule} allows one to explicitly classify when each value is attained; we refer
 to
\eqref{eqn:2143formula} below.
The multiplicity-freeness of $R/J$ has geometric significance: the corresponding matrix Schubert variety is \emph{spherical}, i.e., it has a dense orbit of a Borel subgroup of $L_{\bf I}\times L_{\bf J}$. In upcoming work, the first two authors classify spherical matrix Schubert varieties. This is analogous to the classification of spherical Schubert varieties (\cite{Hodges.Yong, GHY1, Can.Saha, GHY2}). 
    
    It is convenient  to instead first study the ideal $I$ generated by the $3\times 3$ minor and $z_{33}$. 
 Here the Levi datum is 
$(\mathbf{I},\mathbf{J}) = (\{0, 2, 3\}, \{0, 2, 3\})$. In this case, Theorem~\ref{thm:LRrule} asserts
    $$
    c^{R/I}_{(\lambda^{(1)},\lambda^{(2)} | \mu^{(1)},\mu^{(2)})} = \#\left\{z^M\in\! {\mathrm{Std}}_\prec I: {\sf RSK}(M)\in
    {\mathcal{LR}}({\bf I},{\bf J}, (\lambda^{(1)},\lambda^{(2)}), (\mu^{(1)},\mu^{(2)}))\right\}.
    $$

We use two claims.

\begin{claim}\label{claim:2143std}
    Suppose $M\in{\sf Mat}_{m,n}(\mathbb{Z}_{\geq 0})$, ${\sf RSK}(M) = (P,Q)$, and the common shape of $P$ and $Q$ is $\nu$.
    The monomial
     $z^{M}\in {\operatorname{Std}}_\prec I$ if and only if 
     \begin{equation}\label{eqn:2143shape}
        \nu = (\max(\lambda_1^{(1)},\mu_1^{(1)}),\min(\lambda_2^{(1)}+\lambda_1^{(2)},\mu_2^{(1)}+\mu_1^{(2)})).
     \end{equation}
\end{claim}
\begin{proof}[Proof of Claim~\ref{claim:2143std}]
    If ${\sf RSK}^{-1}((P,Q)) = M$ such that $z^M$ is a standard monomial, then 
    \[M_{3,3} = 0 \text{\ and ${\sf width}_{[1,3]}({\sf row}(M)) < 3$.}\] 
    Since ${\sf width}_{[1,3]}({\sf row}(M)) < 3$, Example~\ref{exa:tabwidth} and Proposition~\ref{prop:goodschensted} imply that $\ell(\nu) < 3$. Moreover, since $M_{3,3}=0$, either $P$ or $Q$ must not have a $3$ in its top row (this is most easily seen using the ``orthodox" description of ${\sf RSK}^{-1}$ in, e.g., \cite[Section 4.1]{Fulton} or \cite[Section 7.11]{ECII}). Equivalently, either $\nu_1 = \lambda_1^{(1)}$ or $\nu_1 = \mu_1^{(1)}$. In fact, $\nu_1 = \max(\lambda_1^{(1)},\mu_1^{(1)})$, as $\nu \supseteq \lambda^{(1)}$ and $\nu \supseteq \mu^{(1)}$. Since \[|\nu| = \lambda_1^{(1)} + \lambda_1^{(2)} + \lambda_2^{(1)} = \mu_1^{(1)} + \mu_1^{(2)} + \mu_2^{(1)}\] and $\ell(\nu)\leq 2$, $\nu$ must be of the form~\eqref{eqn:2143shape}. So, $(P,Q)$ must satisfy the conditions above.

    Conversely, assume that $(P,Q)$ satisfies the conditions given above. Since $\ell(\nu) < 3$, 
    ${\sf width}_{[1,3]}({\sf row}(M)) < 3$. 
    Moreover, since either $\nu_1 = \lambda_1^{(1)}$ or $\nu_1 = \mu_1^{(1)}$,  $M_{3,3}$ equals $0$ (again using the ``orthodox" description of ${\sf RSK}^{-1}$). This proves the claim.
\end{proof}

\begin{claim}\label{claim:2143multfreeweak}
    Let $(P,Q)\in {\mathcal{LR}}({\bf I},{\bf J}, (\lambda^{(1)},\lambda^{(2)}), (\mu^{(1)},\mu^{(2)}))$ such that $z^M\in{\sf Std}_\prec(I)$, where ${\sf RSK}(M) = (P,Q)$. Then $P,Q$ are the unique semistandard Young tableaux of shape~\eqref{eqn:2143shape} such that $P|_{[1,2]} = T_{\lambda^{(1)}}[1,2]$, $Q|_{[1,2]} = T_{\mu^{(1)}}[1,2]$, and the remaining boxes of $P,Q$ are filled with $3$s.
\end{claim}
\begin{proof}
    Since $M$ is standard, by Claim~\ref{claim:2143std}, the common shape $\nu$ of $P$ and $Q$ is of the form~\eqref{eqn:2143shape}. By Remark~\ref{rem:supersemistandard}, if $(P,Q)\in {\mathcal{LR}}({\bf I},{\bf J}, (\lambda^{(1)},\lambda^{(2)}), (\mu^{(1)},\mu^{(2)}))$, both $P|_{[1,2]}$ and $Q|_{[1,2]}$ must be supersemistandard. Moreover, there is precisely one way to fill the skew shape $\nu/\lambda^{(1)}$ or $\nu/\mu^{(1)}$ with either $\lambda^{(2)}_1$-many $3$s or $\mu^{(2)}_1$-many $3$s, respectively, so $P,Q$ must be unique. 
\end{proof}

From Claim~\ref{claim:2143multfreeweak} and Theorem~\ref{thm:LRrule}, it follows that
\begin{equation}\label{eqn:2143formula}
c^{R/I}_{\underline\lambda|\underline\mu} = \begin{cases}
	1 & \text{if }\exists (P, Q)\in {\mathcal{LR}}({\bf I},{\bf J}, \underline\lambda, \underline\mu)\text{ of shape }\nu\text{ as in \eqref{eqn:2143shape},}\\
	0 & \text{otherwise.}
\end{cases}
\end{equation}

This formula can be made completely explicit from the description of $(P, Q)$ in Claim~\ref{claim:2143multfreeweak}. For any given $\underline\lambda$ and $\underline\mu$, $c^{R/I}_{\underline\lambda|\underline\mu}=1$ if and only if the corresponding $\nu$ is a partition shape, $\lambda^{(1)},\mu^{(1)}\subseteq\nu$, and the skew shapes $\nu/\lambda^{(1)}$ and $\nu/\mu^{(1)}$ are horizontal strips.

We are now done since $I$ and $J$ are related by a $180$-degree rotation of ${\sf Mat}_{m, n}$,
 \[c^{R/J}_{(\lambda^{(2)},\lambda^{(1)}|\mu^{(2)},\mu^{(1)})}=c^{R/I}_{(\lambda^{(1)},\lambda^{(2)}|\mu^{(1)},\mu^{(2)})}.\]
\end{example}

\section{GL-stable, in-KRS ideals are bicrystalline}\label{sec:in-RSK}

Thus far, we have viewed ${\sf RSK}$ as a combinatorial tool for associating representation theory (namely, a crystal structure) to the monomial basis of $\complexes[{\sf Mat}_{m, n}]$. However, one can also view ${\sf RSK}$ as a linear operator on $\complexes[{\sf Mat}_{m, n}]$, transitioning between the monomial basis and an alternate basis: the \emph{standard bitableaux} of \cite{DRS} (for more on this perspective, see \cite{Stelzer.Yong}). We describe this alternate basis, which extends the construction of the basis of $V_{\lambda}(k)$ given in Section~\ref{subsec:reptheoryprelims} (specifically
\eqref{eqn:Vkdef}).

\begin{definition}
    Given two increasing sequences of integers 
    \[R = (r_1<\dots<r_d)\subseteq[m] \text{\ and $C = (c_1<\dots<c_d)\subseteq[n]$,}\] 
    associate the determinant
    \[\Delta_{[R|C]} = \begin{vmatrix}z_{r_1c_1} & \dots & z_{r_1c_d}\\
    \vdots & \ddots & \vdots\\
    z_{r_dc_1} & \dots & z_{r_dc_d}\end{vmatrix}\in\complexes[{\sf Mat}_{m, n}].\]
    Let $(P, Q)$ be a pair of fillings (not necessarily semistandard) of a partition shape $\lambda$ that are strictly increasing along columns, where $P$ uses entries from $[m]$ and $Q$ uses entries from $[n]$. The \emph{bitableau} $[P|Q]\in\complexes[{\sf Mat}_{m, n}]$ is the product
    \[[P|Q] = \prod_{k=1}^{\ell(\lambda')}\Delta_{[P'_k|Q'_k]},\]
    where $P'_k$ is the set of integers in the $k$th column of $P$. If $P$ and $Q$ are semistandard Young tableaux, then $[P|Q]$ is called a \emph{standard bitableau} of shape $\lambda$. 
\end{definition}

\begin{example}
    An example of a (non-standard) bitableau is
    \[\left[\ytabb{1 & 2\\ 4 & 3}\bigg\vert \ytabb{ 1 & 3\\ 2 & 4}\right] = \begin{vmatrix}z_{11} & z_{12}\\ z_{41} & z_{42}\end{vmatrix}\begin{vmatrix}z_{23} & z_{24}\\ z_{33} & z_{34}\end{vmatrix},\]
    and an example of a standard bitableau is
    \[\left[\ytabb{1 & 1\\ 3}\bigg\vert\ytabb{1 & 3\\ 2}\right] = \begin{vmatrix}z_{11} & z_{12}\\ z_{31} & z_{32}\end{vmatrix}\begin{vmatrix}z_{13}\end{vmatrix}.\]
\end{example}

The following result is a consequence of the \emph{straightening law} of P.~Doubilet--G.~C.~Rota--J.~Stein
\cite[Theorem 8.1]{DRS}.

\begin{theorem}\cite[Theorem 8.3]{DRS}
    The standard bitableaux $[P|Q]$ form a vector space basis for $\complexes[{\sf Mat}_{m, n}]$.  
\end{theorem}

We extend the definition of ${\sf RSK}$ as follows: if $M\in {\sf Mat}_{m,n}({\mathbb Z}_{\geq 0})$ and ${\sf RSK}(M)=(P,Q)$ then the operator 
\[{\sf RSK}:{\mathbb C}[{\sf Mat}_{m,n}]\to {\mathbb C}[{\sf Mat}_{m,n}]\] 
is defined by linearly extending the map
\[{\sf RSK}(z^M):=[P|Q].\]
Hence ${\sf RSK}^{-1}([P|Q])=z^M$. 

In \cite{BC}, W.~Bruns--A.~Conca consider classes of ideals with bases of standard bitableaux. 

\begin{definition}[{\cite[Definition~4.4]{BC}}]
An ideal $I$ possessing a vector space basis $\mathcal{B}$ of standard bitableaux is \emph{in-KRS} if 
\[\mathrm{span}_{\mathbb C}({\sf RSK}^{-1}({\mathcal B}))=\mathrm{span}_{\mathbb C}(\{{\sf RSK}^{-1}([P|Q]): [P|Q]\in {\mathcal B}\})={\mathrm{init}}_{\antidiag}I.\footnote{Our conventions about ${\sf RSK}$ differ from those
of \cite{BC}, so their equivalent definition is in terms of $\diag$.}\]
\end{definition}

\begin{example}(Classical determinantal ideals, revisited again)\label{ex:theOGex}
Example~\ref{exa:det} provides an example of an in-KRS ideal. By \cite[Corollary 3.4.2]{BCRV}, 
$I_k$ has a vector space basis ${\mathcal B}$ consisting of
standard bitableaux $[P|Q]$ whose shape $\lambda$ contains a column of length $k$ (i.e., $[P|Q]$ is divisible by some $k\times k$ minor of $Z$).
By Proposition~\ref{prop:goodschensted}, it follows that
\[{\sf RSK}\inv(\mathcal{B}) = \{z^M:{\sf width}_{[1, m]}({\sf row}(M))\geq k\}.\]
Since the $k\times k$ minors of $Z$ form a Gr\"obner basis for $I_k$ under $\antidiag$ \cite[Theorem 1]{Sturmfels}, another application of Proposition~\ref{prop:goodschensted} shows that 
\[\init_\antidiag(I_k) = \mathrm{span}_{\mathbb C}(\{z^M:{\sf width}_{[1, m]}({\sf row}(M))\geq k\}).\] 
We conclude that $I_k$ is an in-KRS ideal.
\end{example}

From the results of \cite{BC}, we deduce another family of bicrystalline ideals.

\begin{definition}
    For partitions $\lambda$ and $\mu$, write $\mu\supseteq\lambda$ if $\mu_i\geq\lambda_i$ for all $i$ (i.e., the Young diagram for $\lambda$ is a subset of the Young diagram for $\mu$), and write $\mu\geq \lambda$ if for all $k$, 
\[\sum_{i\geq k}\mu_i\geq \sum_{i\geq k}\lambda_i.\] 
\end{definition}
The definition immediately implies that if $\lambda\supseteq\mu$ then $\lambda\geq\mu$.

\begin{example}
The complete set of partitions $\leq \ydiags{2,2}$ is given by:
\[\left\{\, \emptyset, \, \ydiag{1},\, \ydiag{2},\, \ydiag{1,1}, \, \ydiag{3}, \, \ydiag{2,1}, \,\ydiag{4},\, \ydiag{3,1},\,
\ydiag{2,2}\,\right\}.\]
There are infinitely many $\mu\geq \ydiags{2,2}\, $, among them are: $\ydiags{3,2},\, \ydiags{4,2},\, \ydiags{5,2},\ldots$.
\end{example}

\begin{definition}\label{def:inKRS}
    For a partition $\lambda$, let 
    \[I^{(\lambda)}\subseteq\complexes[{\sf Mat}_{m, n}]\] 
    be the ideal spanned as a vector space by all standard bitableaux of shape $\mu\geq\lambda$.
\end{definition}

\begin{theorem}[{\cite{BC, BV}}]\label{thm:shapeideals}
    Let $I^{(\lambda)}\subseteq\complexes[{\sf Mat}_{m, n}]$. Then:
    \begin{enumerate}
        \item[(I)]\cite[Proposition 11.2]{BV} $I^{(\lambda)}$ is the ideal generated by all (not necessarily standard) bitableaux of shape $\lambda$.
        \item[(II)]\cite[Theorem 5.2]{BC} An ideal $I\subseteq\complexes[{\sf Mat}_{m, n}]$ is $\mathbf{GL}$-stable and has a vector space basis of standard bitableaux if and only if
        \begin{equation}\label{eqn:July24xyz}
       I=I^{\left(\theta^{(1)}\right)}+I^{\left(\theta^{(2)}\right)}+\cdots+I^{\left(\theta^{(f)}\right)}
       \end{equation} 
       for some partitions 
        $\theta^{(i)}$, $1\leq i\leq f$.
    \end{enumerate}
\end{theorem}

\begin{example}
The determinantal ideal $I_k$ from Example~\ref{ex:theOGex} is by definition generated
by $k\times k$ minors of a generic $m\times n$ matrix $Z$. Therefore, by Theorem~\ref{thm:shapeideals}(I), $I_k=I^{(1^k)}$. Clearly, the basis ${\mathcal B}$
described in the example is the same thing as the one in Definition~\ref{def:inKRS}.
\end{example}
 
Moreover, Example~\ref{ex:theOGex} is an instance of a more general conclusion:

\begin{theorem}[{\cite[Corollary 5.3]{BC}}]\label{thm:detinKRS}
    Every $\mathbf{GL}$-stable ideal  with a vector space basis of standard bitableaux (i.e., every sum of ideals $I^{(\lambda)}$) is an in-KRS ideal.
\end{theorem}

We are now ready to state the main conclusion of this section:
\begin{proposition}\label{cor:inKRSbicrystal}
    Every $\mathbf{GL}$-stable ideal $I$ with a vector space basis of standard bitableaux is 
    $(\{0,m\},\{0,n\},\antidiag)$-bicrystalline.
\end{proposition}
\begin{proof}
    By Theorem~\ref{thm:shapeideals}(II), the ideal $I$ can be expressed as a sum \eqref{eqn:July24xyz} of ideals $I^{\left(\theta^{(i)}\right)}$ for some finite list of partitions $\{\theta^{(i)}\}_{1 \leq i\leq f}$. By Theorem~\ref{thm:detinKRS}, $I$ and all the $I^{\left(\theta^{(i)}\right)}$ are in-KRS. Viewing \eqref{eqn:July24xyz} as a sum of vector spaces with bases of standard bitableaux implies that
\[\init_\antidiag I = \sum_{i}\init_\antidiag I^{\left(\theta^{(i)}\right)}.\]
Thus by Proposition~\ref{prop:gbuniontestsets} it suffices to prove the statement for a single ideal $I^{(\lambda)}$. Definition~\ref{def:inKRS} and Theorem~\ref{thm:detinKRS} together show that ${\mathrm{init}}_{\antidiag}I^{(\lambda)}$ is the set of monomials of the form $z^M={\sf RSK}\inv([P|Q])$ for some standard bitableau $[P|Q]$ of some shape $\mu\geq\lambda$. Now, suppose $\varphi$ is a bicrystal operator from \eqref{eqn:thefour}. Suppose 
$\varphi(M)\neq \varnothing$. 
Then by Proposition~\ref{prop:pullback},
\[z^{\varphi(M)}={\sf RSK}\inv([P'|Q'])\in {\mathrm{init}}_{\antidiag}(I^{(\lambda)})\]
for some standard bitableau $[P'|Q']$ of the same shape $\mu\geq\lambda$. Thus 
\[\varphi(M)\in\monos_\antidiag I.\] 
Since $M$ and $\varphi$ were arbitrary, we have shown that $\monos_{\antidiag}I^{(\lambda)}$
is closed under the bicrystal operators. Hence $I^{(\lambda)}$ is $\mathbf{GL}$-bicrystalline by Proposition~\ref{prop:reformulate}.
\end{proof}

We conclude that powers and symbolic powers of determinantal ideals are bicrystalline with respect to $\antidiag$:

\begin{corollary}\label{cor:detpowerbicrystal}
The $r$-th ordinary power $I_k^r$ and $r$-th symbolic power $I_k^{(r)}$ of the determinantal ideal $I_k$ are 
$(\{0,m\},\{0,n\},\antidiag)$-bicrystalline.
\end{corollary}
\begin{proof} 
By \cite[Proposition~3.5.8 and Theorem~4.3.9]{BCRV} and \cite[Theorems 3.5.2 and 4.3.6]{BCRV}, respectively, $I_k^r$ and $I_k^{(r)}$ are in-KRS ideals with vector space bases of standard bitableaux. The result then follows from Proposition~\ref{cor:inKRSbicrystal}.
\end{proof}

\begin{example}[Application of Theorem~\ref{thm:LRrule} to $I_k^r$ and $I_k^{(r)}$]\label{exa:detpower2}
The results \cite[Theorem~4.3.9]{BCRV} and \cite[Theorem~4.3.6]{BCRV} describe Gr\"obner bases for 
$I_k^r$ and $I_k^{(r)}$ under $\antidiag$, making Theorem~\ref{thm:LRrule} effective for them.
For $I_k^r$, the Gr\"obner basis is given by (possibly non-standard) bitableaux of shapes $\lambda$ with $kr$ boxes and at most $k$ columns. Meanwhile, for $I_k^{(r)}$, the Gr\"obner basis is given by (possibly non-standard) bitableaux of shapes $\mu$ where each column has length at least $k$ and the sum of row lengths $\mu_k+\mu_{k+1}+\dots$ is exactly $r$.

It is immediate from Theorem~\ref{thm:LRrule} and the \emph{Cauchy identity} (that is, \eqref{eqn:detcauchy} for $k>m,n$)
that the $GL_m\times GL_n$ character expansions for $R/I_k^r$ and $R/I_k^{(r)}$ are multiplicity-free
sums of the form $s_\lambda\otimes s_{\lambda}$ over appropriately restricted collections of partitions $\lambda$.

The descriptions of the Gr\"obner bases explain the difference of $s_{\ydiags{1,1,1}}\otimes s_{\ydiags{1,1,1}}$
in the two character expansions from Example~\ref{exa:powers}. The Gr\"obner basis for $I_2^2$ consists of all bitableaux of shape $\ydiags{2, 1, 1}$ or $\ydiags{2, 2}$, whereas the Gr\"obner basis for $I_2^{(2)}$ consists of all bitableaux of shape $\ydiags{1, 1, 1}$ or $\ydiags{2, 2}$. This makes it clear that the highest weight matrix $M=\left[\begin{smallmatrix} 0 & 0 & 1 \\ 0 & 1 & 0 \\ 1 & 0 & 0\end{smallmatrix}\right]$ (corresponding to the lead term of the standard bitableau of shape $\ydiags{1, 1, 1}$) lies in $\monos_{\antidiag}I_2^{(2)}$ but not $\monos_{\antidiag}I_2^2$. Thus $s_{\ydiags{1,1,1}}\otimes s_{\ydiags{1,1,1}}$ appears in the character expansion of $\complexes[{\sf Mat}_{3, 3}]/I_2^{2}$, but not in the character expansion of $\complexes[{\sf Mat}_{3, 3}]/I_2^{(2)}$. One can verify that any other highest-weight matrix $M'$ lies in $\monos_{\antidiag} I_2^2$ if and only if it lies in $\monos_{\antidiag} I_2^{(2)}$, so the two character expansions otherwise agree.
\end{example}

\begin{example}[Powers of a Schubert determinantal ideal] The square and symbolic square of the Schubert determinantal ideal $J$ from Example~\ref{ex:2143HilbSeries} agree: $J^2=J^{(2)}$. Now, \[{\mathrm{init}}_{\antidiag}(J^2)=\langle z_{11}^2, \, z_{11}\cdot z_{31}z_{22}z_{13}, \, (z_{31}z_{22}z_{13})^2\rangle.\]
This ideal is $({\bf I},{\bf J},\antidiag)$-bicrystalline. Its expansion is not multiplicity-free. The reader may verify using Theorem~\ref{thm:LRrule} that
\[c^{R/J^2}_{\ydiags{1}, \, \ydiags{1}| \, \ydiags{1}, \, \ydiags{1}}=2.\]

Which Schubert determinantal ideals $I$ satisfy $I^d=I^{(d)}$, 
either for a given $d$, or for all $d$? As explained in \cite[Section 3]{Fulton:duke}, these ideals are indexed by permutations 
$w\in {\mathfrak S}_n$. For $n\leq 5, d=2$ there are four cases
where $I_w^2\neq I_w^{(2)}$, namely, 
\[w\in \{14523,\, 15423, \, 14532, \,15432\}.\] 
The first of these is a classical
determinantal ideal; see \cite{Trung, BC03}  for discussion of the problem in that case.
\end{example}

The ideals $I^{(\lambda)}$ are particularly well-behaved, but they are not the only $\mathbf{GL}$-stable ideals in $\complexes[{\sf Mat}_{m, n}]$. The following class of ideals are perhaps the most natural $\mathbf{GL}$-stable ideals to consider, although they usually fail to have bases of standard bitableaux. 

\begin{definition}\label{def:StevenSamideals}
    For a partition $\lambda$, let 
    \[I_\lambda\subseteq\complexes[{\sf Mat}_{m, n}]\] 
    be the (necessarily $\mathbf{GL}$-stable) ideal generated by the $\lambda$-isotypic (and irreducible) component $V_\lambda\boxtimes V_\lambda$ of $\complexes[{\sf Mat}_{m, n}]$, i.e., the smallest $\mathbf{GL}$-stable ideal containing the highest-weight bitableau of shape $\lambda$. 
\end{definition}

The ideal $I_\lambda$ for $\lambda = \ydiags{2}$ appeared back in Example~\ref{ex:diagbicrystal}, where we saw that it was bicrystalline under $\diag$ but not $\antidiag$. In general, we pose the following problem:
\begin{problem}\label{problem:bic}
Classify the set of $\lambda$ such that  $I_\lambda$ is $(\{0,m\},\{0,n\},\prec)$-bicrystalline for some $\prec$.
 \end{problem}

Solving Problem~\ref{problem:bic} is non-trivial, since we do not know of explicit generating sets for these ideals. This situation contrasts with the finite generating sets given for the ideals $I^{(\lambda)}$ in Theorem~\ref{thm:shapeideals}(I). Corollary~\ref{cor:Ilambdaideals} below gives a simple solution for two infinite families:

\begin{theorem}[{\cite[Proposition 11.15]{BV}}]
    As a vector space, 
    \[I_\lambda = \bigoplus_{\mu\supseteq\lambda} (V_\mu\boxtimes V_\mu).\]
\end{theorem}

\begin{corollary}\label{cor:Ilambdaideals}
    If $\lambda$ is a rectangle with $\ell(\lambda) = m$ or a single column, then 
    $I_\lambda\subseteq\complexes[{\sf Mat}_{m, n}]$ 
    is $(\{0,m\},\{0,n\},\antidiag)$-bicrystalline.
\end{corollary}
\begin{proof}
    If $\lambda$ is a rectangle with $\ell(\lambda) = m$ or a single column, then for partitions $\mu$ with $\ell(\mu)\leq m$ it is easy to see that $\mu\geq\lambda$ if and only if $\mu\supseteq\lambda$. Thus 
    $I_\lambda = I^{(\lambda)}$, 
    implying the result via Proposition~\ref{cor:inKRSbicrystal}.
\end{proof}

\begin{example}
Let $\lambda$ be a rectangle with $\ell(\lambda)=m$. By the proof of Corollary~\ref{cor:Ilambdaideals}, $I_{\lambda}$ is spanned as a vector space by all standard bitableaux of shape $\mu\supseteq \lambda$. Thus $I_\lambda$ is in-KRS in this case by Theorem~\ref{thm:detinKRS}, so $\init_{\antidiag} I_\lambda$ is the span of all monomials $z^M$ such that ${\sf RSK}(z^M)$ has shape $\mu\supseteq \lambda$. In other words, for any fixed $\mu$, the unique matrix $M_{\mu}$ such that 
\[{\sf RSK}(M_{\mu})\in {\mathcal{LR}}(\{0,m\},\{0,n\},\mu,\mu)\] 
lies outside $\monos_\antidiag I_\lambda$ if and only if ${\mu}\not\supseteq \lambda$. By Theorem~\ref{thm:LRrule}, we obtain a character formula for $\complexes[{\sf Mat}_{m, n}]/I_\lambda$ in this case:
$$\sum_{\mu\not\supseteq\lambda} s_{\mu}(x_1,\ldots,x_m)s_{\mu}(y_1,\ldots,y_n).$$ 
\end{example}

\begin{remark}
The \emph{jet scheme} of a determinantal variety is another source of $GL_m\times GL_n$ invariant ideals, although it does not lie in ${\sf Mat}_{m,n}$; see \cite{Yuen} and references therein.  For example, let $\mathfrak{J}$ be the second jet scheme of the determinantal variety $\mathfrak{X}_1\subseteq{\sf Mat}_{2, 2}$. Then the $GL_2\times GL_2$ character for $\complexes[\mathfrak{J}]$ is not multiplicity-free and begins
\begin{multline}\nonumber
1+3s_{\ydiags{1}}\otimes s_{\ydiags{1}}+3s_{\ydiags{1,1}}\otimes s_{\ydiags{1,1}}+3s_{\ydiags{1,1}}\otimes s_{\ydiags{2}}+ 3s_{\ydiags{2}}\otimes s_{\ydiags{1,1}}+6s_{\ydiags{2}}\otimes s_{\ydiags{2}}+
10s_{\ydiags{3}}\otimes s_{\ydiags{3}}+8s_{\ydiags{3}}\otimes s_{\ydiags{2,1}}\\
+8s_{\ydiags{2,1}}\otimes s_{\ydiags{3}}+10s_{\ydiags{2,1}}\otimes s_{\ydiags{2,1}}+\cdots.
\end{multline}
What is a rule for these coefficients? Does a version of the GCS thesis apply to such ideals? In general, the Gr\"obner bases
for these ideals are not well-understood.
\end{remark}

\section{Non-commutative ideals and the GCS thesis}\label{sec:non-comm}
Although in this paper we are primarily concerned with $R = \mathbb{C}[{\sf Mat}_{m,n}]$, we can apply the Gr\"obner crystal structure principle to any ring $R$ with an action of some semisimple linear algebraic group $G$ and a standard basis endowed with some crystal structure. In this section, we consider $GL_n$ acting on 
$R = T(\mathbb{C}^n) =  \mathbb{C}\oplus\mathbb{C}^n\oplus(\mathbb{C}^n\otimes \mathbb{C}^n)\oplus\cdots$, 
the tensor algebra of $\mathbb{C}^n$, with its natural grading. 

First, we observe that there exists a natural crystal structure indexing a basis of $R$. Let $v_1,\ldots,v_n$ denote the standard basis of $\mathbb{C}^n$ as a vector space. The $k$th graded piece of $R$ has a basis given by the pure tensors $v_{i_1}\otimes \cdots \otimes v_{i_k}$. Moreover, $R$ carries an action of $GL_n$ induced by its action on the standard representation $\mathbb{C}^n$; i.e., given $g\in GL_n$ and $v_{i_1}\otimes\cdots\otimes v_{i_k}\in T(\mathbb{C}^n)$,
    $$g\cdot (v_{i_1}\otimes\cdots\otimes v_{i_k}) = gv_{i_1}\otimes \cdots \otimes gv_{i_k}.$$ 

As a $GL_n$-representation, $\mathbb{C}^n$ has an associated crystal structure obtained by associating each $v_i$ with the tableau $\ytabs{
        i
    }$ 
and using the usual crystal structure on tableaux from Definition~\ref{def:tabbicops}. This crystal structure may be extended to tensors $\ytabs{
        i_1
}\otimes \cdots \otimes \ytabs{
        i_k
}$ using Kashiwara's tensor product operation on crystals (\cite{Kashiwara, Kashiwara2, Kashiwara3}), yielding a crystal for $T(\mathbb{C}^n)$.\footnote{This crystal is, up to change of conventions, the same as the word crystal of Example~\ref{exa:wordcrystal}. Bump and Schilling in~\cite[Section 2.3]{BS} give an excellent explanation of tensor product crystals and their relationship with word crystals; note, however, that their conventions are opposite ours.}

While ideals $I\subseteq R$ do not have Gr{\"o}bner bases in the sense of Section~\ref{sec:bicrystalline} (as $R$ is non-commutative), they may have Gr{\"o}bner--Shirshov bases. Gr{\"o}bner--Shirshov bases are analogues of Gr{\"o}bner bases in the non-commutative setting that share many of the same properties (see~\cite{BokutChen} and~\cite{Bokut} for precise definitions). In particular, the standard monomials of any Gr{\"o}bner--Shirshov basis for $I$ form a vector space basis for $R/I$.

Using the machinery of Gr{\"o}bner--Shirshov bases, we can extend the notion of a bicrystalline ideals to a vastly more general setting. 

\begin{definition}\label{def:GCS}
    Let $A$ be a unital, associative algebra such that $A\cong T(\mathbb{C}^n)/I$ for some homogenenous ideal $I\subseteq T(\mathbb{C}^n)$. Assume $A$ has an action of a reductive linear algebraic group $G$. Let $J\subseteq A$ be a homogenenous ideal such that $G\cdot J = J$. Assume further that $J$ has a Gr{\"o}bner--Shirshov basis with respect to a term order $\prec$, with associated set of standard monomials $\mathfrak{M}$. A \emph{Gr{\"o}bner crystal structure (GCS)} for the triple $(A,J,\prec)$ is a normal $G$-crystal $\mathfrak{B}$ on the monomials of $A$ such that $\mathfrak{M}$ forms a normal subcrystal of $\mathfrak{B}$. We say that $J$ is $(G, \prec)$-\emph{crystalline} for $\mathfrak{B}$ if $\mathfrak{B}$ is a GCS for $(A,J,\prec)$.   
\end{definition}

\begin{remark}
	Bicrystalline ideals in $\complexes[{\sf Mat}_{m, n}]$ are a special case of crystalline ideals.
	Let $U\cong\complexes^m$ and $W\cong\complexes^n$ be vector spaces, let $V = U\boxtimes W$, and let $I\subseteq T(V^*)$ be the two-sided, homogeneous ideal $\langle v\otimes v'-v'\otimes v:v, v'\in V^*\rangle$. 
	Then $T(V^*)/I\cong\mathrm{Sym}(V^*)$, which we identify with $\complexes[{\sf Mat}_{m, n}]$ as in Section~\ref{subsec:reptheoryprelims}. 
	A Gr\"obner--Shirshov basis for an ideal $J\subseteq T(V^*)/I$ in this special case is the same thing as a Gr\"obner basis (see~\cite[Chapter 1]{Bokut}). 
\end{remark}

\begin{example}[Crystalline non-commutative ideal]\label{ex:altcrystalline}
    Let $R = T(\complexes^n)$ and let
    \[I = \langle v_i\otimes v_j + v_j\otimes v_i\rangle\subseteq R\] 
    be the two-sided, homogeneous ideal defining the exterior algebra $\Lambda(\mathbb{C}^n)$. $I$ is stable under the action of $GL_n$. In this example, the set of standard monomials (using lexicographic order) is (\cite[pg. 333]{BokutChen}): $${\operatorname{Std}}_{\prec} I = \{v_{i_1}\otimes\cdots \otimes v_{i_k} \; | \; i_1 < i_2 < \cdots < i_k\}.$$ The set of tableaux associated with ${\operatorname{Std}}_\prec I$, namely, $$\left\{\ytabb{
        i_1
    }\otimes \cdots \otimes \ytabb{
        i_k
    } \; | \; i_1 < i_2 < \cdots < i_k\right\},$$ together with the empty symbol, is closed under the crystal operators described above. So, the crystal described above is a GCS for the triple $(R,I,\prec)$. Using this fact, we recover the character formula for $\Lambda(\mathbb{C}^n)$:
    $$s_{\emptyset}(x_1,\ldots,x_n) + s_{\ydiags{1}}(x_1,\ldots,x_n)+s_{\ydiags{1,1}}(x_1,\ldots,x_n)+\cdots+s_{(1^n)}(x_1,\ldots,x_n).$$
    Figure~\ref{fig:altcrystal} depicts part of the crystal structure for $T(\mathbb{C}^3)$, where the elements of 
    ${\operatorname{Std}}_{\prec} I$ are highlighted in blue. The fact that ${\operatorname{Std}}_{\prec} I$ is closed under the crystal operators described above corresponds to the fact that every connected component of the crystal in Figure~\ref{fig:altcrystal} consists either entirely of blue elements or entirely of black elements.
\end{example}

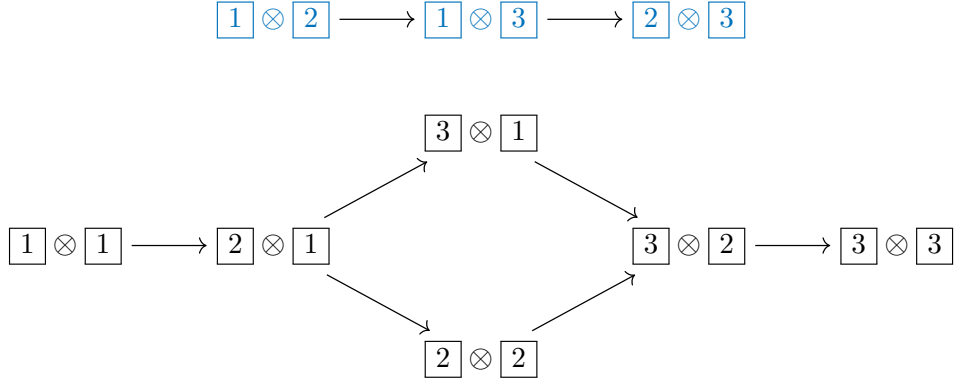
\begin{figure}
    \centering
\[\begin{tikzcd}[cramped]
	& {\color{RoyalBlue}\ytabb{1}\otimes \ytabb{2}} & {\color{RoyalBlue}\ytabb{1}\otimes \ytabb{3}} & {\color{RoyalBlue}\ytabb{2}\otimes \ytabb{3}} \\
	&& {\ytabb{3}\otimes \ytabb{1}} \\
	{\ytabb{1}\otimes \ytabb{1}} & {\ytabb{2}\otimes \ytabb{1}} && {\ytabb{3}\otimes \ytabb{2}} & {\ytabb{3}\otimes \ytabb{3}} \\
	&& {\ytabb{2}\otimes \ytabb{2}}
	\arrow[from=1-2, to=1-3]
	\arrow[from=1-3, to=1-4]
	\arrow[from=2-3, to=3-4]
	\arrow[from=3-1, to=3-2]
	\arrow[from=3-2, to=2-3]
	\arrow[from=3-2, to=4-3]
	\arrow[from=3-4, to=3-5]
	\arrow[from=4-3, to=3-4]
\end{tikzcd}\]
    \caption{Part of the crystal for $T(\mathbb{C}^3)$, with elements of $\mathrm{Std}_\prec I$ in blue.}
    \label{fig:altcrystal}
\end{figure}

We now shift to an example of an ideal $I\subseteq R = T(\mathbb{C}^d)$ that is \emph{not} crystalline for a particular crystal structure for $R$ and $GL_n$ (here $d$ is \emph{not} necessarily equal to $n$). 

Instead of considering the crystal structure on $R$ induced by the standard representation $\mathbb{C}^d$, we instead consider the crystal structure on $R$ induced by a crystal for some Schur module $V_{\lambda}$ of dimension $d$. That is, we take the crystal structure for the degree $\ell$ component of $R$ to be the crystal structure for $V_{\lambda}^{\otimes \ell}$.

\begin{example}[Crystal for $T(V_{(k)})$]\label{ex:plethysmcrystal}
    Consider the $GL_2$ representation \[\mathrm{Sym}^k(\mathbb{C}^2)\cong T(V_{(k)}),\] 
    the Schur module indexed by a row of length $k$. Set \[R = T(V_{(k)});\] 
    $R$ carries an action of $GL_2$ induced by the $GL_2$-action on $V_{(k)}$. We identify the basis elements of ${\mathrm{Sym}}^k(\mathbb{C}^2)$ with homogeneous polynomials of degree $k$ in variables $z_1,z_2$. These basis elements 
    \[v_i = z_1^{i}z_2^{k-i}\] 
    have an associated crystal structure by identifying $v_i$ with the tableau $P_i$ of shape $(k)$ filled with $i$-many $1$'s and $(k-i)$-many $2$'s. For instance, if $k = 3$, the element $v_1 = z_1^1z_2^2$ is associated with the tableau $P_1 = \ytabs{1 & 2 & 2}$. Kashiwara's crystal tensor product induces a crystal structure on tensors $P_{i_1}\otimes \cdots \otimes P_{i_k}$, yielding a crystal for $T(V_{(k)})$.
\end{example}

With respect to this new crystal structure on $R$, the ideal $I$ of Example~\ref{ex:altcrystalline} is \emph{not} in general crystalline, as demonstrated by the following example.

\begin{example}[Plethysm]
    Let $R$ be as in Example~\ref{ex:plethysmcrystal}. As in Example~\ref{ex:altcrystalline}, let 
    \[I = \langle v_i\otimes v_j + v_j\otimes v_i\rangle.\] 
The question is whether the crystal structure descends to $R/I$. Now,   
    $I$ is stable under the action of $GL_2$ described in Example~\ref{ex:plethysmcrystal}. As before, the set $${\operatorname{Std}}_{\prec} I = \{v_{i_1}\otimes\cdots \otimes v_{i_k} \; | \; i_1 < i_2 < \cdots < i_k\}$$ is a basis for 
    \[R/I\cong \Lambda({\mathrm{Sym}}^k(\mathbb{C}^2)).\] 
    However, $\mathrm{Std}_\prec I$ is \emph{not} in general closed under the crystal operators of Example~\ref{ex:plethysmcrystal}. Figure~\ref{fig:plethysm} depicts a portion of the crystal for $T({\mathrm{Sym}}^3(\mathbb{C}^2))$, where the elements of ${\operatorname{Std}}_{\prec} I$ are highlighted in blue. The fact that 
    ${\operatorname{ Std}}_{\prec} I$ is not closed under the crystal operators described above corresponds to the fact that there exist connected components of the crystal in Figure~\ref{fig:plethysm} that contain both blue and black elements.
\end{example}

\begin{remark}
    As explained by \'A. Guti\'errez in~\cite{Gutierrez.Alvaro}, the problem of finding a crystal structure on $\Lambda({\mathrm{Sym}^k(\mathbb{C}^2)})$ is closely related to a conjecture of Stanley in~\cite[pg. 182]{StanleySymChains} that, for fixed $n,m$, the sub-poset $L(n,m)$ of Young's lattice beneath the rectangular partition $(n^m)$ admits a rank-symmetric saturated chain decomposition. More precisely, an explicit solution to Stanley's conjecture would yield a crystal structure on $\Lambda({\mathrm{Sym}^k(\mathbb{C}^2)})$. Guti\'errez's work shows that sometimes the converse also holds; he constructs crystal structures on $\Lambda^{\ell}({\mathrm{Sym}^k(\mathbb{C}^2)})$ for $\ell\leq 4$ and $k$ arbitrary which yield symmetric chain decompositions for $L(n,m)$, where $n\leq 4$ and $m$ is arbitrary. 
\end{remark}

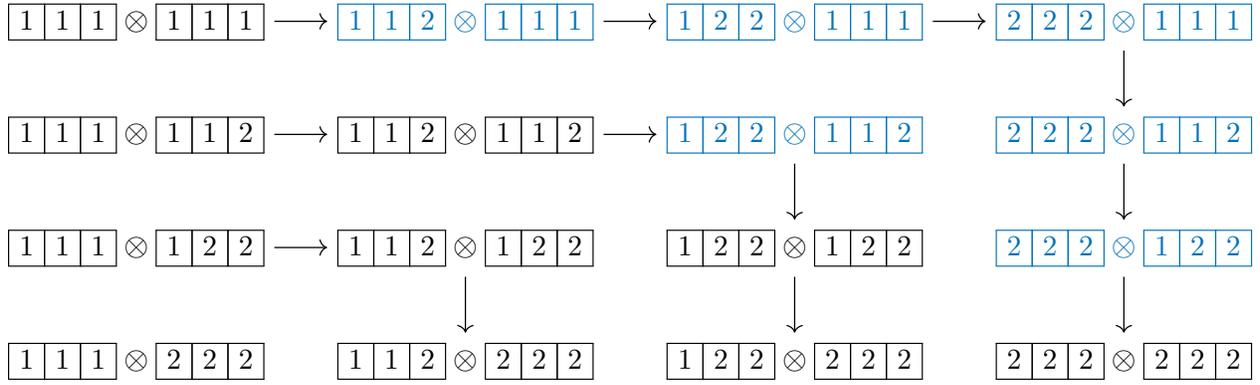
\begin{figure}
    \centering
\[\begin{tikzcd}[ampersand replacement=\&,cramped,column sep=1.7em]
	{\ytabb{1 & 1 & 1}\otimes \ytabb{1 & 1 & 1}} \& {\color{RoyalBlue}\ytabb{1&1&2}\otimes\ytabb{1&1&1}} \& {\color{RoyalBlue}\ytabb{1&2&2}\otimes\ytabb{1&1&1}} \& {\color{RoyalBlue}\ytabb{2&2&2}\otimes\ytabb{1&1&1}} \\
	{\ytabb{1&1&1}\otimes\ytabb{1&1&2}} \& {\ytabb{1&1&2}\otimes\ytabb{1&1&2}} \& {\color{RoyalBlue}\ytabb{1&2&2}\otimes\ytabb{1&1&2}} \& {\color{RoyalBlue}\ytabb{2&2&2}\otimes\ytabb{1&1&2}} \\
	{\ytabb{1&1&1}\otimes\ytabb{1&2&2}} \& {\ytabb{1&1&2}\otimes\ytabb{1&2&2}} \& {\ytabb{1&2&2}\otimes\ytabb{1&2&2}} \& {\color{RoyalBlue}\ytabb{2&2&2}\otimes\ytabb{1&2&2}} \\
	{\ytabb{1&1&1}\otimes\ytabb{2&2&2}} \& {\ytabb{1&1&2}\otimes\ytabb{2&2&2}} \& {\ytabb{1&2&2}\otimes\ytabb{2&2&2}} \& {\ytabb{2&2&2}\otimes\ytabb{2&2&2}}
	\arrow[from=1-1, to=1-2]
	\arrow[from=1-2, to=1-3]
	\arrow[from=1-3, to=1-4]
	\arrow[from=1-4, to=2-4]
	\arrow[from=2-1, to=2-2]
	\arrow[from=2-2, to=2-3]
	\arrow[from=2-3, to=3-3]
	\arrow[from=2-4, to=3-4]
	\arrow[from=3-1, to=3-2]
	\arrow[from=3-2, to=4-2]
	\arrow[from=3-3, to=4-3]
	\arrow[from=3-4, to=4-4]
\end{tikzcd}\]\
    \caption{Part of the crystal for $T({\operatorname{Sym}}^3(\mathbb{C}^2))$.}
    \label{fig:plethysm}
\end{figure}

Computing the character of $\Lambda({\mathrm{Sym}}^k(\mathbb{C}^2))$ is a special case of the \emph{plethysm problem}: given $\lambda$ and $\mu$, what is the character of $\mathbb{S}^{\mu}(\mathbb{S}^{\lambda}(V))$ for a complex vector space $V$ (where $\mathbb{S}^{\mu}$ is the \emph{Schur functor} indexed by $\mu$)? Using Weyl's construction (see, e.g.,~\cite[Lecture 6]{FH}), the Schur functor $\mathbb{S}^{\lambda}(V)$ for a complex vector space $V$ and partition $\lambda\vdash d$ is defined to be the image 
\[c_{\lambda}\cdot V^{\otimes d}\subseteq V^{\otimes d}\] 
of the \emph{Young symmetrizer} $c_{\lambda}$ associated to $\lambda$. For partitions $\lambda$ and $\mu$, define the ideal $I_{\lambda,\mu}$ by \[
I_{\lambda,\mu} = \langle\ker(c_{\mu}\cdot V_{\lambda}^{\otimes d})\rangle \subseteq T(V_{\lambda}) = R.
\] The $d$-th graded component of $R/I_{\lambda,\mu}$ is precisely the representation $\mathbb{S}^{\mu}(\mathbb{S}^{\lambda}(V))$. 

\begin{question}
    Which $I_{\lambda,\mu}$ have Gr{\"o}bner--Shirshov bases?
\end{question} 
When $I_{\lambda,\mu}$ has a Gr{\"o}bner--Shirshov basis, we may ask the following:
\begin{question}
    Which $I_{\lambda,\mu}$ are crystalline?
\end{question}

\section{Concluding remarks}\label{sec:concluding}

In D.~Hilbert's work proving the existence of finite generators for the algebra of invariants ${\Bbbk}[V]^G$ 
of a finite (or compact) group acting on a vector space $V$ over a field ${\Bbbk}$ of characteristic $0$, he introduced
the notion of finite free resolutions of standard graded modules \[M=\bigoplus_{t\geq 0} M_t\] over a polynomial ring $S$. 
These resolutions imply that the Hilbert function 
\[f_M(t)=\dim_{{\Bbbk}}(M_t)\] 
of $M$ agrees with the Hilbert polynomial $p_M(t)$, at least for $t$ \emph{sufficiently large}.\footnote{Note that for coordinate rings of varieties arising from representation theory, ``sufficiently large'' often means $t\geq0$ or $t\geq1$. Examples of this (near) \emph{Hilbertian property} include all Schubert determinantal ideals, but also many other examples, as explained in \cite{Hilbertian}.} As explained in Example~\ref{exa:stdgraded}, the torus character of $M$ is precisely an encoding of its Hilbert function as a generating series (the Hilbert series). The perspective of this paper replaces the torus by a spectrum of ``fat tori'', which is to say, Levi groups. This gives a conceptual bridge between the 
Hilbert function values $f_M(t)$ and constants from combinatorial representation theory, such as the Littlewood--Richardson coefficients. In the latter situation, one has polynomiality properties of sequences of Littlewood--Richardson coefficients \cite{Derksen}; that is, for fixed $\lambda,\mu,\nu$, the sequence $c_{t\lambda,t\mu}^{t\nu}$ for $t\geq 0$ is interpolated by a polynomial in $t$.  This rhyme of themes, and the hint of a unifying theory in it, philosophically motivates us to study Levi spectra of coordinate rings.

We demonstrated our GCS thesis in precise terms for the class of bicrystalline ideals. Our results, which include the Gr\"obner-determinantal, Knutson determinantal, and ${\bf GL}$-stable in-KRS ideals, cover many of the motivating examples mentioned in Section~\ref{subsec:motivating}. Extending Example~\ref{exa:quiver} to quiver loci for any non-equioriented $A_n$-quiver is work in progress between I.~Cavey, A.~Hardt, and the third author. Example~\ref{exa:veronese} is to be explained in a vastly larger context (see, e.g., \cite{Marberg}) by work of the first author. In other examples, we show hints of potential applications of our methods to varieties of 
interest such as matrix matroid ideals, ASM varieties, and matrix Hessenberg varieties. This list is by no means exhaustive. 

In the bicystalline cases, Theorem~\ref{thm:LRrule} provides a uniform formula for the irreducible multiplicities of a Levi-stable ideal. 
The combinatorics of our multiplicity formula plays a key role in a forthcoming classification of spherical matrix Schubert varieties by the first and second authors; see Example~\ref{ex:2143HilbSeries}. In various instances, one can attempt to relate our formula to the combinatorial data indexing an ideal. See our questions about matrix Hessenberg varieties (Example~\ref{exa:Hessenberg}) and ASM varieties (Example~\ref{exa:ASM}), for instance. 
It would be interesting to explain such combinatorics in some generality. 

We believe that in many cases, the irreducible multiplicities we consider have ``concavity'' properties or semigroup structure in analogy with the classical Littlewood--Richardson coefficients, e.g., \cite{Knutson.Tao, Okounkov, StHuh}. Generalizations of such properties have been examined within the classical representation-theory context (see, e.g., \cite{Kumar} and the references therein). Our paper suggests a venue for potential generalizations in a different direction.

There are longstanding challenges that motivate our central thesis. One notable case arises from studying the character of ${\mathfrak{gl}}_n = {\sf Mat}_{n,n}$ under the conjugation action of $GL_n$.\footnote{Actually, one keeps track of an additional dilation action to avoid infinite dimensional weight spaces, giving rise to a $GL_n\times T_1$-character.} For a partition $\lambda$, let ${\mathcal O}_\lambda$ denote the nilpotent orbit consisting of matrices in ${\mathfrak{gl}}_n$ with Jordan form of type $\lambda$, and let 
${\overline{\mathcal O}}_\lambda$ be its Zariski closure. The study of the $GL_n$-module structure of $\mathbb{C}[{\overline{\mathcal{O}}}_\lambda]$ is an old problem in geometric representation theory. There is no known general description of a standard basis for $\mathbb{C}[{\overline{\mathcal{O}}}_\lambda]$ (although generators \cite{Weyman:nilpotent}, and even minimal generators \cite{Huang}, are known for the ideal defining the orbit closure). The Gr\"obner-theoretic question is deeply intertwined with the representation theory. See \cite{Shimozono.Weyman1, Shimozono.Weyman2} for more on this problem.

Finally, there is the question of finding a ``nice'' basis for ${\mathbb C}[{\sf Mat}_{m,n}]/I$ when $I$ is $({\bf I},{\bf J},\prec)$-bicrystalline. ${\mathrm{Std}}_{\prec}I$ is ``nice'' because it has a crystal structure. It is a basis of $T_m\times T_n$-weight vectors: each standard monomial spans a one-dimensional irreducible torus representation inside $\complexes[{\sf Mat}_{m, n}]/I$. Now, one might ask for some basis that ``respects'' the $L_{\bf I}\times L_{\bf J}$ action rather than merely the $T_m\times T_n$-action. However, even the bitableau basis of ${\mathbb C}[{\sf Mat}_{m,n}]$ (see Section~\ref{sec:in-RSK}) does \emph{not} respect the {\bf GL}-action in a completely analogous manner: specifically, no subset spans the irreducible subrepresentation $V_{\lambda}\boxtimes V_{\lambda}$. We interpret our main results as indication that the monomial basis for $\complexes[{\sf Mat}_{m, n}]$ should be considered ``nice'' from not only the standpoint of combinatorial commutative algebra, but also that of representation theory.

\appendix
\section{An elementary proof of Theorem~\ref{thm:minorscover}(I)}\label{theappendix}
In this appendix we provide an alternate proof of Theorem~\ref{thm:minorscover}(I) (originally proved by A. Knutson in \cite[Theorem 7]{Knutson}), giving an elementary combinatorial argument to derive the lead terms of the basic minors $\Delta_v^{(k)}$ in the specialized matrix $Z_v$. Our proof is by induction, using the following operation in the inductive step.

\begin{definition}
    The \emph{$i$-deletion} of a permutation $v\in\mathfrak{S}_n$ is the permutation 
    ${\sf del}_i(v)\in \mathfrak{S}_{n-1}$
     obtained by deleting row $i$ and column $v(i)$ from the permutation matrix $M_v$.
\end{definition}

For all $i, j\in[n]$, let $\phi_{i,j}$ be the bijection
\[([n]\setminus\{i\})\times([n]\setminus\{j\})\xrightarrow{\phi_{i, j}}[n-1]\times[n-1]\]
given by deleting row $i$ and column $j$ from an $n\times n$ table of positions. Explicitly, $\phi_{i,j}$ maps
\[(a, b) \mapsto \begin{cases} 
      (a, b) &  a < i, b < j,\\
      (a-1, b) & a > i, b < j, \\
      (a, b-1) & a < i, b > j, \\
      (a-1, b-1) & a > i, b > j.
   \end{cases}\]

\begin{lemma}\label{lemma:delbij}
Let $v\in\mathfrak{S}_n$, fix $i'\in[n]$, and let $v' = {\sf del}_{i'}(v)$. Let $\phi := \phi_{i', v(i')}$.
\begin{itemize}
\item[(I)] For all $a\in[n]\setminus\{i'\}$ and $b\in[n]\setminus\{v(i')\}$ we have
   \[{\sf drift}_{v'}(\phi(a, b)) = \begin{cases}
   {\sf drift}_v(a, b) - 1, & a > i' \text{ or } b > v(i'),\\
   {\sf drift}_{v}(a, b), & \text{ else.}
   \end{cases}\]
   \item[(II)] The map $\phi$ preserves the antidiagonal lexicographic order $\prec$ on $Z_v$: if $z_{a, b}\prec z_{a', b'}$ in $Z_v$ with $a, a'\neq i$ and $b, b'\neq v(i)$, then $z_{\phi(a, b)}\prec z_{\phi(a', b')}$ in $Z_{v'}$.
   \item[(III)] For all $k\in[n]$ such that ${\sf drift}_v(i', v(i'))<k$, the map $\phi$ restricts to a bijection
   \[\{(i, v(i)):i\in[n]\setminus\{i'\}, {\sf drift}_v(i, v(i)) < k\}\to\{(i, v'(i)):i\in[n-1], {\sf drift}_{v'}(i, v'(i)) < k-1\}.\]
   \end{itemize}
\end{lemma}
\begin{proof}
    (I) and (II) are immediate from the definitions. For (III), it is also immediate that $\phi$ restricts to a bijection
    \[\{(i, v(i)):i\in[n]\setminus\{i'\}\}\to\{(i, v'(i)):i\in[n-1]\}.\]
    We wish to show that the further restriction of $\phi$ to 
    \[\{(i, v(i)):i\in[n]\setminus\{i'\}, {\sf drift}_v(i, v(i)) < k\}\] 
    has the claimed codomain; i.e., that if ${\sf drift}_v(i, v(i)) < k$ for some $i\in[n]\setminus\{i'\}$, then ${\sf drift}_{v'}(\phi(i, v(i))) < k-1$. This is immediate from part (I) if $i > i'$ or $v(i) > v(i')$. Otherwise, if $i<i'$ and $v(i)<v(i')$, then part (I) states that \[{\sf drift}_{v'}(\phi(i, v(i))) = {\sf drift}_v(i, v(i)).\] 
    In this case, the definition of the rank function implies that 
    \[r_v(i', v(i')) - r_v(i, v(i)) \leq (i'-i)+(v(i')-v(i))-1.\]
    It follows that
    \[
        {\sf drift}_v(i', v(i'))-{\sf drift}_v(i, v(i)) = (i'-i)+(v(i')-v(i))-(r_v(i', v(i'))-r_v(i, v(i))) \geq 1.
    \]
    Thus \[{\sf drift}_{v'}(\phi(i, v(i))) = {\sf drift}_v(i, v(i))\leq {\sf drift}_v(i', v(i')) -1 < k-1,\] 
    so the restriction of $\phi$ in the lemma statement has the claimed codomain. Part (I) immediately implies that this restriction is surjective, and since $\phi$ is injective by definition we conclude that its restriction is a bijection.
\end{proof}

\begin{lemma}\label{lemma:driftbound}
    For any $v\in\mathfrak{S}_n$, if ${\sf drift}_v(i, j) < k$, then $\max\{i, j\}\leq k$.
\end{lemma}
\begin{proof}
    By definition of the rank function, $r_v(i, j)\leq \min\{i, j\}$. Thus 
    \[{\sf drift}_v(i, j) := i+j-1-r_v(i, j) \geq \max\{i, j\}-1.\]
    It follows that if $k > {\sf drift}_v(i, j)$, then $\max\{i, j\}\leq k$ as claimed.
\end{proof}

Expanding the minor $\Delta_v^{(k)}$ using the Leibniz formula, we will index terms of minors $\Delta_v^{(k)}$ by permutations $w\in \mathfrak{S}_k$. We say that $w$ or the corresponding term $\mathbf{m}_w$ \emph{uses} a position $(i, j)$ or the entry $Z_v^{(k)}(i, j)$ if $w(i) = j$.

\begin{lemma}\label{lemma:pigeon}
    Let $\mathbf{m}_w$ be a nonzero term of a basic minor $\Delta_v^{(k)}$ ($w\in\mathfrak{S}_k$), and let $i\in[k]$ be such that $w(i)\neq v(i)$ and ${\sf drift}_v(i, v(i))<k$. Then $\mathbf{m}_w$ uses a position strictly southeast of $(i, v(i))$ (i.e., there exists some $i' > i$ such that $w(i') > v(i')$). 
\end{lemma}
\begin{proof}
    Since $\mathbf{m}_w$ is a nonzero term of $\Delta_v^{(k)}$, each entry $(i, w(i))$ in $Z_v^{(k)}$ must be nonzero. Figure~\ref{fig:pigeon} illustrates the situation, using the assumption that $w(i)\neq v(i)$.
    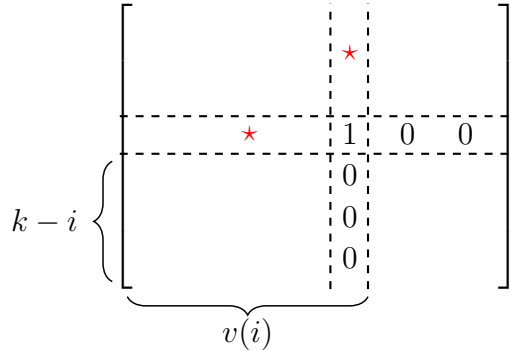
\begin{figure}[ht]
    \centering
    \begin{tikzpicture}[brace1/.style={decoration={brace,amplitude=7pt},
            decorate}]
    \matrix (m)[
        matrix of math nodes,
        nodes in empty cells,
        left delimiter={[},
        right delimiter={]},
        cells = {anchor = center}
        ]{
        && && && && && && \\
        && && && && \textcolor{red}{\star} && && \\
        && && && && && && \\
        && && && && && && \\
        && && \textcolor{red}{\star} && && 1 && 0 && 0 \\
        && && && && 0 && && \\
        && && && && 0 && && \\
        && && && && 0 && && \\
        };
    \draw[dashed, thick] (-2.6, 0.4) -- (2.6, 0.4);
    \draw[dashed, thick] (-2.6, -0.1) -- (2.6, -0.1); 
    \draw[dashed, thick] (0.2, -1.9) -- (0.2, 1.9);
    \draw[dashed, thick] (0.7, -1.9) -- (0.7, 1.9);
    
    \draw[brace1, line width=0.2mm] (-2.7,-1.8) -- (-2.7,-0.2);
    \draw (-3.6, -1) node {$k-i$};
    \draw[brace1, line width=0.2mm] (0.7,-2) -- (-2.5,-2);
    \draw (-0.9, -2.5) node {$v(i)$};
    \end{tikzpicture}
    \caption{The situation of Lemma~\ref{lemma:pigeon}. Stars indicate positions used by $w$.}\label{fig:pigeon}
    \end{figure}
    Since $w$ is a permutation, $\mathbf{m}_w$ uses a position in each of the $k-i$ rows $>i$. \emph{A priori}, at most $v(i)$ of these positions are in columns $\leq v(i)$. However, $\mathbf{m}_w$ only uses one position from each column, and it cannot use positions directly south of $1$'s in $Z_v$ (otherwise $\mathbf{m}_w = 0$). Accounting for the positions already used in Figure~\ref{fig:pigeon} and the $r_v(i, v(i))-1$ additional $1$'s that lie strictly northwest of $(i, v(i))$, there are only $v(i)-1-r_v(i, v(i))$ remaining columns $\leq v(i)$ in which these $k-i$ stars can go. But we assumed that
    \[{\sf drift}_v(i, v(i)) = i+v(i)-1-r_v(i, v(i)) < k \implies v(i)-1-r_v(i,v(i))<k-i.\]
    Thus $\mathbf{m}_w$ must use a position strictly southeast of $(i, v(i))$, as claimed.
\end{proof}

\begin{proof}[Proof of Theorem~\ref{thm:minorscover}(I)]
    We argue (for all $k$ simultaneously) by induction on the number of positions $i\in[n]$ such that ${\sf drift}_v(i, v(i)) < k$. The base case is when there are no such positions, so ${\sf drift}_v(i, v(i))\geq k$ always. Since 
    \[{\sf drift}_v(i, v(i)) < i+v(i)-1\text{\ for all $i$,}\] 
    this implies that $Z_v^{(k)}$ is generic weakly northwest of its main antidiagonal. Hence in this base case the lead term of $\Delta_v^{(k)}$ is its antidiagonal term, which is equal to $\prod_{(i, j)\in D_v(k)}z_{ij}$.

    For the inductive step, suppose there exists some $i\in[n]$ such that ${\sf drift}_v(i, v(i)) < k$. Then $i$ and $v(i)$ must in fact lie in $[k]$ by Lemma~\ref{lemma:driftbound}, so $(i, v(i))$ is the position of a $1$ in $Z_v^{(k)}$. We make the following key claim:
    \begin{claim}\label{claim:induct}
        If $\mathbf{m}_w$ is a nonzero term of $\Delta_v^{(k)}$ avoiding $(i, v(i))$, then there exists a nonzero term $\mathbf{m}_{w'}$ in $\Delta_v^{(k)}$ using $(i, v(i))$ such that $\mathbf{m}_w\prec\mathbf{m}_{w'}$.
    \end{claim}
    Claim~\ref{claim:induct} implies that the lead term of ${\Delta}_v^{(k)}$ uses $(i, v(i))$, which implies by Lemma~\ref{lemma:delbij}(II) that the lead term of $\Delta_v^{(k)}$ equals the lead term of $\Delta_{v'}^{(k-1)}$ for $v = {\sf del}_i(v)$ (up to relabelling of the variables by $\phi\inv_{i, v(i)}$). By Lemma~\ref{lemma:delbij}(III),
    \[|\{j\in[k-1] : {\sf drift}_{v'}(j, v'(j)) < k-1\}| = |\{j\in[k] : {\sf drift}_{v}(j, v(j)) < k\}|-1.\]
    We may therefore apply the inductive hypothesis to see that
    \[\init_\prec(\Delta_{v'}^{(k-1)}) = \prod_{(a', b')\in D_{v'}(k-1)}z_{a'b'}.\]
    Lemma~\ref{lemma:delbij}(I) and (III) together show that $\phi_{i, v(i)}$ gives a bijection between the variables indexed by positions $(a, b)\in D_v(k)$ and $(a', b')\in D_{v'}(k-1)$. This completes the proof of Theorem~\ref{thm:minorscover}(I), once we prove Claim~\ref{claim:induct}.
    \begin{proof}[Proof of Claim~\ref{claim:induct}]
    We begin with a simplifying assumption. If $Z_v^{(k)}$ contains a $1$ in its first row or column, then \emph{every} nonzero term of $\Delta_v^{(k)}$ uses that $1$ and Claim~\ref{claim:induct} is vacuously true. We henceforth assume that $Z_v^{(k)}$ has no $1$ in its first row or column. 

    Fix $(i, v(i))$ to be the westernmost $1$ in $Z_v^{(k)}$ such that ${\sf drift}_v(i, v(i)) < k$. Let $\mathbf{m}_w$ be a nonzero term of $\Delta^{(k)}_v$ avoiding $(i, v(i))$. Note that 
    \[r_v(i, v(i)) = 1\] 
    by construction, so our assertion that 
    ${\sf drift}_v(i, v(i))<k$ 
    merely states that $(i, v(i))$ lies weakly northwest of the main antidiagonal of $Z_v^{(k)}$. Our choice of $i$ ensures that all $1$'s in $Z_v^{(k)}$ west of $(i, v(i))$ lie strictly southeast of the main antidiagonal. 

    By Lemma~\ref{lemma:pigeon}, $\mathbf{m}_w$ uses a nonzero entry of $Z^{(k)}_v$ in some position $(i', j')$ strictly southeast of $(i, v(i))$. If $\mathbf{m}_w$ uses multiple such positions, choose $(i', j')$ so that $i'$ is minimal. Let 
    $a = i'-i+1$ and let 
    \[(j_1,\dots, j_a) := (w(i), w(i+1),\dots, w(i')).\] 
    Then the restriction of $w$ to the interval $[i, i']$ is a pattern embedding of $ua\in {\mathfrak S}_a$ 
    obtained by adjoining $u\in\mathfrak{S}_{a-1}$ with~$a$. Let $w'\in\mathfrak{S}_k$ be the permutation agreeing with $w$, except its restriction to $[i, i']$ embeds the longest permutation 
    $a(a-1)\cdots 21\in\mathfrak{S}_a$ 
    (where the pattern embedding $\mathfrak{S}_a\into\mathfrak{S}_k$ is still given by $\ell\mapsto j_\ell$). Figure~\ref{fig:swap1} illustrates the construction of $\mathbf{m}_{w'}$ from $\mathbf{m}_w$.
    \begin{figure}
        \centering
        \begin{tikzpicture}
        \matrix (m1) at (0, 0) [
        matrix of math nodes,
        nodes in empty cells,
        left delimiter={[},
        right delimiter={]},
        cells = {anchor = center}
        ]{
        && && && && && && && && \\
        && && && && && && && && \\
        z && \textcolor{red}{\star} && z && z && z && 1 && 0 && 0 && 0\\
        z && z && \textcolor{red}{\star} && z && z && 0 && && && \\
        \textcolor{red}{\star} && z && z && z && && 0 && && && \\
        z && z && z && && \textcolor{red}{\star} && 0 && && && \\
        z && z && && && && 0 && && \textcolor{red}{\star} && \\
        z && && && && && 0 && && && \\
        };
        \draw[dashed, thick] (-3.6, 1.35) -- (3.6, 1.35);
        \draw[dashed, thick] (-3.6, 0.85) -- (3.6, 0.85); 
        \draw[dashed, thick] (0.52, -2.1) -- (0.52, 2.05);
        \draw[dashed, thick] (1.02, -2.1) -- (1.02, 2.05);
        \draw (0.2, -2.5) node {$\mathbf{m}_w$};

        \draw[->, thick] (4, 0) -- (5, 0);
        \draw (-4, -1.1) node {$i'$};
        \draw (-4, 1.12) node {$i$};
        \draw (0.8, 2.25) node {$v(i)$};
        \draw (-2.275, 2.25) node {$j_1$};
        \draw (-1.5, 2.25) node {$j_2$};
        \draw (-3.1, 2.25) node {$j_3$};
        \draw (0.05, 2.25) node {$j_4$};
        \draw (2.4, 2.25) node {$j'$};
        
        \matrix (m2) at (9, 0)[
        matrix of math nodes,
        nodes in empty cells,
        left delimiter={[},
        right delimiter={]},
        cells = {anchor = center}
        ]{
        && && && && && && && && \\
        && && && && && && && && \\
        z && z && z && z && z && 1 && 0 && \textcolor{red}{\star} && 0\\
        z && z && z && z && \textcolor{red}{\star} && 0 && && && \\
        z && z && \textcolor{red}{\star} && z && && 0 && && && \\
        z && \textcolor{red}{\star} && z && && && 0 && && && \\
        \textcolor{red}{\star} && z && && && && 0 && && && \\
        z && && && && && 0 && && && \\
        };
        \draw[dashed, thick] (5.4, 1.35) -- (12.6, 1.35);
        \draw[dashed, thick] (5.4, 0.85) -- (12.6, 0.85); 
        \draw[dashed, thick] (9.52, -2.1) -- (9.52, 2.05);
        \draw[dashed, thick] (10.02, -2.1) -- (10.02, 2.05);
        \draw (9.2, -2.5) node {$\mathbf{m}_{w'}$};
    \end{tikzpicture}
    \caption{The relationship between $\mathbf{m}_w$ and $\mathbf{m}_{w'}$.}\label{fig:swap1}
    \end{figure}

    By the minimality of $i'$, the positions used by $\mathbf{m}_{w'}$ in the row interval $[i+1, i']$ all lie weakly northwest of the main antidiagonal of $Z_v^{(k)}$. Thus every position used by $\mathbf{m}_{w'}$ is nonzero except for the $0$ at $(i, j')$. To construct a term avoiding this $0$, let $i_1$ satisfy $w'(i_1) = v(i)$. Note that $i_1 < i$, since $\mathbf{m}_w$ indexes a nonzero term of $\Delta_v^{(k)}$. If there is a $0$ in position $(i_1, j')$ of $Z_v^{(k)}$, then $v(i_1) < j'$ and there must be some $i_2$ such that $w'(i_2) = v(i_1)$ and $i_2 < i_1$. Iterating this procedure, we must eventually reach a position $(i_r, j')$ that is nonzero in $Z_v^{(k)}$ because we assumed that $Z_v^{(k)}$ has only variables in its first row. Now cyclically permute the rows $(i, i_1, i_2,\dots, i_r)$ of $M_{w'}$ to obtain a new permutation $w''$. Figure~\ref{fig:swap2} illustrates this construction.
    \begin{figure}
    \centering
        \begin{tikzpicture}
        \matrix (m1) at (0, 0) [
        matrix of math nodes,
        nodes in empty cells,
        left delimiter={[},
        right delimiter={]},
        cells = {anchor = center}
        ]{
        z && z && z && z && \textcolor{red}{\star} && z && z && z && z\\
        && && && \textcolor{red}{\star} && 1 && 0 && 0 && 0 && 0\\
        && && && && 0 && 1 && 0 && 0 && 0\\
        && && && && 0 && 0 && && && \\
        && \textcolor{red}{\star} && && 1 && 0 && 0 && 0 && 0 && 0\\
        && && && 0 && 0 && 0 && && &&\\
        && 1 && 0 && 0 && 0 && 0 && 0 && \textcolor{red}{\star} && 0\\
        && 0 && && 0 && 0 && 0 && && && \\
        };
        \draw[dashed, thick] (-3.65, -1.1) -- (3.65, -1.1);
        \draw[dashed, thick] (-3.65, -1.6) -- (3.65, -1.6); 
        \draw[dashed, thick] (-2.55, -2.3) -- (-2.55, 2.3);
        \draw[dashed, thick] (-2.05, -2.3) -- (-2.05, 2.3);
        \draw (0.2, -2.5) node {$\mathbf{m}_{w'}$};

        \draw[->, thick] (3.9, 0) -- (4.7, 0);
        \draw (-4, -1.35) node {$i$};
        \draw (-4, -0.25) node {$i_1$};
        \draw (-4, 1.45) node {$i_2$};
        \draw (-4, 1.95) node {$i_3$};
        \draw (-2.3, 2.5) node {$v(i)$};
        \draw (2.4, 2.5) node {$j'$};
        
        \matrix (m2) at (8.5, 0) [
        matrix of math nodes,
        nodes in empty cells,
        left delimiter={[},
        right delimiter={]},
        cells = {anchor = center}
        ]{
        z && z && z && z && z && z && z && \textcolor{red}{\star} && z\\
        && && && && \textcolor{red}{\star} && 0 && 0 && 0 && 0\\
        && && && && 0 && 1 && 0 && 0 && 0\\
        && && && && 0 && 0 && && && \\
        && && && \textcolor{red}{\star} && 0 && 0 && 0 && 0 && 0\\
        && && && 0 && 0 && 0 && && &&\\
        && \textcolor{red}{\star} && 0 && 0 && 0 && 0 && 0 && 0 && 0\\
        && 0 && && 0 && 0 && 0 && && && \\
        };
        \draw[dashed, thick] (4.9, -1.1) -- (12.15, -1.1);
        \draw[dashed, thick] (4.9, -1.6) -- (12.15, -1.6); 
        \draw[dashed, thick] (6.4, -2.3) -- (6.4, 2.3);
        \draw[dashed, thick] (5.9, -2.3) -- (5.9, 2.3);
        \draw (8.7, -2.5) node {$\mathbf{m}_{w''}$};
    \end{tikzpicture}
    \caption{The relationship between $\mathbf{m}_{w'}$ and $\mathbf{m}_{w''}$.}\label{fig:swap2}
    \end{figure}

    By construction, the resulting term $\mathbf{m}_{w''}$ of $\Delta^{(k)}_v$ is nonzero and uses position $(i, v(i))$. Moreover, the lexicographically-first difference between $\mathbf{m}_{w''}$ and $\mathbf{m}_w$ is that the variable $z_{i_r, j}$ occurs in the former term but not the latter. Thus 
    \[\mathbf{m}_w\prec \mathbf{m}_{w''},\]
    proving Claim~\ref{claim:induct}.
    \end{proof}
This completes the proof of the theorem.\end{proof}

\section*{Acknowledgements}
We thank Casey Appleton, Sara Billey, Anders Buch, Ian Cavey, Robert Donley, Alex Fink, Shiliang Gao, Rebecca Goldin, William Graham, Andrew Hardt, Daoji Huang, Eric Marberg,  Martha Precup, Brendon Rhoades, Rich\'ard Rim\'anyi, Melissa Sherman-Bennett, Mark Shimozono, Kyu Sato, and Steven Sam for comments about \cite{AAA} that provided stimulation for this work. Specifically, we thank Steven Sam for suggesting study of the bicrystalline condition as one on ideals rather than varieties, and other ideas  that broadened our view of the notion. We are grateful to Allen Knutson for remarks that led to the results presented in Section~\ref{sec:Knutsonideals} (and in the Appendix). 
This work benefited from a visit by AS and AY to the Institute of Advanced Study, Princeton in February 2025, and a visit by AY to the Institute of Mathematical Sciences, Chinese University of Hong Kong in March 2025. 
AS was supported by an NSF graduate fellowship. AY was supported by a Simons Collaboration grant. The authors were partially supported by an
NSF RTG in Combinatorics (DMS 1937241).


\begin{thebibliography}{99}

\bibitem{Abhyankar} Abhyankar, Shreeram Shankar. Combinatoire des tableaux de Young, varietes determinantielles et calcul de fonctions de Hilbert. Rend Sere. Mat. Univers. Politech. Torino 42 (1984), 65--88.

\bibitem{Almousa.Gao.Huang} Almousa, Ayah; Gao, Shiliang; Huang, Daoji. Standard monomials for positroid varieties. Preprint, 2024.  arXiv:2309.15384v2

\bibitem{BFZ}
Berenstein, Arkady; Fomin, Sergey; Zelevinsky, Andrei. Parametrizations of canonical bases and totally positive matrices. Adv. Math. 122 (1996), no. 1, 49--149. 

\bibitem{BZ}
Berenstein, Arkady; Zelevinsky, Andrei. Tensor product multiplicities, canonical bases and totally positive varieties. Invent. Math. 143 (2001), no. 1, 77--128.

\bibitem{Berget.Fink}
Berget, Andrew; Fink, Alex. Equivariant Chow classes of matrix orbit closures. Transform. Groups 22 (2017), no. 3, 631--643. 

\bibitem{BokutChen}
Bokut, L. A.; Chen, Yuqun. Gröbner-Shirshov bases and their calculation. Bull. Math. Sci. 4 (2014), no. 3, 325--395.

\bibitem{Bokut}
Bokut, Leonid; Chen, Yuqun; Kalorkoti, Kyriakos; Kolesnikov, Pavel; Lopatkin, Viktor. Gr{\"o}bner–Shirshov Bases: Normal Forms, Combinatorial and Decision Problems in Algebra. World Scientific Publishing, 2020.

\bibitem{BC}
Bruns, Winfried; Conca, Aldo. KRS and determinantal ideals, in {\it Geometric and combinatorial aspects of commutative algebra (Messina, 1999)}, 67--87, Lecture Notes in Pure and Appl. Math., 217, Dekker, New York.

\bibitem{BC03}
Bruns, Winfried; Conca, Aldo. Gr\"obner bases and determinantal ideals. Commutative algebra, singularities and computer algebra (Sinaia, 2002), 9--66, NATO Sci. Ser. II Math. Phys. Chem., 115, Kluwer Acad. Publ., Dordrecht, 2003.

\bibitem{BCRV}
Bruns, Winfried; Conca, Aldo; Raicu, Claudiu; Varbaro, Matteo. Determinants, Gr\"obner bases and cohomology. Springer Monographs in Mathematics. Springer, Cham, 2022, xiii+507.

\bibitem{BV}
Bruns, Winfried; Vetter, Udo. {\it Determinantal rings}, Monograf\'ias de Matem\'atica, 45, Inst. Mat. Pura Apl. (IMPA), Rio de Janeiro, 1988.

\bibitem{Buchberger}
Buchberger, Bruno. Ein Algorithmus zum Auffinden der Basiselemente des Restklassenringes nach einem nulldimensionalen Polynomideal, Ph.D. thesis, Innsbruck, 1965.

\bibitem{Buchsbaum}
Buchsbaum, David A.; Eisenbud, David. Generic free resolutions and a family of generically perfect ideals. Advances in Math. 18 (1975), no. 3, 245--301.

\bibitem{BS}
Bump, Daniel; Schilling, Anne. Crystal bases. Representations and combinatorics. World Scientific Publishing Co. Pte. Ltd., Hackensack, NJ, 2017. xii+279 pp. 

\bibitem{Caniglia}
Caniglia, L.; Guccione, J.A.; Guccione, J.J. Ideals of generic minors. Comm. Algebra 18 (1990), no. 8, 2633--2640.

\bibitem{ConcaV}
Conca, Aldo; Varbaro, Matteo. Square-free Gr\"obner degenerations. Invent. Math. 221 (2020), no. 3, 713--730. 

\bibitem{CLO}
Cox, David; Little, John; O'Shea, Donal. Using algebraic geometry. Graduate Texts in Mathematics, 185. Springer-Verlag, New York, 2005. xii+572 pp.

\bibitem{bicrystal2}
Danilov, V. I.; Koshevoi, G. A.. Arrays and the combinatorics of Young tableaux. Russ. Math. Surv 60 (2005), 269--334.

\bibitem{Strickland}
De Concini, Corrado; Strickland, Elisabetta. On the variety of complexes. Adv. in Math. 41 (1981), no. 1, 57--77. 

\bibitem{DRS}
Doubilet, Peter; Rota, Gian-Carlo; Stein, Joel. On the foundations of combinatorial theory. IX. Combinatorial methods in invariant theory. Studies in Appl. Math. 53 (1974), 185--216.

\bibitem{Can.Saha} Can, Mahir Bilen; Saha, Pinaki. Applications of homogeneous fiber bundles to the Schubert varieties. Geom. Dedicata 217 (2023), no. 6, Paper No. 103, 24 pp.

\bibitem{Derksen} Derksen, Harm; Weyman, Jerzy. On the Littlewood--Richardson polynomials. J. Algebra 255 (2002), no. 2, 247--257.

\bibitem{Eisenbud}
Eisenbud, David. Commutative algebra. With a view toward algebraic geometry. Graduate Texts in Mathematics, 150. Springer-Verlag, New York, 1995. {\rm xvi}+785 pp. 

\bibitem{FNR} Feh\'er, L\'aszl\'o M.; N\'emethi, Andr{\'a}s; Rim\'anyi, Rich\'ard. Equivariant classes of matrix matroid varieties. Comment. Math. Helv. 87 (2012), no. 4, 861--889.

\bibitem{FZ}
Fomin, Sergey; Zelevinsky, Andrei. Double Bruhat cells and total positivity. J. Amer. Math. Soc. 12 (1999), no. 2, 335--380. 

\bibitem{FZ:cluster}
Fomin, Sergey; Zelevinsky, Andrei. Cluster algebras. I. Foundations. J. Amer. Math. Soc. 15 (2002), no. 2, 497--529.

\bibitem{FukudaGFans} Fukuda, Komei; Jensen, Anders N.; Thomas, Rekha R. Computing Gr\"obner fans. Math. Comp. 76 (2007), no. 260, 2189--2212.

\bibitem{Fulton:duke} 
Fulton, William. Flags, Schubert polynomials, degeneracy loci, and determinantal formulas. Duke Math. J. 65 (1992), no. 3, 381--420.

\bibitem{Fulton} 
Fulton, William. Young tableaux. With applications to representation theory and geometry. London Mathematical Society Student Texts, 35. Cambridge University Press, Cambridge, 1997. {\rm x}+260 pp. 

\bibitem{FH} Fulton, William; Harris, Joe. Representation theory. A first course. Graduate Texts in Mathematics, 129. Readings in Mathematics. Springer-Verlag, New York, 1991. {\rm xvi}+551 pp.

\bibitem{GHY1}
Gao, Yibo; Hodges, Reuven; Yong, Alexander. Classification of Levi-spherical Schubert varieties. Selecta Math. (N.S.) 29 (2023), no. 4, Paper No. 55, 40 pp.

\bibitem{GHY2}
Gao, Yibo; Hodges, Reuven; Yong, Alexander. Levi-spherical Schubert varieties. Adv. Math. 439 (2024), Paper No. 109486, 14 pp.

\bibitem{Goldin.Precup}
Goldin, Rebecca; Precup, Martha. Matrix Hessenberg schemes over the minimal sheet. Preprint, 2025. \textsf{arXiv:2501.02639}

\bibitem{Gutierrez.Alvaro}
Guti{\'e}rrez, {\'A}lvaro. Towards plethystic $\mathfrak{sl}_2$ crystals. Preprint, 2024. \textsf{arXiv:2412.15006}

\bibitem{Hodge} 
Hodge, W. V. D. The base for algebraic varieties of given dimension on a Grassmannian variety. J. London Math. Soc. 16 (1941), 245--255.

\bibitem{Hodges.Yong}
Hodges, Reuven; Yong, Alexander. Coxeter combinatorics and spherical Schubert geometry. J. Lie Theory 32 (2022), no. 2, 447--474.

\bibitem{Hopkins}
Hopkins, Sam; RSK via local transformations. Notes date July, 2014 available at
\url{https://www.samuelfhopkins.com/docs/rsk.pdf}.

\bibitem{Howe}
Howe, Roger. Perspectives on invariant theory: Schur duality, multiplicity-free actions and beyond. The Schur lectures (1992) (Tel Aviv), 1--182, Israel Math. Conf. Proc., 8, Bar-Ilan Univ., Ramat Gan, 1995.

\bibitem{StHuh}
Huh, June; Matherne, Jacob P.; M\'esz\'aros, Karola; St. Dizier, Avery. Logarithmic concavity of Schur and related polynomials. Trans. Amer. Math. Soc. 375 (2022), no. 6, 4411--4427. 

\bibitem{gfan} Jensen, Anders N. \emph{{G}fan, a software system for {G}r{\"o}bner fans and tropical varieties.} Available at \url{http://home.imf.au.dk/jensen/software/gfan/gfan.html}

\bibitem{Huang} Huang, Hang. Minimal set of generators of ideals defining nilpotent orbit closures. Proc. Amer. Math. Soc. 152 (2024), no. 2, 447--454. 

\bibitem{Kashiwara}
Kashiwara, Masaki. Crystalizing the $q$-analogue of universal enveloping algebras. Comm. Math. Phys. 133 (1990), no. 2, 249--260. 

\bibitem{Kashiwara2}
Kashiwara, M. On crystal bases of the $Q$-analogue of universal enveloping algebras. Duke Math. J. 63 (1991), no. 2, 465--516. 

\bibitem{Kashiwara3}
Kashiwara, Masaki. Crystal bases of modified quantized enveloping algebra. Duke Math. J. 73 (1994), no. 2, 383--413.

\bibitem{KashNak}
Kashiwara, Masaki; Nakashima, Toshiki. Crystal graphs for representations of the $q$-analogue of classical Lie algebras. J. Algebra 165 (1994), no. 2, 295--345.

\bibitem{Klein.Weigandt} Klein, Patricia; Weigandt, Anna. Bumpless pipe dreams encode Gr\"obner geometry of Schubert polynomials. Preprint, 2021. \textsf{arXiv:2108.08370}

\bibitem{Knutson}
Knutson, Allen. \emph{Frobenius splitting, point-counting and degeneration}, preprint, 2009. \textsf{arXiv:0911.4941v1}.

\bibitem{Knutson.Miller}
Knutson, Allen; Miller, Ezra. Gr\"obner geometry of Schubert polynomials. Ann. of Math. (2) 161 (2005), no. 3, 1245--1318. 

\bibitem{KM:adv}
Knutson, Allen; Miller, Ezra. Subword complexes in Coxeter groups. Adv. Math. 184 (2004), no. 1, 161--176.

\bibitem{KMY}
Knutson, Allen; Miller, Ezra; Yong, Alexander. Gr\"obner geometry of vertex decompositions and of flagged tableaux. J. Reine Angew. Math. 630 (2009), 1--31. 

\bibitem{Knutson.Tao}
Knutson, Allen; Tao, Terence. The honeycomb model of ${\rm GL}_n({\bf C})$ tensor products. I. Proof of the saturation conjecture. J. Amer. Math. Soc. 12 (1999), no. 4, 1055--1090.

\bibitem{Kumar}
Kumar, Shrawan. Tensor product decomposition. Proceedings of the International Congress of Mathematicians. Volume III, 1226--1261, Hindustan Book Agency, New Delhi, 2010.

\bibitem{SMT79}
Lakshmibai, V.; Musili, C.; Seshadri, C. S. Geometry of $G/P$. Bull. Amer. Math. Soc. (N.S.) 1 (1979), no. 2, 432--435.

\bibitem{Littelmann94}
Littelmann, Peter. A Littlewood--Richardson rule for symmetrizable Kac-Moody algebras. Invent. Math. 116 (1994), no. 1-3, 329--346.

\bibitem{Lusztig} Lusztig, G. Canonical bases arising from quantized enveloping algebras. J. Amer. Math. Soc. 3 (1990), no. 2, 447--498. 

\bibitem{Lusztig2} Lusztig, G. Canonical bases arising from quantized enveloping algebras. II. Common trends in mathematics and quantum field theories (Kyoto, 1990). Progr. Theoret. Phys. Suppl. No. 102, (1990), 175--201 (1991). 

\bibitem{Manivel}
 Manivel, Laurent. \emph{Symmetric functions, Schubert polynomials and degeneracy loci}. Translated from the 1998 French original by John R. Swallow. SMF/AMS Texts and Monographs, American Mathematical
Society, Providence, 2001.

\bibitem{Marberg}
Marberg, Eric; Pawlowski, Brendan. Gr\"obner geometry for skew-symmetric matrix Schubert varieties. Adv. Math. 405 (2022), Paper No. 108488, 56 pp.

\bibitem{Miller.Sturmfels}
Miller, Ezra; Sturmfels, Bernd. \emph{Combinatorial commutative algebra}. Graduate Texts in Mathematics, 227. Springer-Verlag, New York, 2005. xiv+417 pp.

\bibitem{MoraGFan} Mora, Teo; Robbiano, Lorenzo. The Gr\"obner fan of an ideal. Computational aspects of commutative algebra. J. Symbolic Comput. 6 (1988), no. 2-3, 183--208.

\bibitem{AAA}
Price, Abigail; Stelzer, Ada; Yong, Alexander. Representations from matrix varieties, and filtered RSK. Preprint, 2024.
\textsf{arXiv:2403.09938}

\bibitem{Okounkov}
Okounkov, Andrei. Why would multiplicities be log-concave? The orbit method in geometry and physics (Marseille, 2000), 329--347, Progr. Math., 213, Birkh\"auser Boston, Boston, MA, 2003. 

\bibitem{RWY:coho}
Reiner, Victor; Woo, Alexander; Yong, Alexander. Presenting the cohomology of a Schubert variety. Trans. Amer. Math. Soc. 363 (2011), no. 1, 521--543.

\bibitem{Seccia}
Seccia, Lisa. Knutson ideals of generic matrices. Preprint, 2021.
\textsf{arXiv:2101.06496}

\bibitem{Shimozono}
Shimozono, Mark. Crystals for Dummies. https://www.aimath.org/WWN/kostka/crysdumb.pdf, 2005.

\bibitem{Shimozono.Weyman1}
Shimozono, Mark; Weyman, Jerzy. Bases for coordinate rings of conjugacy classes of nilpotent matrices. J. Algebra 220 (1999), no. 1, 1--55. 


\bibitem{Shimozono.Weyman2}
Shimozono, Mark; Weyman, Jerzy. Graded characters of modules supported in the closure of a nilpotent conjugacy class. European J. Combin. 21 (2000), no. 2, 257--288. 

\bibitem{StanleySymChains}
Stanley, Richard P. Weyl groups, the hard Lefschetz theorem, and the Sperner property. SIAM J. Algebraic Discrete Methods 1 (1980), no. 2, 168--184.

\bibitem{ECII}
Stanley, Richard P. Enumerative combinatorics. Vol. 2. With a foreword by Gian-Carlo Rota and appendix 1 by Sergey Fomin. Cambridge Studies in Advanced Mathematics, 62. Cambridge University Press, Cambridge, 1999. xii+581 pp.

\bibitem{Hilbertian}
Stelzer, Ada; Yong, Alexander. Schubert determinantal ideals are Hilbertian. J. Algebra 677 (2025), 278--293.

\bibitem{Stelzer.Yong}
Stelzer, Ada; Yong, Alexander. RSK as a linear operator. Preprint, 2024. \textsf{arXiv:2410.23009}

\bibitem{Stembridge:rational}
Stembridge, John R. Rational tableaux and the tensor algebra of ${\rm gl}_n$. J. Combin. Theory Ser. A 46 (1987), no. 1, 79--120.

\bibitem{Sturmfels}
Sturmfels, Bernd. Gr\"obner bases and Stanley decompositions of determinantal rings. Math. Z. 205 (1990), 137--144. https://doi.org/10.1007/BF02571229

\bibitem{Trung}
Ng\^{o} Vi\^{e}t Trung. On the symbolic powers of determinantal ideals. J. Algebra 58 (1979), no. 2, 361--369.

\bibitem{bicrystal1}
van Leeuwen, Marc. Double crystals of binary and integral matrices. Elec. J. of Combinatorics 13 (2006), no. 1, R86.

\bibitem{Weigandt} Weigandt, Anna. Prism tableaux for alternating sign matrix varieties. Preprint, 2017. \textsf{arXiv:1708.07236}

\bibitem{Weyman:nilpotent} Weyman, J. The equations of conjugacy classes of nilpotent matrices. Invent. Math. 98 (1989), no. 2, 229--245.

\bibitem{WY:governing}
Woo, Alexander; Yong, Alexander. Governing singularities of Schubert varieties. J. Algebra 320 (2008), no. 2, 495--520.

\bibitem{WY}
Woo, Alexander; Yong, Alexander. A Gr\"obner basis for Kazhdan--Lusztig ideals. Amer. J. Math. 134 (2012), no. 4, 1089--1137.

\bibitem{WY:survey}
Woo, Alexander; Yong, Alexander. Schubert geometry and combinatorics. Preprint, 2023. \textsf{arXiv:2303.01436}

\bibitem{Yuen}
Yuen, Cornelia. Jet schemes of determinantal varieties. Algebra, geometry and their interactions, 261--270, Contemp. Math., 448, Amer. Math. Soc., Providence, RI, 2007.

\end{thebibliography}
\end{document}